\documentclass[review,onefignum,onetabnum]{siamonline220329}
\usepackage{mathtools}
\usepackage{esint}

\usepackage[normalem]{ulem}
\usepackage{lettrine}
\usepackage[utf8]{inputenc}
\usepackage{amsmath}
\usepackage{amsfonts}
\usepackage[english]{babel}
\usepackage[T1]{fontenc}
\usepackage{mathrsfs}
\usepackage{graphicx}
\graphicspath{{./Images/}}
\usepackage{setspace}
\usepackage{bigints}
\usepackage[left=2cm,right=2cm,top=3cm,bottom=3cm]{geometry}
\usepackage[normalem]{ulem}
\usepackage{lettrine}
\usepackage{amssymb}
\usepackage{bm}
\usepackage{xcolor}
\usepackage[english]{babel}
\usepackage[T1]{fontenc}
\usepackage{mathrsfs}
\usepackage{graphicx}
\graphicspath{{./Images/}}
\usepackage{setspace}
\usepackage{bigints}
\usepackage{setspace}
\newsiamthm{thm}{Theorem}
\newsiamthm{prop}{Proposition}
\newsiamthm{lem}{Lemma}
\newsiamthm{cor}{Corollary}
\newsiamremark{remark}{Remark}
\newsiamremark{example}{Example}
\newcommand{\rn}{\mathbb{R}^n}

\makeatletter
\newcommand{\bigperp}{%
	\mathop{\mathpalette\bigp@rp\relax}%
	\displaylimits
}

\newcommand{\bigp@rp}[2]{%
	\vcenter{
		\m@th\hbox{\scalebox{\ifx#1\displaystyle2.1\else1.5\fi}{$#1\perp$}}
	}%
}
\makeatother

\newcommand{\ol}{\overline}

\DeclareMathOperator*{\Diff}{Diff}
\DeclareMathOperator*{\Lin}{Lin}

\newcommand{\ad}{\text{ad}}

\def\fra{\color{blue}}

\def\rr{\mathbb{R}}
\definecolor{mydarkred}{RGB}{180, 10, 10}

\usepackage{cancel}
\def\bel{\begin{equation}\label}
	\def\eeq{\end{equation}}
\def\ds{\displaystyle}

\def\setmap{\rightsquigarrow}
\def\U{\mathbf{U}}

\numberwithin{equation}{section}

\usepackage{fancyhdr}
\author{Francesca Angrisani\footnote{Laboratoire Jacques Louis Lions (LJLL), Sorbonne, Paris. (\email{francesca.angrisani@sorbonne-universite.fr}).} \,\,and\,\, Franco Rampazzo\footnote{Department of Mathematics T. Levi-Civita, University of Padova, Italy. (\email{rampazzo@math.unipd.it}).}}

\title{Goh and Legendre\textendash{}Clebsch conditions for nonsmooth control systems}

\usepackage{fancyhdr}
\begin{document} 
		\nolinenumbers 
	
	\maketitle

	\normalsize
	
	\noindent

	\begin{abstract}
		Higher order necessary conditions for a minimizer of an optimal control problem--- in particular,  Goh and Legendre\textendash{}Clebsch type conditions--- are
		generally obtained for systems whose dynamics is at least continuously differentiable in the state variable.
		Here, by making use of the notion of set-valued Lie bracket introduced in "Set-valued differentials and a nonsmooth version of Chow-Rashevski's theorem" by F. Rampazzo and H.J.Sussmann and extended in "Iterated Lie brackets for nonsmooth vector fields" by E. Feleqi and F.Rampazzo , we obtain  Goh and Legendre\textendash{}Clebsch type conditions for a  control affine system with locally Lipschitz continuous dynamics. In order to manage the simultaneous lack of smoothness of the adjoint equation and of the Lie bracket-like variations, we will exploit the notion of Quasi Differential Quotient, introduced in "A geometrically based criterion to avoid infimum-gaps in Optimal Control" by M. Palladino and F.Rampazzo. \\ We finally exhibit an example where the established higher order condition is capable to rule out the optimality of a control verifying a first order Maximum Principle.
	\end{abstract}
	\noindent
	.
	
	\section{Introduction}
	

	\subsection{The problem and the main result} 
	Our investigation will concern   optimal control problems on an interval $[0,T]$ of the form
	$$(P)\qquad\left\{
	\left.
	\begin{array}{l}
		\quad \ds	
		\min\Psi(x(T)),\\ \\
		
		\begin{cases}
			
			\displaystyle \frac{dx}{dt}=f(x(t))+\sum\limits_{i=1}^m g_i(x(t))u^i(t), \quad  a.e.  t \in [0,T],\\
			\displaystyle x(0) =\hat{x},\qquad x(T) \in \mathfrak{T}.
	\end{cases} \end{array}\right.\right.$$ 
	where the  minimization is performed over the set of the {\it feasible processes} $(u,x)$, and, for some integers $n,m$, the state $x$ and the control  $u$ take values in $\mathbb R^n$  and in  a  subset  $U\subseteq\mathbb{R}^m$, respectively.    By  {\it process}  we mean  a  pair  $(u,x)$
	such that $\displaystyle u\in L^\infty([0,T],U)$
	and 
	$x\in W^{1,1}([0,T],\mathbb{R}^n)$ is the corresponding  solution of the above Cauchy problem.  The subset $\mathfrak{T}\subset \mathbb{R}^{n}$ is called the  {\it target}, and  a process  $(u,x)$ is  called   {\it feasible}  as soon as $x(T)\in \mathfrak{T}$.  The vector fields $f,g_1,\ldots,g_m$ as well as the cost $\Psi$ are assumed to be of class $C^{0,1}_{\mathrm loc}$, namely locally  Lipschitz continuous. 
	
	A feasible process  $(\overline{u},\overline{x})$ is a {\it  local $L^1$ minimizer} for problem ($P$) if there exists $\delta>0$ which verifies $$\Psi(\overline{x}(\overline{T})) \le \Psi(x(T))$$
	for all feasible processes $(u,x)$ such that $\|x-\ol{x}\|_{C^0} + \|u-\bar u\|_1
	<\delta$.
	Let us consider the the {\it Hamiltonian}
	$H(x,p,u):=p\cdot\Big[f(x)+\sum\limits_{i=1}^m g_i(x)u^i\Big]$. As is known,  if $f$, $\Psi$ and every $g_i$ are of class  $C^1$, the standard, first order,  {\it Maximum Principle}   for a  local $L^1$ minimizer  $(\overline{T},\bar u,\bar x)$ includes the existence of an adjoint path $ (p(\cdot))\neq 0$, which, besides verifying the {\it adjoint equation} $\dot p (t) = \displaystyle-p(t) \cdot\frac{\partial H}{\partial x} \left(x(t),p(t).u^i(t)\right),$ maximizes the { Hamiltonian}
	at $u=u(t)\in U$, for almost every $t$, namely  \bel{hamiltoniano}H(x(t),p(t),u(t))=\max_{u\in U} H(x(t),p(t),u). \eeq Furthermore,   a crucial  {\it transversality} condition with  respect to the target $\mathfrak{T}$ and the cost $\Psi$ has to  be verified. 
	
	Also the  classical  {\it second order } Legendre\textendash{}Clebsch condition from Calculus of Variations has been generalized 
	to Optimal Control (\cite{Schattler}, Subsect. 2.8). In the case of problem ($P$) with $m=1$  and $g=g_1$
	it translates into the  step $2$ and step $3$ relations (valid for every $ t\in [0,T]$)
	$$
	\qquad \left.\begin{array}{l}\text{(LC2)}\qquad	p(t)\cdot [f,g](\bar x(t)) = 0\quad \footnotemark
		
		\\\\\\ \text{(LC3)}\qquad p(t)\cdot [g,[f,g]](\bar x(t)) \leq 0     .\end{array}\right.$$
	\footnotetext{As customary, if $X$ and $Y$ are  vector fields,  $[X,Y]$ denotes their {\it Lie bracket},  i.e. the  vector field that (in each coordinate chart) is  defined as $[X,Y]:=DY\cdot X-DX\cdot Y$. Furthermore, we say that an iterated  bracket is {\it of step  $q$ } ($\geq 2$) as soon as it is  formally made by means of  $q$ vector fields. For instance, $[Z,[X,Y]]$, $[[X,Y],[Y,Z]]$, and  $[[Z,[X,Y]],[X,W]]$ are of step $3,4$, and $5$, respectively.} 
	Actually, for  the step $3$ relation in (LC)  to have a classical meaning,  $f$ is required to be of class $C^2$.
	Let us recall that an essential hypothesis in  the  Legendre\textendash{}Clebsch condition (LC2)-(LC3)  consists in the fact  that  the control $\ol{u}$ has to be  {\it singular} (but see,  \cite{FrankTonon2} for generalizations), meaning that  $\ol{u}$'s values lie in  the interior $int \,U$ of the control set $U$.   Furthermore, still under  this singularity hypothesis,  in the general case $m\geq 1$, the {\it Goh conditions}
	$$
	(\mathrm{Goh})\qquad \qquad	p(t)\cdot [g_i,g_j](\bar x(t)) = 0 \quad \forall i,j=1,\ldots,m, \qquad\forall t\in [0,T]
	$$
	are necessary as well for $(\bar u,\bar x)$ to be a minimizer.
	
	Aiming to a generalization to the case of non-smooth dynamics, let us begin with recalling that various versions of a   first order {\it non-smooth} Maximum Principles have been established. Actually ,     as soon as the vector fields $f,g_1,\ldots,g_m$ and the cost $\Psi$ are  locally Lipschitz continuous,  by means of some notion of generalized differentiation  (an instance   being the Clarke's generalized Jacobian) the adjoint equation  condition 
	has been variously extended to an {\it adjoint  differential inclusion}, and  non-smooth generalization has been obtained for  the transversality condition as well. So,
	we find somehow natural 
	to wonder the following issues: 
	
	{\it What happens to the Goh conditions} \text{(Goh)} {\it if $g_1,\ldots,g_m$ are just locally Lipschitz continuous? And what about the  Legendre\textendash{}Clebsch conditions} (LC2)-(LC3) {\it    when ($m=1$) and the vector fields $f,g_1$ are just differentiable with locally Lipschitz continuous derivatives?  }

	A  partial answer  to the above inquiries has been proposed ---relatively to the sole Goh conditions---  in \cite{papermandatoadesaim}. In order to state that result, we need to recall from \cite{RS1} the definition of {\it set-valued Lie bracket} $[X,Y]_{\mathrm{set}}$ of two locally Lipschitz continuous vector fields:
	$$
	\ds [X,Y]_{\mathrm{set}}(x) := \mathrm{co} \Bigg\{ v = \lim_{x_k\to x}[X,Y](x_k) , \,\, \,\,\,(x_k)_{ k\in\mathbb{N}}\subset \mathrm{Diff}(X)\cap \mathrm{Diff}(Y) \Bigg\}\quad \forall x\in \mathbb{R}^n$$
	(where we use  $\mathrm{Diff}(\Phi)$ to  denote the set of differentiability points of a function $\Phi$). \\ In \cite{papermandatoadesaim} the  set-valued Lie bracket $[g_i,g_j]_{\mathrm{set}}$ replaces  the classical Lie bracket, giving rise to the non-smooth Goh\textendash{}like conditions
	$$\text{(Goh)}_{nsmooth} \quad     0\in p(t)\cdot [g_i,g_j]_{\mathrm{set}}(\ol{x}(t))\qquad i,j=1,\ldots,m
	$$
	where $p(\cdot)$ is a solution of the corresponding adjoint inclusion. Another crucial tool in \cite{papermandatoadesaim} is the notion of {\it Quasi Differential Quotient}, (QDQ),  a set-valued generalization of the notion of derivative,  which  includes two different kind of nonsmoothness within the same frame, namely  the one related to
	set-valued brackets and the one involved in adjoint Clarke-type differential inclusion. 
	
	Let us point out  that  in \cite{papermandatoadesaim} the control set $U$ is a symmetric  cone.  This eases the construction of variations  giving rise to the (generalized) Goh conditions, in that the unboundedness of the controls  in a sense allows to disregard the action of the drift.   In the present paper we  continue  the program initiated in \cite{papermandatoadesaim}  by, on the  one hand, considering the  standard   case when the controls $u$ take values in (the interior of)  any control set $U$ and, on the other hand, establishing   both  Goh-type conditions \text{(Goh)}$_{n-sth}$ and   non-smooth Legendre\textendash{}Clebsch conditions, under the only hypothesis the the values of $\ol{u}$ belong to the interior of $U$. The latter conditions are as follows,
	$$ \text{(LC2)}_{nsmooth}\qquad
	0\in p(t)\cdot [f,g_j]_{\mathrm{set}}(\ol{x}(t))\qquad \forall j=1,\ldots,m, \qquad\qquad $$
	$$ \text{(LC3)}_{nsmooth}\qquad	{0\ge \min \Big\{ p(t)\,\cdot [g,[f,g]]_{\mathrm{set}}(\ol{x}(t)) \Big\} \qquad  \quad \text{ a.e.} \ t \in [0,T] },
	$$
	\text{(LC3)}$_{nsmooth}$ being referred to the case $m=1$, in which one sets $g:=g_1$. 
	
	Let us make precise immediately a  few important  facts about  a  step-$3$ bracket $[g,[f,g]]_{\mathrm{set}}(x)$ utilized in \text{(LC3)}$_{nsmooth}$.
	The general definition is as follows:
	$$\begin{array}{c}
		\ds [Z,[X,Y]]_{\mathrm{set}}(x) := \mathrm{co} \left\{\begin{array}{l} v = \ds \lim_{(x_k,y_k)\to (x,x)} D[X,Y](x_k)\cdot Z(y_k) - 
			DZ(y_k)\cdot D[X,Y](x_k),\\\qquad\qquad\qquad\ds(x_k,y_k)_{ k\in\mathbb{N}}\subset \mathrm{Diff}([X,Y])\times \mathrm{Diff}(Z)\end{array} \right\}
	\end{array}
	$$ 
	As one might expect, $[Z,[X,Y]]_{\mathrm{set}}$ is meaningful when $Z\in C^{0,1}_{\mathrm loc}$,    and  $X$ and $Y$ are of class $C_{\mathrm loc}^{1,1}$. Hence we will assume $f$ and $g$ to be  $C_{\mathrm loc}^{1,1}$, while we will suppose $g_i$ ($i \neq 1$) and $\Psi$  just of class $C_{\mathrm loc}^{0,1}$. However, it is important to notice that  
	this set-valued bracket is {{ larger than the one achievable by a recursive approach}} (which would be equivalent to take sequences $(x_k,y_k)$ of the special form $(x_k,x_k)$). The need for such a choice is intimately connected to the possibility of establishing asymptotic formulas for set-valued brackets, and  is discussed e.g. in \cite{RS2}. 
	
	The main point in the proof of \text{(Goh)}$_{nsmooth}$, \text{(LC2)}$_{nsmooth}$, and \text{(LC2)}$_{nsmooth}$ (which correspond to \eqref{goh22}, \eqref{goh12}, and \eqref{LC32}, in Theorem  \ref{TeoremaPrincipale}  below, respectively) consists in showing that $[g_i,g_j]_{\mathrm{set}}$, $[f,g_i]_{\mathrm{set}}$ $[g,[f,g]]_{\mathrm{set}}$ {\it are indeed $QDQ$s of suitable control variations}. This will require an attentive exploitation of arguments valid for the smooth case together with a careful use of convolution arguments which will allow us to show that the set-valued brackets are indeed Quasi Differential Quotients. Then, once observed  that the solution of the adjoint differential inclusion is a QDQ as well, one concludes the specific proofs by applying some set-separation  and open-mapping results (valid in the QDQ theory) as well as  some    non-empty intersection arguments.
	
	Let us conclude by  pointing out that, for the sake of simplicity, in the whole paper the state $x$ will range over $\mathbb{R}^n$. However, because of the local character of  the considered arguments (and to the chart-invariant character of both QDQ theory and set-valued brackets),  the adaptation to a differential manifold turns out to be quite natural. Other possible generalizations are mentioned in the next Subsection and performed in the last Section.
	\subsection{Organization of the paper}
	
	This paper is organized as follows. In Section 2 we  state the  minimum problem together with the main result (Theorem \ref{TeoremaPrincipale}). We end the section with an  example of a  process  which is ruled out from being optimal in that, though it  satisfies the first order  Maximum Principle and both Goh conditions and the step 2 Legendre\textendash{}Clebsch condition, it does't satisfy the step 3  Legendre\textendash{}Clebsch necessary condition. Section 3 is mainly devoted to the presentation of   Quasi Differential Quotients.    Section 4 is devoted  to the proof of the main result. On the one hand, we introduce  set-valued bracket and we show that they are QDQ by means of multiple mollifications. This, together with the utilization of Clarke's gradient for the  adjoint inclusion,  allows us to create an environment where one can apply set-separation arguments connected with QDQ's. On the other hand, we have exploited { asymptotic formulas} for the solutions of  (smooth) control affine systems. These are    kinds of computations  which, more or less explicitly (see e.g. \cite{Schattler}), can be found in much classical literature.\footnote{ The reader that would want to see these calculations performed explicitly can refer to the Appendix.  } The afore-mentioned set-separation argument and a successive non-empty intersection argument allow us to conclude the proof of the main theorem.  In the last Section we suggest two kinds of  generalization of the main result. To begin with, we consider the case where the end-time $T$ is not fixed (and both the cost and the target depend on it). Secondly, we suggest how the main theorem might also be extended to a situation where the values of the optimal control $\ol{u}$ are in the interior of $U$ only on a finite family of subintervals of the whole interval $[0,T]$  (so that the higher order necessary condititions would hold only  on these intervals, still almost everyehere).  

	\section{The main result}\label{sec1}
	Let us consider the optimal control problem
	$$(P)\qquad\left\{
	\left.
	\begin{array}{l}
		\quad \ds	
		\min \Psi(x(T)),\\ \\
		\begin{cases}
			
			\displaystyle \frac{dx}{dt}=f(x(t))+\sum\limits_{i=1}^m g_i(x(t))u^i(t), \quad  a.e.  t \in [0,T],\\
			\displaystyle x(0) =\hat{x},\qquad x(T) \in \mathfrak{T}.
	\end{cases} \end{array}\right.\right.$$ \\
	{ The subset  $\mathfrak{T}\subseteq \rr^n$ is called   the {\it target set},  and both the cost $\Psi$ and  the vector fields $f,g_1,\ldots,g_m$ will be assumed of class $C^{0,1}_{loc}$, namely  locally Lipschitz continuous. The {\it control maps} $ u$ belong to $\in\mathcal{U}:=L^\infty([0,T],U)$, the subset $U\subseteq\rr^m$ ($m\geq 1$) being called the {\it set of control values}. For a given control map  $u\in\mathcal{U}$, the Caratheodory solution $x\in W^{1,1}([0,T],\mathbb{R}^n)$  of the Cauchy problem in $(P)$ is called the {\it trajectory corresponding to $u$}, and $(u,x)$ is said an {\it admissible process}
		provided $x(T)\in \mathfrak{T}$.   }
	\begin{definition}\label{weakmin}
		We will say that an admissible process $(\overline{u},\overline{x})$ is a {\rm local $L^1$ minimizer for problem $(P)$} if there exists $\delta>0$ such that $$\Psi(\overline{x}({T})) \le \Psi(x(T))$$
		for all admissible processes $(u,x)$ such that $\|x-\ol{x}\|_{C^0} + \|u-\bar u\|_1
		<\delta,$ where $\|\cdot\|_{C^0}$ {\rm \big[}resp. $\|\cdot\|_1${\rm\big]} denotes the norm of $C^0([0,T],\rr^n)${\rm\big[}resp. of $L^1([0,T],\rr^m)${\rm\big]}.  Furthermore, we will say that a control $u$ 
		is {\rm singular } provided $u(t) \in \mathrm{int}(U)$, the  the interior of $U$, for  almost every $t\in [0,T]$.
		
	\end{definition}
	Let us  define the  (unmaximized) {\it Hamiltonian} $H$ 
	by setting  $$H(x,p,u):= p\cdot \Big(f(x)+\sum\limits_{i=1}^m g_i(x)u^i\Big)\quad\,\,\,\forall (x,p,u)\in \rr^n\times(\rr^n)^*\times U.
	$$ 
	{  In the statement of  Theorem  \ref{TeoremaPrincipale} below we will make use of the notion of {\it QDQ-approximating multicone}  $ \mathfrak{C}$, whose definition is postponed to Section  \ref{brief introduction}. However, for a first,  intuitive, understanding of the statement it is enough  to  replace the expression  "QDQ-approximating multicone"   with some  other more familiar concept of smooth or non-smooth  tangent cone.\footnote{For instance,  the Boltianski tangent cone, or the cone of convex analysis, or even the image of a convex cone via a Clarke generalized differential.} This is obviously possible thanks to the  of the vast generality of the concept  of { QDQ-approximating multicone}.
		\begin{theorem}[{A non-smooth Maximum Principle with Goh and Legendre\textendash{}Clebsch conditions}] \label{TeoremaPrincipale}
			Let us assume that $f, g_1, \ldots, g_m \in C^{0,1}_{\mathrm loc} (\mathbb{R}^n,\mathbb{R})$ and $\Psi \in C^{0,1}_{\mathrm loc}(\mathbb{R}^n,\mathbb{R})$.	Let $(T,\ol{u},\ol{x})$ be a singular,   local $L^1$ minimizer for  problem $(P)$, and
			let  $\mathfrak{C}$ be a QDQ-approximating multicone to the target set $\mathfrak{T}$ at $x(T)$. 
			Then there exist multipliers $(p,\lambda) \in AC\Big([0,\ol{S}],(\mathbb{R}^n)^*\Big) \times \mathbb{R^*} $, with $\lambda\geq 0$, such that the following facts are verified.
			
			\begin{itemize}
				\item[\rm\bf i)]{\sc(Non-triviality)}\,\,\, $(p,\lambda)\neq 0.$ \,\,\,\,
				
				\item[\rm\bf ii)]{\sc(Adjoint differential inclusion)} $$\frac{dp}{dt}\in -  \partial^C_x H(\ol{x}(t),p(t),\ol{u}(t))\qquad a.e.\quad t\in [0,T] \,\,\,  $$
				\item[\rm\bf iii)]{\sc(Transversality)} $$p({T})\in -\lambda  \partial^C\Psi\Big(\ol{x}({T})\Big)-\ol{\bigcup_{\mathcal{T}\in  \mathfrak{C}}\mathcal{T}^\perp}  .\,\,$$

				\item[\rm\bf iv)]{\sc(Hamiltonian's maximization)}\,\,\,
				$$\max_{u \in U} H(\ol{x}(t),p(t),u)=H(\ol{x}(t),p(t),\ol{u}(t)) \qquad a.e.\quad t\in [0,T]\,\,\,  \,\,
				$$
			\end{itemize}\,
			{Moreover, the following higher order conditions are verified:}
			\begin{itemize}
				\item[\rm\bf v)]{\sc(Nonsmooth Goh condition)}\\ 
				\bel{goh22}
				{0\in p(t)\,\cdot [g_j,g_i]_{\mathrm{set}}(\ol{x}(t))  \qquad i,j \in \{1,\ldots,m\}  }
				\qquad a.e.\quad t\in I	\eeq 
				\item[\rm\bf vi)]{\sc(Nonsmooth Legendre\textendash{}Clebsch  condition of step $2$)}\\ 
				\bel{goh12}
				{0\in p(t)\,\cdot [f,g_i]_{\mathrm{set}}(\ol{x}(t))  \qquad i \in \{1,\ldots,m\} } \qquad a.e.\quad t\in I
				\eeq

				\item[\rm\bf vii)]{\sc(Nonsmooth Legendre\textendash{}Clebsch condition of step $3$ with $m=1$)}\\ 
				If $m=1$, for almost all $t \in [0,T]$ such that $f$ and  $g:=g_1$ are of class $C^{1,1}$ around $\ol{x}(t)$, then the adjoint  differential inclusion of point {\bf ii)} reduces to the usual adjoint differential  equation and
				\bel{LC32}
				{0\ge \min \Big\{ p(t)\,\cdot [g,[f,g]]_{\mathrm{set}}(\ol{x}(t)) \Big\}. }
				\eeq
				
			\end{itemize}
		\end{theorem}

		\subsection{A simple example} In the following example we are showing how a process  $(\bar u(\cdot),\bar x(\cdot))$ which verifies the first  six conditions of the  theorem  stated in the previous section is ruled out from being a minimizer by the failure of the step-$3$ necessary condition {\bf vii)}. Let us consider  the family of $L^\infty$ controls from {{the}} time  interval $[0,4]$  into $U=[-1,1]$, and the  optimal control problem in $\Bbb R^3$
		$$\qquad\left\{
		\left.
		\begin{array}{l}
			\quad \ds	
			\min_{u\in L^1([0,4],U)} \Psi(x(4)),\\ \\
			\begin{cases}
				
				\displaystyle \frac{dx}{dt}=f(x)+ g(x)u(t), \quad \text{a.e.} \quad  t \in [0,4],\\
				\displaystyle x(0) =(0,0,-4),\qquad x(4) \in \mathfrak{T} := \Bbb R^2 \times [0,+\infty),
		\end{cases} \end{array}\right.\right.$$ 
		where \footnote{Agreeing with a general use in  Differential Geometry we utilize high indexes to mean the components of a state variable and low indexes for the components of a adjoint variable. So here $x^{_2}$ is not "$x$ square", but just the second component  of $x$.  } $$
		\Psi(x) := ({x^{_2}}-1)^2 ,  
		$$
		
		{
			$$
			f(x) :=  \begin{cases} \displaystyle\frac{3}{2}({x^{_1}})^2  \frac{\partial}{\partial {x^{_2}}} + \frac{\partial}{\partial {x^{_3}}} \qquad \text{if}   &    x^1 \in (-\infty,0)\\\\ \displaystyle ({x^{_1}})^2 \frac{\partial}{\partial {x^{_2}}} + \frac{\partial}{\partial {x^{_3}}} \qquad  \text{if}   &    x^1 \in [0,\infty)
			\end{cases} \qquad \qquad  \text{and} \qquad \qquad g(x):=  \frac{\partial}{\partial {x^{_1}} } .
			$$} \\ 
		
		Every datum in this problem is $C^2$(at least), except for the vector field $f$, which is  of class $C^{1,1}$ and  not $C^2$. This implies  the classical  step-$3$ Legendre\textendash{}Clebsh  condition is not meaningful and has to be replaced e.g. with its non-smooth counterpart  {\bf vii)}  of Theorem \ref{TeoremaPrincipale}.\\  Let us consider the  trajectory $\bar x=(\bar x^{_1},\bar x^{_2},\bar x^{_3}):[0,4]\to\Bbb R^3, \ \ \bar x(t)=(0,0,-4+t)$   which originates from the implementation of the   singular control $\bar u(t)\equiv 0$.
		By direct computation one easily sees  that a control $u\equiv \tilde u$ for a suitable $\tilde u>0$  is more convenient than $\bar u$, so that  the process  $(\bar u,\bar x)$ is not a local minimizer. Yet,  we shall see that  the non-optimality of  $\bar u$  cannot be recognized   by the  maximum conditition {\bf iv)} or by the  step-$2$ Legendre\textendash{}Clebsh condition {\bf vi)} of Theorem \ref{TeoremaPrincipale}. However, the generalized step-$3$ Legendre\textendash{}Clebsh condition {\bf vii)} turns out to be  not verified, which shows  the lack of optimality of  $(\bar u,\bar x)$. \\
		{To begin with, as a  QDQ-approximating multicone $ \mathfrak{C}$ one simply chooses the singleton made of  the target $\mathfrak{T}$ itself,\footnote{By Definition \ref{qdqmulticones}, as a (QDQ)-approximating cone to a vector half-space at the origin one can trivially choose  the half space itself.}   namely $ \mathfrak{C}:=\{\mathfrak{T}\}$.
			{Since $f \in C^1$,}} the corresponding multiplier $p:[0,4]\to (\Bbb R^3)^*$ verifies {the following adjoint equation}
		$$\dot p(t) = -p(t) \cdot  \left(\frac{ \partial f}{\partial x}(\bar x(t))  + \frac{\partial g}{\partial x}\bar u(t) \right)  =  -p(t)\cdot  \frac{\partial f}{\partial x} (\bar x(t)) \equiv0,$$ 
		so that  $p(\cdot)$  is  constant.\\
		Therefore, in agreement with the tranversality condition { \bf iii)}, $p(4) =  (p_1,p_2,p_3)(4) = \displaystyle\lambda \frac{\partial \Psi}{\partial x} - (0,0,-1) =\lambda(0,-2,0) - (0,0,-1)$ and, by  choosing $\lambda =1$, we get 
		$p(t) = (p_1,p_2,p_3)(t) = (0,-2,1)\,\, \forall t\in [0,4].$ 
		Since ${\bar x^{_1}}(t)\equiv 0$,  one has, for every $u\in [-1,1]$,
		{{$$
				H(\bar x(t),p(t),u):=p(t)\cdot\Big[f(\bar x(t))+ g(\bar x(t))u\Big]= -  \left({2({\bar x^{_1}}(t))^2} \right) +\equiv 1,$$}
			so the Hamiltonian maximization {\bf iv)} is trivially verified. 
			Furthermore, \footnote{For every subset $E\subset F$ we use $\chi_E$ to denote the characteristic function of $E$.}
			$$
			[f,g](x) = \left(-3{x^{_1}} \chi_{(-\infty,0)}(x^{_1}) - 2{x^{_1}}\chi_{[0,+\infty)}(x^{_1})\right)\frac{\partial}{\partial {x^{_2}}}
			$$$$ [g,[f,g]]_{\mathrm{set}}(x) =\left\{ \left(-3\chi_{(-\infty,0)}(x^{_1}) - 2\chi_{{\color{blue}(}0,+\infty)}\right)(x^{_1})\frac{\partial}{\partial {x^{_2}}}\right\} \ \ \forall x \  \text{with} \ {x^{_1}}\neq 0$$$$
			\quad  [g,[f,g]]_{\mathrm{set}}(0,{x^{_2}},{x^{_3}}) = 
			\left\{ \alpha\frac{\partial}{\partial {x^{_2}}} \ \ \alpha \in[-3,-2]\right\}\quad\forall ({x^{_2}},{x^{_3}})\in\Bbb R^2$$
			Therefore (still because of the identity  ${\bar x^{_1}}(t)\equiv 0$),
			$$
			p(t)\cdot [f,g](\bar x(t)) = (0,-2,1)\cdot \left( - 2 {\bar x^{_1}}(t)\right)\frac{\partial}{\partial {x^{_2}}} = - 2\cdot 0 = 0,
			$$
			and condition  {\bf vi)} is satisfied as well.  However,
			$$
			\min \Big\{ p(t)\,\cdot [g,[f,g]]_{\mathrm{set}}(\ol{x}(t)) \Big\} = 
			\min \Big\{ -2\alpha,\,\,\,\,\,\alpha\in[-3,-2]\Big\} = 4 >0,
			$$
			that is, the  step-$3$ condition  {\bf vii)} {\it is not met}, so $(\ol{u},\ol{x})$ is not a minimizer.  }

		{ \section{Quasi Differential Quotients}\label{brief introduction}
			
			The proof of the main result (i.e. Theorem \ref{TeoremaPrincipale} ) and even its statement make a crucial utilization of the notions of Quasi Differential Quotient (QDQ) and QDQ-approximating cones. Let us briefly illustrate these issues. Let us begin with a  standard notational convention: if $n_1,n_2$ are non-zero natural numbers, we will use $\Lin(\rr^{n_1},\rr^{n_2})$ to denote both the vector space of linear maps from $\rr^{n_1}$ to $\rr^{n_2}$
			and the isomorphic vector space of ${n_1}\times {n_2} $ real matrices.
			{ {The ball with radius $\delta>0$ and center $0$ is denoted by $B_{\delta}$, and  $d(\cdot,\Lambda)$  stands for the (Euclidean) distance between a point and a  set $\Lambda$. }}

			The notion of  Quasi Differential Quotient (QDQ),  which provides  a comprehensive  framework   for  the different  kinds of nonsmoothness considered in this paper,   was introduced in \cite{RampPal} as  a subclass of Sussmann’s Approximate Generalized Differential Quotients (AGDQs) (\cite{SussmannAGDQ}). Unlike the latter, QDQs verify a quite useful open mapping property.\footnote{AGDQs verify only a {\it punctured} open mapping property.} 
			
			\begin{definition}[Quasi Differential Quotient]\label{qdq}\cite{RampPal}	Assume that $\mathcal{F} : \rr^n \rightsquigarrow \rr^m$       is a 
				set-valued map, $(\bar x,\bar y) \in \rr^n\times\rr^m  $,   $\Lambda\subset \Lin( \rr^n, \rr^m) $   is a compact set,  and $\Gamma\subseteq\rr^n$  is any  subset.
				We
				say that $\Lambda$ is a {\em Quasi Differential Quotient  (QDQ) of $\mathcal{F}$ at  $(\bar x,\bar y)$ in the direction of $\Gamma$ }  if there
				exists a $\delta^* \in (0,+\infty]$ and a modulus\footnote{We call{ \it modulus }  a non-decreasing non-negative function $\rho:[0,+\infty[\to [0,+\infty[$ such that $\lim\limits_{\delta\to 0^+}\rho(\delta)=0$.} $\rho:[0,+\infty[\to [0,+\infty[$  such that,
				for every $\delta \in (0,\delta^*)$,  there is a continuous map  \linebreak
				$(L_\delta,h_\delta):\left(\bar{x}+B_{\delta}\right)\cap\Gamma\to \Lin( \rr^n, \rr^m) \times \rr^m$ 
				such that,	whenever $x
				\in  (\bar x+B_\delta)\cap\Gamma$, 
				\begin{equation} \label{approssimazione}\begin{array}{c}d(L_\delta(x), \Lambda)\leq \rho(\delta), 
						\quad |h_\delta(x)|\leq \delta \rho(\delta)\\\\
						\bar y +  L_\delta(x)\cdot(x-\bar x)  + h_\delta(x)\in \mathcal{F}(x) .\end{array}\end{equation}
			\end{definition}
			Notice that  the name "QDQ" is reminiscent of  the fact that   the differential quotients at a point $\bar x$ of a differentiable map $f$ approximate the differential $Df(\bar x)$  as the $L_\delta(x)$ in the above definition approximate  $\Lambda$.
			
			{ In what follows, if $W$ is a real  vector space and  $A,B\subset W$, $\Omega\subset \mathbb{R}$,    we will use  the notations \\ $A+B:=\{a+b,\,\,\,(a,b)\in A\times B\}$ and  $\Omega A:=\{\omega a,\,\,\,\omega\in\Omega,  a\in A\}$.  When  $\Omega:=\{w\}$ for some $\omega\in \Bbb R$, we simply write $\omega A$ instead of $\{\omega\}A$ }.
			
			Quasi Differential Quotients enjoy some elementary properties of {\it locality, linearity}, and good behaviour with respect to  { \it set product} (see  \cite{AngrisaniRampazzoQDQ} ). Moreover   a {\it product rule} holds true. \\
			Finally, the gradient of a differentiable single-valued function $\mathcal{F}$ is also a $QDQ$ for $\mathcal{F}$. \\
			We will make  use of  these properties, but for brevity we omit the proofs, which  can be found e.g. in  \cite{AngrisaniRampazzoQDQ}.
			
			Let us state  explicitly t}he less intuitive property of {\it chain rule} for the $QDQ$.}
	\begin{prop}[Chain rule]\label{chain}
		\cite{AngrisaniRampazzoQDQ}  	Let   $\mathcal{F}:\mathbb{R}^N \rightsquigarrow \mathbb{R}^n$ and $\mathcal{G}:\mathbb{R}^n \rightsquigarrow \mathbb{R}^l$ be set-valued maps, and consider the composition $\mathcal{G}\circ \mathcal{F}:\mathbb{R}^N \ni x  \rightsquigarrow \bigcup\limits_{y \in \mathcal{F}(x)} \mathcal{G}(y) \in\mathbb{R}^l$.
		Assume that ${\mathbf S}_\mathcal{F}$ is a QDQ of $\mathcal{F}$ at $(\bar x, \bar y)$ in the direction of $\Gamma^\mathcal{F}$ and ${\mathbf S}_\mathcal{G}$ is a QDQ of $\mathcal{G}$ at  $(\bar y, \bar z)$ in a direction $\Gamma_\mathcal{G}$ verifying $\Gamma_\mathcal{G}\supseteq \mathcal{F}(\Gamma_\mathcal{F})$. Then the set $${\mathbf S}_\mathcal{G} \circ {\mathbf S}_\mathcal{F}:=\Big\{ML,\quad M\in {\mathbf S}_\mathcal{G},  \, L\in {\mathbf S}_\mathcal{F}\Big\}$$
		is a QDQ of $\mathcal{G} \circ \mathcal{F}$ at $(\bar x, \bar z)$ in the direction of $\Gamma_\mathcal{F}$.
	\end{prop}
	
	The following  result will be useful for computing $QDQ$s of multiple variations for a given control process.
	
	\begin{prop}\label{proppseudoaffine}\cite{papermandatoadesaim}
		Let $N,q$ be positive integers, let $(\mathbf{e}_1, \cdots, \mathbf{e}_N)$ be the canonical basis of $\mathbb{R}^N$, and let $\mathcal{F}:\mathbb{R}_+^N \to \mathbb{R}^q$ be a  (single-valued) map such that
		\begin{equation}\label{pseudoaffine}
			\mathcal{F}(\bm{\eta})-\mathcal{F}(0)=\sum\limits_{i=1}^{N} \big(\mathcal{F}(\eta^i \mathbf{e}_i)- \mathcal{F}(0)\big)+ o({|\bm\eta|}),\qquad \qquad\quad\forall \bm{\eta} =(\eta^1,\dots,\eta^N)\in\mathbb{R}_{+}^N.\end{equation} 
		
		If,  for any $i=1,\ldots,N$,  ${\mathbf S}_i\subset \Lin(\mathbb{R},\mathbb{R}^q)$ is a QDQ at $0\in\mathbb{R}$ of the map $\alpha\mapsto \mathcal{F}(\alpha \mathbf{e}_i)$ in the direction of  a set $\Gamma_i\subseteq\mathbb{R}_+$, then 
		the  compact set
		$$
		{\mathbf S} :=\Big\{L\in \Lin(\mathbb{R}^N,\mathbb{R}^q),\quad	L(v) = L_1v^1 + \ldots + L_N v^q,    \quad (L_1,\ldots, L_N)
		\in {\mathbf S}_1\times\dots \times {\mathbf S}_N\Big\} \subset \Lin(\mathbb{R}^N,\mathbb{R}^q) \qquad \footnote{When regarded as a set of matrices,  ${\mathbf S}$ is made of $q \times N$ matrices whose $i$-th column is an element of ${\mathbf S}_i$.}$$
		is a QDQ of $\mathcal{F}$ at $0$ in the direction of  $\Gamma:=\Gamma_1\times\cdots\Gamma_N$.
		
	\end{prop}

	Lemma \ref{variazionale} below generalizes the classical relation between a smooth o.d.e. and its corresponding variational equation. More precisely, it states  that, for every vector field measurable in time and locally Lipschitz in space,  {the set-valued image at $t$ of the corresponding variational differential inclusion is a QDQ of the flow map, for  any $t$ in the interval of existence.} The statement  involves the well-known notion of {\it Clarke generalized Jacobian}, whose definition is recalled below.
	\begin{definition}
		Let $n_1,n_2,N$ be positive integers and consider a subset  $A\times B\subseteq \Bbb R^{n_1}\times \Bbb R^{n_2}$, with $B$ open. Let us consider a map  $ \mathcal{G}:   A\times B\to \rr^N,$ and,  for every $(\bar z_1,\bar z_2)\in A\times B$, let us  use
		the notation 
		$ \partial^C_{z_2} \mathcal{G}(\bar z_1,\bar z_2) \subset \Lin(\Bbb R^{n_2}, \Bbb R^N)$ to mean the  Clarke generalized Jacobian  at $\bar z_2$  of the map $z_2\mapsto \mathcal{G}(\bar z_1,z_2)$, namely we set
		$$ \partial^C_{z_2} \mathcal{G}(\bar z_1,\bar z_2):=  \mathrm{co} \Bigg\{ v = \lim_{z_{2_k}\to \bar z_2}\frac{\partial \mathcal{G}}{\partial z_2}(\bar z_1,z_{2_k}) , \,\, \,\,\,(z_{2_k})_{ k\in\mathbb{N}}\subset \mathrm{Diff}( \mathcal{G}(\bar z_1,\cdot)) \Bigg\}. $$  
		
	\end{definition}
	
	\begin{lem}\label{variazionale}\cite{papermandatoadesaim}
		Let $T>0$ and $\mathcal{F}:\mathbb{R}\times\mathbb{R}^n\to \mathbb{R}^n$ be a time-dependent vector field such that, for every $x\in\mathbb{R}^n$, $\mathcal{F}(\cdot,x)$ is a bounded Lebesgue measurable map and, moreover, there exists a map  $K_\mathcal{F}\in L^\infty([0,T],\rr_+)$ verifying
		$$
		|\mathcal{F}(t,x)-\mathcal{F}(t,y)| \leq K_\mathcal{F}(t)|x-y|\qquad \forall t\in\mathbb{R},\quad \forall x,y\in \mathbb{R}^n.$$
		{  Let $q\in\mathbb{R}^n $ and,  for every $\xi$ in a neighbourhood $W$ of $q$, let us use  $t\mapsto \Phi_t^\mathcal{F}(\xi)$ to denote the solution on $[0,T]$ to  the Cauchy problem $$\begin{cases} \ds\frac{dx}{dt}=\mathcal{F}(t,x(t))\\ x(0)=\xi. \end{cases}$$}

		Then, for every  $t\in [0,T]$,
		the set
		$$\Big\{L(t),\,\, \text{$L$ is a solution of the {\rm variational inclusion}  }\,\,\,L'(s) \in  \partial^C_y \mathcal{F}\left(s,\Phi_s^\mathcal{F}(q)\right)\cdot L(s),\,\, L(0)=\mathbf{id}\Big\},$$
		where { $\mathbf{id}$ denotes the identity matrix},	is a QDQ of the map  $\xi\mapsto \Phi_t^\mathcal{F}(\xi)$ at $\xi=q$ in the direction of $\mathbb{R}^n$.
	\end{lem}
	
	\subsection{$QDQ$ approximating cones and set separation}
	\begin{definition} Let  $V$ be  a finite-dimensional real vector space. A subset $C\subseteq V$ is called a {\rm cone} if $\alpha v\in C, \forall \alpha\ge 0$ and $\forall v \in C.$ A family $\mathcal{C}$  whose elements are cones is called a {\rm multicone}. A {\rm convex multicone} is  a multicone whose elements are convex cones.
		For any given subset $E \subseteq V$, the set $E^\perp:=\{v \in \rr^n, \, v \cdot c \le 0 \,\,\forall c \in C\}\subseteq V^*$ \footnote{If $V$ is a finite-dimensional real vector space we use $V^*$ to denote its  dual space.}is a closed cone, called the {\rm polar cone of $E$}.
		We say that two cones $C_1$, $C_2$ are {\rm linearly { separable} if  $C_1^\perp \cap -C_2^{\perp}\supsetneq \{0\}$}, that is,  if there exists a linear form $\mu\in V^*\backslash\{0\}$ such that $\mu c_1\geq 0, \mu c_2\leq 0$ for all $(c_1,c_2)\in C_1\times C_2$. \end{definition}

	Let us introduce the notion of {\it transversality of multi-cones}, according to \cite{Sussmann2}.
	\begin{definition}
		Two { convex} cones $C_1, C_2$ of a vector space $V$  are said to be {\rm transversal} if $$C_1-C_2:=\big\{c_1-c_2, \quad (c_1,c_2)\in C_1\times C_2\big\} = V.$$
		Two multicones $\mathcal{C}_1$ and $\mathcal{C}_2$ are called {\rm transversal} if  $C_1\in \mathcal{C}_1$ and $C_2\in \mathcal{C}_2$ are transversal as soon as $(C_1,C_2)\in \mathcal{C}_1\times\mathcal{C}_2$. Two transversal cones $C_1$, $C_2$ are called {\rm strongly transversal} if $C_1\cap C_2 \supsetneq \{0\} $. \end{definition}

	One easily checks the following equivalences for a pair of cones $C_1$ and $C_2$:\begin{itemize}
		\item{\it $C_1$ and $C_2$ are linearly separable   if and only if   they are not transversal}.
		\item{\it  $C_1$ and $C_2$  are strongly transversal if and only if they are transversal and  there exist a non-zero linear form $\mu$  and an element $c\in C_1\cap C_2$  such that $\mu c>0$. }
	\end{itemize} 
	
	The latter  characterization is exploited to extend the notion of strong tranversality to multicones. \begin{definition} One  says that two transversal multicones $\mathcal{C}_1$, $\mathcal{C}_2$ are {\rm strongly transversal} if there exists a  linear form $\mu$ such that, for any choice of cones $C_i \in \mathcal{C}_i$, $i=1,2$, one has $\mu c >0$, for some 
		$c\in C_1\cap C_2$.

	\end{definition}
	\begin{lem} \label{nontrasversalitaforte} \cite{Sussmann2}  
		Let $\mathcal{C}_1$ and $\mathcal{C}_2$ be two multicones that are not strongly transversal. If there exists a linear operator $\mu\in C_2^\perp\backslash(-C_2^\perp)$ for all $C_2\in {\mathcal C}_2$, then there are two cones $C_1 \in \mathcal{C}_1$ and $C_2 \in \mathcal{C}_2$ that are not transversal.
	\end{lem}

	{\begin{definition}[QDQ-approximating multicones]\label{qdqmulticones}
			Let $E$ be any subset of an Euclidean space $\rn$ and let $x \in E$. A convex multicone $\mathcal{C}$ is said to be a {\rm QDQ-approximating multicone to $E$ at $x$} if there exists a  set-valued map $\mathcal{F}:\rr^N \setmap \rr^n$, a convex cone $\Gamma\subset\rr^N$, and a QDQ  $\Lambda$ of $\mathcal{F}$ at $(0,x)$ in the direction of $\Gamma$   such that  $$ \mathcal{F}(\Gamma)\subseteq E, \qquad\mathcal{C}
			=:\{L\cdot \Gamma, \,\, L \in \Lambda\}.$$
			When $\Lambda$ is a singleton, i.e. $\Lambda=\{L\}$, one  simply says  that  $\mathcal{C}=:\{ L\cdot \Gamma\}$ --or simply {\rm $C:= L\cdot \Gamma$--} is a QDQ-approximating cone to $E$ at $x$.
	\end{definition}}
	\begin{center}
32		\includegraphics[scale=0.25]{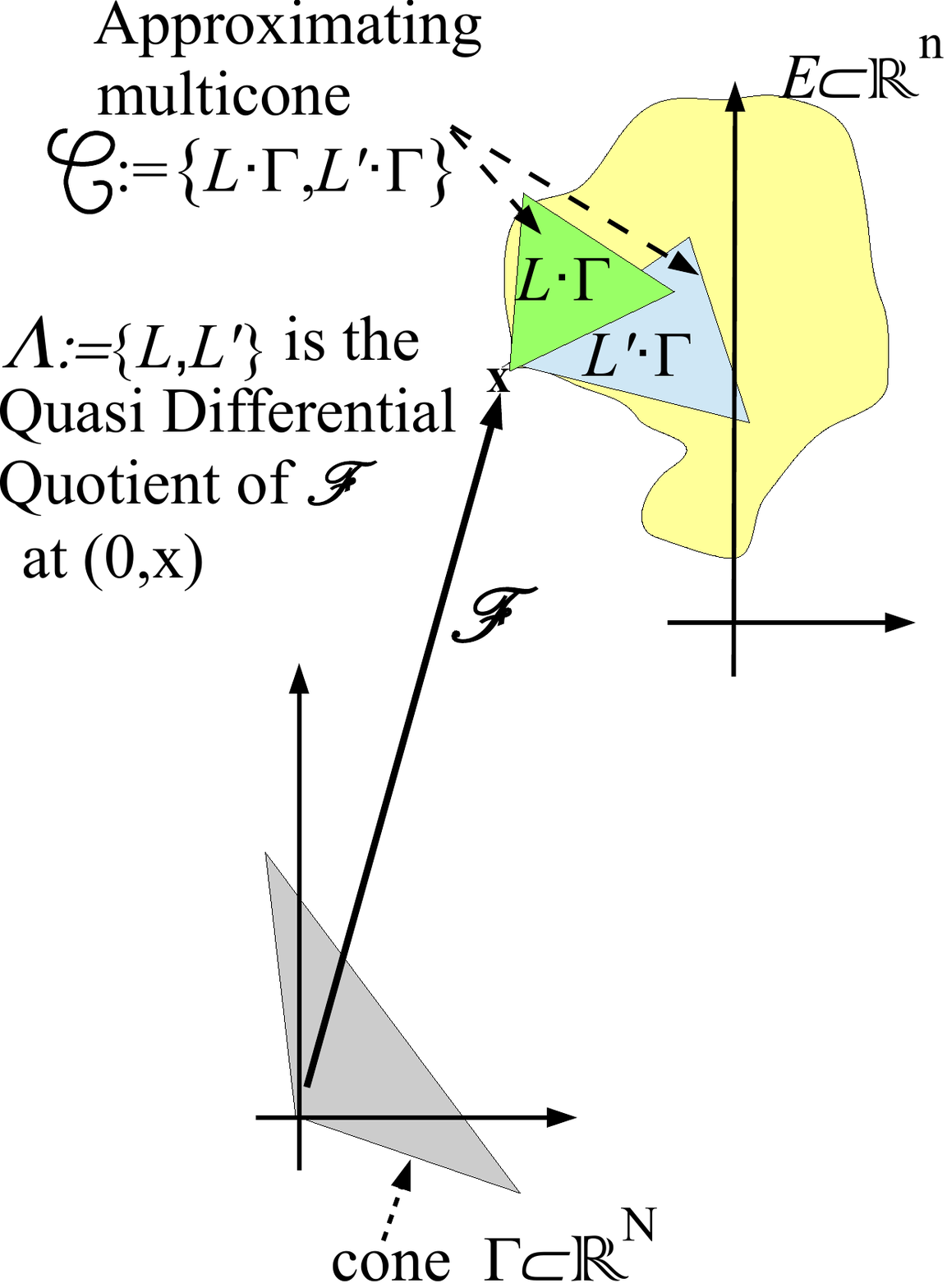}
	\end{center}
	\begin{definition}[Local separation of sets]
		Two subsets $E_1$ and $E_2$ are {\rm locally separated at $x$} if  there exists a neighbourhood $W$ of $x$ such that $\displaystyle E_1\cap E_2 \cap W =\{x\}.$ \footnote{For instance, in $\rr^3$ a plane $\pi$ and any  line not included in $\pi$ are locally separated at the point of intersection.     }
	\end{definition}
	%

	\begin{lem}\label{OpenMappingConsequence}{If two subsets $E_1$ and $E_2$ are locally separated at $x$ and if $\mathcal{C}_1$ and $\mathcal{C}_2$ are \\ QDQ-approximating multicones  for $E_1$ and $E_2$, respectively, at $x$, then  $\mathcal{C}_1$ and $\mathcal{C}_2$ are not strongly transversal.}\footnote{See Theorem 4.37, p. 265 in \cite{Sussmann2}, where the lemma was proven in the more general context of $AGDQ$'s, of which QDQ are a special case.}
		\begin{center}		\includegraphics[scale=0.30]{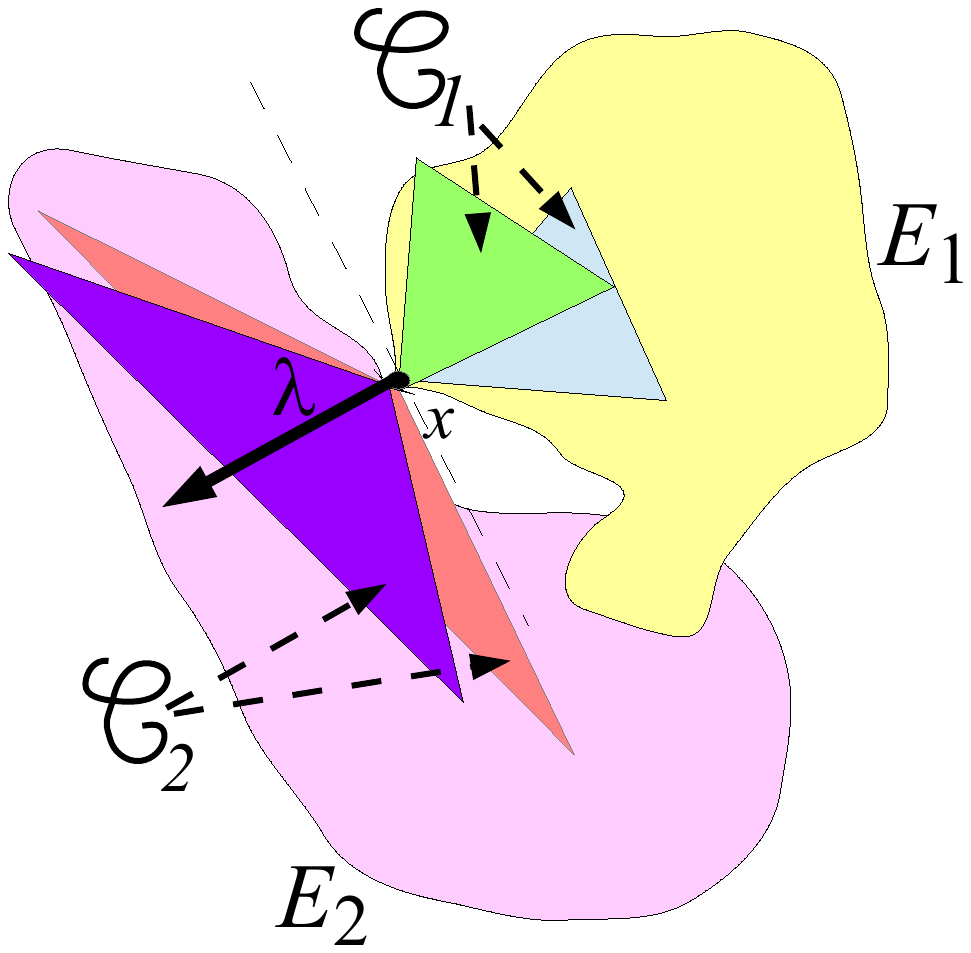}  \end{center}
		
	\end{lem}

	\section{Proof of Theorem 3.2}

	\subsection{Variations}

	A scheme of the proof is as follows.\begin{enumerate} \item One constructs a family of minimizer's {\it variations} that will be implemented almost everywhere in $[0,T]$.
		\item  There are four kinds of such variations, one for each of the conditions {\bf  (iv)-(vii)} of the  statement. The   ones concerning {\bf  (iv)} are the usual needle variations, while {\it higher order}  conditions (Goh and  Legendre\textendash{}Clebsch)  are obtained  through the implementations of suitable polynomial controls, as in the classical theory for smooth systems. 
		\item Actually the above-mentioned polynomial controls are implemented on the (smooth) convoluted system, so that  results concerning higher order variations are well-known. 
		\item Sucessively we make the convolution parameter  tend to zero, and at the limit we consider the so-called {\it set-valued  Lie brackets }.
		\item Further, trough the adjoint inclusion, we {\it transport} the variations from almost al times in $[0,T]$ to the end-time $T$ and impose the set-separation of the {\it profitable set} and the {reachable set}. This is first done for a finite number of times, and successively extended to almost all $t$ in $[0,T]$, via Cantor's theorem.  By interpreting  set-separation one obtains the transversality and all conditions of the thesis.   \end{enumerate}

	\begin{definition}[Variation generators]
		Let us define the set $\mathfrak{V}$ of {\it  variation generators} as the union
		$$\mathfrak{V}:=\mathfrak{V}_{ndl}\bigcup \mathfrak{V}_{Goh} \bigcup \mathfrak{V}_{LC2} \bigcup \mathfrak{V}_{LC3},$$
		where
		$\mathfrak{V}_{ndl}:=  U ,$ 
		$\mathfrak{V}_{Goh} :=\Bigl\{(i,j)\,\,\,i\neq j\,\,\,\, i,j=1,\ldots,m\},$
		$\mathfrak{V}_{LC2}=:\{(0,i)\,\,, i=1,\ldots,m  \},   $
		and $\mathfrak{V}_{LC3}=:\{(1,0,1)\}$.
	\end{definition}

	In order to define control variations associated to the four different  kinds of variation generators, 
	let us consider 
	{the Hilbert space $\mathcal{P}([0,1])$} of polynomials defined on $[0,1]$, endowed with the scalar product $\langle P_1,P_2\rangle=\displaystyle\int_0^1 P_1(s)P_2(s) ds$. We will be concerned with the subspace
	$\mathcal{P}^\sharp([0,1]) \subset\mathcal{P}([0,1])$ defined as 
	$$
	\mathcal{P}^\sharp([0,1]) := \left\{P\in \mathcal{P}([0,1]) : \ \ \ \ 
	0=P(0)=P(1) = \int_0^1 P(t)dt \right\}.
	$$

	Clearly,
	the vector  subspace  $\mathcal{P}^\sharp([0,1])\subset  \mathcal{P}([0,1])$ has  finite codimension.\footnote{The choice of these spaces of functions to construct control variations as well as of the polynomials introduced  in Definition \ref{iP} below finds its  motivation in the expansions of solutions mentioned at the end of the Introduction.}

	\begin{definition}\label{iP} 
		For all   $r,i =1,\ldots,m$  and  $j=0,\ldots,m, \ \ j<i$, we select (arbitrarily)
		the polynomials $P^{r}_{(i,j)}$  
		defined as follows:	
		\begin{itemize}

			
			\item  if $j\not=0$
			we choose $P_{(j,i)}^{m}\in \mathcal{P}^\sharp([0,1])\backslash\{0\}$ arbitrarily, then
			we proceed recursively on $r$, backward from $m$ to $1$,   and  select (arbitrarily) $$\begin{array}{l}\ds P_{(j,i)}^{r}\in \mathcal{P}^\sharp([0,1]) \bigcap \left\{{dP_{(j,i)}^{m}\over dt},{dP_{(j,i)}^{m-1}\over dt},\ldots,{dP_{(j,i)}^{r+1}\over dt}\right\}^{\bigperp}\setminus\Big\{0\Big\},\quad  \forall r\in \left\{1,\ldots,m-1\right\}\setminus\left\{j\right\},\\\ds P_{(j,i)}^{j}\in \mathcal{P}^\sharp([0,1]) \bigcap \left\{{dP^{m}_{(j,i)}\over dt},{dP_{(j,i)}^{m-1}\over dt},\ldots,{dP_{(j,i)}^{j+1}\over dt}\right\}^{\bigperp}\setminus\left\{{dP_{(j,i)}^{i}\over dt}\right\}^{\bigperp}.\end{array}\qquad \footnotemark$$ 
			\footnotetext{Because of a certain sign requirement that is made clear in the proof of Theorem \ref{eccolestime} in the Appendix, we might need to change the sign of one  but not both of the polynomials $\left\{P^{i},P^{ji}_{j}\right\}$.} 
			
			%
			%

			\item if $j=0$, we choose $P_{(0,i)}^{m}\in \mathcal{P}^\sharp([0,1])\setminus\{0\}$ arbitrarily, and once again we proceed recursively on $r$ backward from $m$ to $1$, and select (arbitrarily)
			$$\begin{array}{l}\ds P^{r}_{(0,i)}\in \mathcal{P}^\sharp([0,1]) \bigcap \left\{{dP^{m}_{(0,i)}\over dt},{dP^{m-1}_{(0,i)}\over dt},\ldots,{dP^{r+1}_{(0,i)}\over dt}\right\}^{\bigperp}\setminus\Big\{0\Big\},\quad  \forall r\in \left\{1,\ldots,m-1\right\}\setminus\left\{i\right\},\\\ds P^{i}_{(0,i)}\in \left\{p\in\mathcal{P}([0,1]):\ \ \ \ds\int_0^1{dp\over dt}(t)dt=0, \ \int_0^1t{dp\over dt}(t)dt\neq0 \right\} \bigcap \left\{{dP^{m}_{(0,i)}\over dt},{dP^{m-1}_{(0,i)}\over dt},\ldots,{{dP^{1}}_{(0,i)}\over dt}\right\}^{\bigperp}.\end{array}  $$

			\item Furthermore, let us consider  the polynomial 	$  \tilde P(t):=5t^4-10t^3+6t^2-t                   $, and observe that $\tilde P\in \mathcal{P}^\sharp([0,1])$.


		\end{itemize}
	\end{definition}

	\begin{definition}[Families of control variations]\label{defvar}
		Let us fix a map  $u\in L^\infty([0,T],\rr^m) $ and a time  $\ol{t} \in (0,{T})$, and
		let the polynomials $P^{r}_{(j,i)}, \tilde P$,  ($r,i=1,\ldots,m$, $j=0,\ldots,m,$  $j<i$), be  as in Definition \ref{iP}. To every    $\mathbf{c}=\hat u\in \mathfrak{V}_{ndl}(=U)$ and every $\alpha>0$ we associate  the family  of $\varepsilon$-dependent  controls ${\big\{ {{ u}}_{\alpha,\varepsilon,\mathbf{c},\ol{t}}(t): \, \varepsilon \in [0,\ol{t})\big\}}$ defined as follows.
		\begin{enumerate}
			\item If  $\mathbf{c}=\hat u\in \mathfrak{V}_{ndl}(=U)$  we set  \bel{perneedle}{{u}}_{\alpha,\varepsilon,\mathbf{c},\ol{t}}(t):=\begin{cases}{{u}}(t) &\text{ if } t \in [0,\ol{t}-\varepsilon) \cup (\ol{t},T]\\  \hat u &\text{ if } t \in [\ol{t}-\varepsilon,\ol{t}],\end{cases}\eeq 
			This  family  is usually referred as a {\rm needle variation} of $u$. (Here $\alpha$ is a fictious parameter, and we use it just because it is used for all other kind of variations $\mathbb{c  }$).
			
			\item 
			If
			$\mathbf{c}=(j,i)\in  \mathfrak{V}_{Goh}\cup \mathfrak{V}_{LC2} $ we associate the $\varepsilon$-parameterized family   ${\Big\{ {{u}}_{\alpha,\varepsilon,\mathbf{c},\ol{t}}(t): 0<\sqrt{\varepsilon}\le\ol{t},\, \alpha>0\Big\}}$ of controls defined as
			\bel{perLC2}{u}_{\alpha,\varepsilon, \mathbf{c},\ol{t}}(t):=\begin{cases}
				{{u}}(t) & \text{ if } t \not \in [\ol{t}-\sqrt{\varepsilon},\ol{t}]\\\ds
				{{u}}(t)+\alpha \sum\limits_{r=1}^m {dP^{r}_{(j,i)}\over dt}\left(\frac{t-\ol{t}}{\sqrt{\varepsilon}}+1\right)\mathbf{e}_r& \text{ if }t \in [\ol{t}-\sqrt{\varepsilon},\ol{t}] .
			\end{cases}\eeq
			We call  this family {\rm Goh variation  of ${{u}}$ at $\bar t$} in the case when $i>0,j >0 $, while we call it   {\rm  Legendre\textendash{}Clebsch variation of step 2 {\rm(}of ${{u}}$ at $\bar t${\rm)} }  as soon as $j > 0$ $i=0$.
			\item 
			If $m=1$, $g:=g_1$, 
			$\{\mathbf{c}\} = \mathfrak{V}_{LC3}$,
			we call   {\rm  Legendre\textendash{}Clebsch variation of step $3$ (of ${{u}}$ at $\bar t$)} the family  ${\Big\{ {{u}}_{\alpha,\varepsilon,\mathbf{c},\ol{t}}(t): \, 0<\sqrt[3]{\varepsilon}\le\ol{t},\alpha>0\Big\}}$ of controls defined as
			\bel{perLC}{{u}}_{\alpha,\varepsilon, \mathbf{c},\ol{t}}(t):=\begin{cases}
				{{u}}(t) & \text{ if } t \not \in [\ol{t}-\sqrt[3]{\varepsilon},\ol{t}]\\\ds
				{{u}}(t)+\alpha {d\tilde P\over dt}\left(\frac{t-\ol{t}}{\sqrt[3]{\varepsilon}}+1\right)& \text{ if }t \in [\ol{t}-\sqrt[3]{\varepsilon},\ol{t}]\\
			\end{cases}\eeq
			
	\end{enumerate}\end{definition}
	
	
	
	Theorem \ref{eccolestime}  below is a classical result  concerning  smooth control systems.   Though it can  be recovered  in classical literature, possibly stated in a slightly  different form (see e.g. \cite{Schattler}, Subsect.2.8),
	the interested reader can find its proof in the Appendix.

	\begin{theorem} \label{eccolestime} Let  $(u,x)$ be a process, and 
		let the control $u$ be  singular.
		Choose   $\ol{t}, \varepsilon$ such that  the vector fields  $f$ and $g_i$  are  of class $C^1$ around $x(\ol{t})$ and, moreover,  
		$0<\sqrt[3]{\varepsilon}<\min\{\ol{t},1\}$. 
		Then there exist $M,\bar\alpha>0$ such that, for any $\alpha\in[0, \bar\alpha] $ and any  $\mathbf{c} \in \mathfrak{V} $, using $x_\varepsilon$ to denote the trajectory corresponding to the perturbed control  ${u}_{\alpha,\varepsilon,\mathbf{c},\ol{t}}$,\footnote{Of course $x_\varepsilon$ depends on $\alpha, \mathbf{c}$, and $\ol{t}$, as well. } we get the following three statements.
		
		\begin{enumerate} 	\item If  $\mathbf{c}=(0,i) \in \mathfrak{V}_{LC2}$, for some $i=1,\ldots,m$,
			then $$x_\varepsilon(\ol{t})-x(\ol{t})=M\alpha^2\varepsilon[f,g_i](x(\ol{t}))+o(\varepsilon)\,\,$$
			\item If $\mathbf{c}=(j,i) \in \mathfrak{V}_{Goh}$  (i.e. $j=1,\ldots,m$, $1<i\leq m$) 
			then $$x_\varepsilon(\ol{t})-x(\ol{t})=M\alpha^2\varepsilon[g_j,g_i](x(\ol{t}))+o(\varepsilon)\,\,$$
			\item If $m=1$, $\mathbf{c} = \mathfrak{V}_{LC3}$ and, moreover, $f,g (:=g_1)$ are of class $C^2$ around ${x(\ol{t})}$,
			then $$x_\varepsilon(\ol{t})-x(\ol{t})=M\alpha^3\varepsilon[g,[f,g]](x(\ol{t}))+o(\varepsilon)\,\,$$
		\end{enumerate} 
		
	\end{theorem}

	\subsection{Composing multiple variations}
	For any $(\ol{t},\mathbf{c})\in ]0,{T}] \times \mathfrak{V}$  and any $\varepsilon$ and $\alpha$ sufficiently small, consider the operator  $$\mathcal{A}_{\alpha,\varepsilon,\mathbf{c},\tau}:
	L^\infty([0,T],\rr^m)\to  L^\infty([0,T],\rr^m)  \qquad \mathcal{A}_{\alpha,\varepsilon,\mathbf{c},\tau}({u}):={u}_{\alpha,\varepsilon,\mathbf{c},\tau}.$$ 
	In addition, let $N$ be a natural number  and let us consider  $N$ variation generators $\mathbf{c}_1,\ldots,\mathbf{c}_N \in \mathfrak{V}$ and $N$ instants $ 0<{t}_{1} <  \ldots {t}_N\leq {T}$. For a $\tilde\varepsilon>0$ sufficiently small and  for every process $(u,x)$, let us define the {\it  multiple variation}  $$[0,\tilde\varepsilon]^N\ni \small{\bm{\varepsilon}}\mapsto {{u}}_{\bm{\varepsilon}}:=\mathcal{A}_{\alpha,\varepsilon_N,\mathbf{c}_N,{t}_N} \circ \ldots \circ \mathcal{A}_{\alpha,{\varepsilon_1},\mathbf{c}_1,{t}_1} ({{u}})$$
	and let $({{u}}_{\bm{\varepsilon}},{x}_{\bm{\varepsilon}})$ be the corresponding process. \footnote{Of course, $({{u}}_{\bm{\varepsilon}},{x}_{\bm{\varepsilon}})$  depends on the parameters $\alpha$, $\mathbf{c}_k$ and ${t}_k$ as well.}

	Lemma \ref{lemmapseudoaffine}, which has been proved in \cite{papermandatoadesaim},  will be used to reduce the problem of estimating the effect of multiple perturbations to the task of  evaluating  single perturbations and then summing them up (through an application of  Proposition \ref{proppseudoaffine}).
	\begin{lem} \label{lemmapseudoaffine}
		The map
		$$\bm{\varepsilon}\mapsto
		x_{\bm{\varepsilon}}({T})
		$$ from $[0,\tilde \varepsilon]^N$ into  $\mathbb{R}^{n}$  satisfies  \eqref{pseudoaffine}
		(with $x_{(\cdot)}({T})$ in place of $F(\cdot)$ and  $q:=1+n$),
		namely, one has
		\begin{equation}\label{pseudoaffine2}
			x_{\bm{\varepsilon}}({T})-x_{0}({T})=\sum\limits_{i=1}^{N} \Big(	x_{\varepsilon_i \mathbf{e}_i}({T}) - x_{0}({T})\Big)+ o(|\bm{\varepsilon}|),\qquad \forall \bm{\varepsilon} =(\varepsilon^1,\dots,\varepsilon^N)\in [0,\tilde \varepsilon]^N.\end{equation}
	\end{lem}

	\subsection{Set-valued Lie brackets as $QDQ$s}
	
	Everything we have done so far assumes  that the  vector fields $f,g_1,\dots$ are of class $C^1$ on neighbourhoods of the points $\ol{x}(t_j)$.  {In fact, for  the Legendre\textendash{}Clebsch type  variation including  the bracket $[g,[g,f]]$ ($m=1$)} a $C^2$-regularity is requested in the classical case. Now  we aim to  a result similar to Theorem \ref{eccolestime} for  non-smooth vector fields, hence dealing with set-valued brackets, in order to successively apply a set-separation argument for QDQ's. For this purpose, let us mollify our vector fields and let us see how this leads to an approximation of set-valued brackets as well, which finally shows the latter are indeed QDQ's.
	
	Let $\varphi_\eta:{\Bbb R^n}\to {\Bbb R}^+$ ($\eta>0$) be a mollifier kernel.\footnote{For instance, 	$\varphi(x):= C \exp\left( \frac{1}{\left|{x}\right|^2-1}\right)$ if $|x|<1$ and $\varphi(x):=0$ if $|x|\geq 1$, with the constant $C$ such that $\int_{\Bbb R^n} \varphi(x)dx=1$. Moreover ,  for every $\eta<1$  $\varphi_\eta:\Bbb R^n \to \mathbb{R}$ is defined by  setting $\varphi_\eta(x):={1\over \eta^n} \varphi  \left({x\over \eta}\right)$.} For any $L^1_{\mathrm loc}$ vector field $X$ on  $\mathbb{R}^n$ we denote by $X_\eta$ the ($C^\infty$) $\eta$-convolution (for $\eta$ sufficiently small) 
	$$X_\eta(x):=\int_{\mathcal{B}_1} X(x-h)\varphi_\eta(h)\,dh, \, \, \, \forall x \in \Omega,$$
	where $\mathcal{B}_1$ is the unit ball in ${\Bbb R^n}$.
	In \cite{AngrisaniRampazzoQDQ} we have proved the following simple result.
	\begin{lem} \label{mollifyingeasy}
		Let $X$ and $Y$ be locally Lipschitz continuous vector fields. Then $$|[X_\eta,Y_\eta](x)-[X,Y]_\eta(x)|\le \rho(\eta)\quad \forall x \in \mathbb R^n$$  for a suitable ($x$-dependent) modulus $\eta\mapsto\rho(\eta)$ (where $[X,Y]_\eta$ is the $\eta$-convolution of the $L^\infty_{\mathrm loc}$ vector field $[X,Y]$).  
	\end{lem}
	We need  now an analogous result  for  brackets of length $3$.
	\begin{lem} \label{mollifying}
		Let $X$,$Y$  be $C^{1,1}$ vector fields and let $Z$ be of class $C^{0,1}_{\mathrm loc}$. Then $$|[[X_\eta,Y_\eta],Z_\eta](x)-[[X,Y],Z]_\eta(x)|\le \rho(\eta)  \  \  \quad \quad \forall x \in \mathbb R^n$$
		for a suitable modulus $\rho$  (where $[[X,Y],Z]_\eta$ is the $\eta$-convolution of the $L^\infty$ vector field $[[X,Y],Z]$). 
	\end{lem}
	\begin{proof}
		Upon only requesting local Lipschitz regularity for $X$ and $Y$, one has  that
		\begin{equation}\label{cdcgiafatto}[X_\eta,Y_\eta](x)=\int_{B_1} \varphi_{\eta}(h)[X,Y](x-h)\,dh+E_\eta,\end{equation}
		where \begin{multline*}E_\eta(x):=\int_{\Xi_\eta(x)}\varphi_{\eta}(\nu)\Big[DY(x+\eta\nu)\Big(X(x+\eta\nu))-X_\nu(x)\Big)-DX(x+\eta\nu)\Big(Y(x+\eta\nu)-Y_\nu(x)\Big)\Big]\,d\nu,\end{multline*} 
		with the domain of integration being defined as $\Xi_\eta(x):=\{\nu {{\in \mathbb{R}^n}}: x+\eta\nu \in \mathrm{Diff}(X)\cap \Diff(Y)\}$.
		Since $[X,Y]$ is itself a locally Lipschitz  continuous vector field, this same  idea  applies also to get 
		$$|[[X,Y]_\eta,Z_\eta](x)-[[X,Y],Z]_\eta(x)|\le \rho_2(\eta).$$
		for a suitable modulus $\rho_2$.
		By \eqref{cdcgiafatto} we have \begin{multline*}\Big|[[X_\eta,Y_\eta],Z_\eta](x)-[[X,Y],Z]_\eta(x)\Big|\le \\ \le \Big|[[X,Y]_\eta,Z_\eta](x)-[[X,Y],Z]_\eta(x)\Big|+\Big|[[X,Y]_\eta,Z_\eta](x)-[[X_\eta,Y_\eta],Z_\eta](x)\Big|\le \\ \le  \rho_2(\eta)+\Big|[E_\eta,Z_\eta](x)\Big|\le \\ \le  \rho_2(\eta)+\rho_1(\eta)\|DZ_\eta\|_{\infty}+|DE_\eta(x)|\|Z_\eta\|_{\infty},\end{multline*}
		where $\rho_1$ is a modulus majorizing the difference $|[X,Y]_\eta(x)-[X_\eta,Y_\eta](x)|$. 
		Now observe that the proof is concluded provided we are able to show that $DE_\eta(x)\le \rho_3(\eta)$ for a suitable modulus $\rho_3(\eta)$. 
		Actually, since i) $DE_\eta(x)$ can be rewritten by moving the differential operator D inside  the integral, ii) the  second derivatives of the involved vector fields  are bounded,  iii) the   quantities like  $X(x+\eta\nu)-X_\nu(x)$ and $Y(x+\eta\nu)-Y_\nu(x)$ converge to $0$ together with their derivatives as the parameter $\eta$ goes to zero, the following chain of inequalities hold true  \begin{multline*} \Big|DE_{\eta}(x)\Big|=\left|\int_{\Xi_\eta(x)} D\varphi_{\eta}(\nu)\Big[DY(x+\eta\nu)\Big(X(x+\eta\nu))-X_\nu(x)\Big)-DX(x+\eta\nu)\Big(Y(x+\eta\nu)-Y_\nu(x)\Big)\Big]\,d\nu +\right. \\ \left.\int_{\Xi_\eta(x)}\varphi_{\eta}(\nu)D\Big[DY(x+\eta\nu)\Big(X(x+\eta\nu))-X_\nu(x)\Big)-DX(x+\eta\nu)\Big(Y(x+\eta\nu)-Y_\nu(x)\Big)\Big]\,d\nu\right| \le \\  ||D\phi_{\eta}||_{\infty}||DY||_{\infty}||X(x+\eta\nu))-X_\nu(x)||_{\infty} + ||D\phi_{\eta}||_{\infty}||DX||_{\infty}||Y(x+\eta\nu))-Y_\nu(x)||_{\infty} + \\ ||D^2Y||_{\infty}||X(x+\eta\nu))-X_\nu(x)||_{\infty} + ||DY||_{\infty}||DX(x+\eta\nu))-DX_\nu(x)||_{\infty}+\\
			||D^2X||_{\infty}||Y(x+\eta\nu))-Y_\nu(x)||_{\infty} + ||DX||_{\infty}||DY(x+\eta\nu))-DY_\nu(x)||_{\infty}\end{multline*}
		By taking $\rho_3(\eta)$ coinciding with the whole last term  we are done.
	\end{proof}
	For any  fixed $\tilde t \in (0,T)_{Leb}$,	in Proposition \ref{mainprop} below we determine a QDQ for the function $\bm{\varepsilon}\mapsto  x_{\bm{\varepsilon}}(\tilde{t})$ in the case of a single perturbation, i.e. $N=1$.

	
	
	\begin{definition} [Variation vector sets]	Let us consider the local $L^1$ minimizer $(\ol{u},\ol{x})$ of problem $(P)$.
		Let $(0,T)_{\mathrm{Leb}}\subset [0,T]$ be the (full-measure) set of Lebesgue points of the function $t \mapsto f(\ol{x}(t)) +\displaystyle\sum_{i=1}^mg_i(\ol{x}(t))\ol{u}^i(t) $. For every variation generator $\mathbf{c}\in \mathfrak{V}$ and every $\tilde t \in (0,T)$
		let us associate  the set  $ v_{\mathbf{c},\tilde t} $ defined below, which we will call  the  {\rm variation vector set at $\tilde t$}\footnote{Let us point out that $ v_{\mathbf{c},\tilde t}$ is  a {\it subset} of vectors of  $\mathbb{R}^{1+n}$ , though  it reduces to a singleton as soon as $\mathbf{c}\in \mathfrak{V}_{ndl}$.}.
		$$
		v^{ }_{\mathbf{c},\tilde t}:= \left\{
		\begin{array}{ll}  \left\{ \displaystyle\sum_{i=1}^mg_i(\ol{x}(\tilde t))\Big(u^i-\ol{u}^i(\tilde t)\Big)\right\}&\begin{array}{l}\text{if}\,\,
				\mathbf{c}=u\in \mathfrak{V}_{ndl}, \\\,\,\,\,\, \text{and}\,\,\, \tilde t \in (0,T)_{\mathrm{Leb}}\end{array}
			\\[5mm]
			[g_j,g_i]_{\mathrm{set}}(\ol{x}(\tilde t))\qquad\qquad &\begin{array}{l}\,\,\, \text{if}
				\,\, \mathbf{c}=(j,i)\in \mathfrak{V}_{Goh}.
			\end{array}\\[8mm]
			[f,g_i]_{\mathrm{set}}(\ol{x}(\tilde t))\qquad\qquad &\begin{array}{l}\,\,\, \text{if}
				\,\, \mathbf{c}=(0,i)\in \mathfrak{V}_{LC2}.
			\end{array}		\\[5mm]
			[g,[f,g]]_{\mathrm{set}}(\ol{x}(\tilde t)) &\begin{array}{l}\,\,\, \text{if}
				\,\, m=1,\,\,\, \mathbf{c}=\{(1,0,1)\} = \\ \mathfrak{V}_{LC3} \,\,\, (g:= g_1) \,\,\,\text{and}\,\,\,   \tilde t  \\ \text{ is such that $f$ and $g$ }\\ \text{are of class $C^{1,1}$ around $\ol{x}(\tilde t)$.}
			\end{array}
		\end{array}\right.
		$$
	\end{definition}
	
	\begin{prop} \label{mainprop}Let $\tilde t \in (0,T)_{Leb}$ be a Lebesgue point of the map $t \mapsto f(\ol{x}(t)) +\displaystyle\sum_{i=1}^mg_i(\ol{x}(t))\ol{u}^i(t) $, where $(\ol{u},\ol{x})$ is a  local $L^1$ minimizer  of problem $(P)$.
		%
		Then, for every needle variation  generator $\mathbf{c} =\hat u\in U=\mathfrak{V}_{ndl}$, 
		the singleton $v^{ }_{\mathbf{c}} ={\small \left\{ \displaystyle\sum_{i=1}^mg_i(\ol{x}(\tilde t))\Big(\hat u^i-\ol{u}^i(\tilde t)\Big)\right\}}$ 
		is a QDQ of the  map  $\varepsilon \mapsto x_{\varepsilon}({\tilde t})$
		at $\varepsilon=0$ in the direction of $\rr_+$ (where, as above,  $t\to  x_{\varepsilon}(t)$ is the trajectory associated to the needle control ${\chi}_{[0,T]\backslash{[\tilde t-\varepsilon]}}\ol{u} + {\chi}_{[\tilde t-\varepsilon]}\hat u $).
		
		Furthermore, consider $\mathbf{c}\in \mathfrak{V}\backslash \mathfrak{V}_{ndl}$,
		$0<\sqrt[3]{\varepsilon}<\min\{{\tilde t},1\}$, { and $\alpha\leq \bar \alpha$, $\bar \alpha$ being the same as in the statement of Theorem \ref{eccolestime}}.
		Then
		the set $v^{ }_{\mathbf{c}}$ 
		is a QDQ of the  map  $\varepsilon \mapsto x_{\varepsilon}({\tilde t})$
		at $\varepsilon=0$ in the direction of $\rr_+$, where $t\mapsto  x_{\varepsilon}(t)$ is the trajectory associated to the needle control ${u}_{\alpha,\varepsilon,\mathbf{c},{\tilde t}} $
	\end{prop}
	
	\begin{proof}
		First of all, let us observe that   if $\mathbf{c} $ is a needle variation generator,  i.e. $\mathbf{c} = \hat u \in \mathfrak{V}_{ndl}$, this is just a basic result in  optimal control theory, for   \,\,\, $\displaystyle\sum_{i=1}^mg_i(\ol{x}(\tilde t))\Big(\hat u^i-\ol{u}^i(\tilde t)\Big)$ coincides with  the derivative  $\displaystyle\left(\frac{d}{d\varepsilon}\right)_{\varepsilon=0}\left(x_{\varepsilon}(\tilde t)-\ol{x}(\tilde t)\right)$. 
		
		To examine the case when  $\mathbf{c}\in \mathfrak{V}\backslash \mathfrak{V}_{ndl}$, let us consider the $C^\infty$ mollified vector fields $f_\eta$, and $g_{i,\eta}$ and the corresponding   family of smooth Cauchy problems $(CP_\eta),\,\,\eta\geq0\}$ defined by setting
		$$(CP_\eta)\qquad
		\begin{cases}
			\displaystyle \frac{dx}{dt}=f_\eta(x(t))+\sum\limits_{i=1}^m g_{i,\eta}(x(t))\ol{u}^i(t) \quad  a.e.  \,\, t \in [0,T],\\
			\displaystyle x(0) =\hat{x},\qquad 
		\end{cases} $$
		whose unique solution we refer to as $x_{\eta}$. We will  use the notation $x_{\eta,\varepsilon}$  when we implement perturbed control $u_{\varepsilon}$ in place of $\ol{u}$. We can  extend  the family of our problems to  $\eta \ge 0$, with the understanding that writing for  $\eta=0$ we get the original, not mollified, system.\\
		Since the vector fields $f,g_1,\ldots,g_m$ are locally Lipschitz continuous, one gets, for some $\tilde\varepsilon>0$ and   all $t\in [0,T]$ and for all $\varepsilon\in [0,\tilde\varepsilon]$, $$|x_{\eta,\varepsilon}(t)-x_{\varepsilon}(t)|\le\int_0^t\Big| f_\eta(x_{\varepsilon}(s))+\sum\limits_{i=1}^m g_{i,\eta}(x_{\varepsilon}(s))\ol{u}^i_{\varepsilon}(s) -  f(x_{\varepsilon}(s))+\sum\limits_{i=1}^m g_{i}(x_{\varepsilon}(s))\ol{u}^i_{\varepsilon}(s) \Big|ds\le c\eta $$ for some constant $c>0$ independent of $\varepsilon\in [0,1]$ . \\
		Let us set, for any $i,j=1,\dots,m$,  $$ {Br_\eta} =\begin{cases}
			[f_\eta,g_{i,\eta}](\ol{x}({\tilde t})) &\text{ if }\mathbf{c}=(0,i)\\
			
			[g_{j,\eta},g_{i,\eta}](\ol{x}({\tilde t})) &\text{ if }\mathbf{c}=(j,i)\\
			
			[g_{\eta},[f_\eta,g_{\eta}]](\ol{x}({\tilde t})) &\text{ if } \mathbf{c}=(1,0,1) \,\,\\ \end{cases} \qquad \check{Br}_\eta=\begin{cases}
			[f,g_i]_\eta(\ol{x}({\tilde t})) &\text{ if }\mathbf{c}=(0,i)\\
			
			[g_j,g_i]_\eta(\ol{x}({\tilde t})) &\text{ if }\mathbf{c}=(j,i)\\
			
			[g,[f,g]]_\eta(\ol{x}({\tilde t})) &\text{ if }\mathbf{c}=(1,0,1), 
		\end{cases}
		$$
		and $$Br_{set}=\begin{cases}
			[f,g_i]_{\mathrm{set}}(\ol{x}({\tilde t})) &\text{ if }\mathbf{c}=(0,i)\\
			
			[g_j,g_i]_{\mathrm{set}}(\ol{x}({\tilde t})) &\text{ if }\mathbf{c}=(j,i)\\
			
			[g,[f,g]]_{\mathrm{set}}(\ol{x}({\tilde t})) &\text{ if }\mathbf{c}=(1,0,1), 
		\end{cases}$$ 
		(where it is understood  that  $m=1$, $g:=g_1$, as soon as $\mathbf{c}=(1,0,1)$).
		Theorem \ref{eccolestime}, which can be applied as we are dealing with smooth fields $f_\eta$ and $g_{i,{\eta}}$, together with Lemmas \ref{mollifyingeasy} and \ref{mollifying} give
		{	$$x_{\varepsilon}({\tilde t})=x^\eta_{\varepsilon}({\tilde t})+o(\sqrt{\eta})=\ol{x}({\tilde t})+\varepsilon Br_\eta+o(\varepsilon)+o(\sqrt{\eta}) =\ol{x}({\tilde t})+\varepsilon  \check{Br}_\eta+o(\varepsilon)+o(\sqrt{\eta})+\rho(\eta)
			$$}
		Now let us choose $\eta=\eta(\varepsilon)$ as a function of $\varepsilon$ such that  {
			$o(\sqrt{\eta(\varepsilon)})+\rho(\eta(\varepsilon )) = o(\varepsilon)$
		}, so obtaining 
		$$
		x_{\varepsilon}({\tilde t})  =\ol{x}({\tilde t})+\varepsilon \check{B}r_{\eta(\varepsilon)}+o(\varepsilon).
		$$ 
		Since $ \check{B}r_{\eta(\varepsilon)}$ clearly has vanishing distance from $Br_{set}$ as $\varepsilon\to 0$, this is sufficient  for $Br_{set}$ to be a QDQ of the map $\varepsilon\mapsto x_{\varepsilon}({\tilde t})$. 
		Therefore
		we conclude that %
		$
		v^{ }_{\mathbf{c},{\tilde t}} = Br_{set}$ is  a QDQ of the map $ \varepsilon\mapsto x_{\varepsilon}({\tilde t})$  at $\varepsilon=0$ in the direction of $\Bbb R_+$.
	\end{proof}

	Let us consider the subset $\hat\Lambda\subset  \Lin(\rr^N,\rr^{n+1})$  defined as \footnote{The notation $\overset{meas}\in$ means "is a measurable selection of". As for the case of the following brace \linebreak $ M(\cdot) \overset{meas}\in\partial_y^C\Big(f(\ol{x}(\cdot)) +\displaystyle\sum_{i=1}^mg_i(\ol{x}(\cdot))\ol{u}^i(\cdot)\Big)$,  we are sure that this measurable selection $M(\cdot)$ does exist, in that the set\textendash{}valued map  $t\mapsto \partial_y^C\Big(f(\ol{x}(t)) +\displaystyle\sum_{i=1}^mg_i(\ol{x}(t))\ol{u}^i(t)\Big)$ is closed, and convex.
	}  
	{\small
		$$
		\hat\Lambda:=\left\lbrace 
		\left(\begin{pmatrix} e^{\bigintssss_{{t}_1}^{{T}}{M}\,}  \\  \omega\cdot e^{\bigintssss_{{t}_1}^{{T}}{M}\,}\end{pmatrix}\cdot V_1, \dots, \begin{pmatrix} e^{\bigintssss_{{t}_N}^{{T}}{M}\,} \\  \omega\cdot e^{\bigintssss_{{t}_N}^{{T}}{M}\,}\end{pmatrix}\cdot V_N\right)  , \quad 
		\begin{array}{l}  M(\cdot) \overset{meas}\in\partial_y^C\Big(f(\ol{x}(\cdot)) +\displaystyle\sum_{i=1}^mg_i(\ol{x}(\cdot))\ol{u}^i(\cdot)\Big) \\ \omega \in \partial^C\Psi\big(T))  
			\\\\V_k \in  v_{\mathbf{c}_k,{s}_k}\quad\forall k=1,\ldots,N \end{array}
		\right \rbrace,$$
	}
	where, for every $\tau,t\in [0,T]$ $\tau\leq t$, $e^{\bigintssss_{{\tau}}^{t}{M}\,}$ denotes the fundamental matrix solution of the linear o.d.e. $dv/ds = M(s) v(s)$, namely the   
	value at $t$ of the matrix  solution of the  matrix linear differential equation  $dE/ds = M(s)\cdot E(s)$ with initial condition $E({\tau}_k)=  \mathbf{id}_{n\times n}$.
	
	\begin{theorem}\label{qdqth}
		The set  $\hat\Lambda$ is a QDQ  of the map $${\mathbf{\hat{\Phi}}}:\Bbb R^N\ni \bm{\varepsilon}\mapsto \begin{pmatrix}x_{\bm{\varepsilon}}(T)\\\Psi\big(x_{\bm{\varepsilon}}(T)\big)\end{pmatrix}\in\Bbb R^{n+1}$$
		at $\bm{\varepsilon}=\bm{0}$ in the direction of $\rr_+^{N}$.
	\end{theorem} 
	\begin{proof}
		The  map  ${\mathbf{\hat{\Phi}}}$ is the composition of the map $x\mapsto \begin{pmatrix}x\\\Psi\big(x)\end{pmatrix}$ after the map $\bm{\varepsilon}\mapsto x_{\bm{\varepsilon}}(T)$.  Furthermore,  the Clarke Generalized Jacobian $\partial^C\begin{pmatrix}x\\\Psi(x)\end{pmatrix}$ of the map 
		$x\mapsto \begin{pmatrix} x\\\Psi(x)\end{pmatrix}$ at  $\ol{x}({T})$ coincides with the set 
		
		$$\left\lbrace \begin{pmatrix} \mathbf{id}_{n\times n} \\ \omega \end{pmatrix}, \omega \in  \partial^C\Psi\big((\ol{x})({T})\big)\right \rbrace\subset \Lin(\Bbb R^{n}, \Bbb R^{n+1}).$$ 
		In particular      $\left\lbrace \begin{pmatrix} \mathbf{id}_{n\times n} \\ \omega \end{pmatrix}, \omega \in  \partial^C\Psi\big((\ol{x})({T})\big)\right \rbrace$	  is a QDQ of the map $x\mapsto \begin{pmatrix} x\\\Psi(x)\end{pmatrix}$ at the point $\ol{x}(T)$ in the direction of $\rr^n$.

		In view of  Proposition \ref{proppseudoaffine} and Lemma \ref{lemmapseudoaffine}, in order  	to find  a QDQ of  the map $\bm{\varepsilon}\mapsto x_{\bm{\varepsilon}}(T)$ at $\bm{0}$ {   in the direction of $\rr_+^N$}    it is sufficient to determine, for every $k=1,\ldots,N$, a QDQ at $0$ of  the function 
		${\varepsilon}_k\mapsto x_{\varepsilon_k\mathbf{e}_k}(T)$ { in the direction of $\rr_+$}.
		The latter  is itself the composition of the map  $\varepsilon_k \mapsto x^{ }_{\varepsilon_k\mathbf{e}_k}({t}_k)$ ---of which the set $v^{ }_{\mathbf{c}_k,{t}_k}$ is a QDQ at $0$ in the direction of $\rr_+$ (Proposition \ref{mainprop})---
		with the time-dependent flow of the o.d.e. $\dot x=f({x}(t)) +\displaystyle\sum_{i=1}^mg_i({x}(t))\ol{u}^i(t)$ from time ${t}_k$ to time $T$. Now,   a QDQ of this flow  { at $\hat{x}$ in the direction of $\rr^n$  }  is the set $$\left\{e^{\bigintssss_{{t}_k}^{{T}}{M}\,},\,\,\,\,\,\,\ M(\cdot) \overset{meas}\in\partial^C\Big(f(\ol{x}(\cdot)) +\sum_{i=1}^mg_i(\ol{x}(\cdot))\ol{u}^i(\cdot)\Big)\right\}.$$
		Hence, the chain rule   implies  that, for every $k=1,\ldots,N$, the subset of $\Lin\left(\Bbb R , \Bbb R^{n+1}\right)$ $$\left\{e^{\bigintssss_{{t}_k}^{{T}}{M}\,} \cdot V_k ,\,\,\,\,\, M(\cdot) \overset{meas}\in\partial^C\Big(f(\ol{x}(\cdot)) +\sum_{i=1}^mg_i(\ol{x}(\cdot))\ol{u}^i(\cdot)\Big) \, \,\,\text{and} \,\,\,	 V_k\in  v_{\mathbf{c}_k,{t}_k} \right\} $$ is a QDQ at $0$ { in the direction of $\rr_+$} of  the function  $\varepsilon_k \mapsto \mathbf{\Phi}(\varepsilon_k\mathbf{e}_k)$.  
		Then  Lemma \ref{lemmapseudoaffine} and Proposition \ref{proppseudoaffine} imply  that the subset of $ \Lin\left(\Bbb R^N , \Bbb R^{n+1} \right):$ \\$$\!\left\lbrace 
		\left(e^{\bigintssss_{{t}_1}^{{T}}{M}\,}  \cdot V_1 ,\dots,e^{\bigintssss_{{t}_N}^{{T}}{M}\,}\cdot V_N \right)  , \quad 
		\begin{array}{l} M(\cdot) \overset{meas}\in\partial^C\Big(f(\ol{x}(\cdot)) +\displaystyle\sum_{i=1}^mg_i(\ol{x}(\cdot))\ol{u}^i(\cdot)\Big), \\  
			V_k \in v_{\mathbf{c}_k,{s}_k}\,\,\, \forall k=1,\ldots,N \end{array}
		\right \rbrace$$
		is  a QDQ for $\bm{\varepsilon}\mapsto x_{\bm{\varepsilon}}(T)$.
		Finally, just from the definition of $\mathcal{E}_k(\omega,M)$,  we obtain that $\hat\Lambda \!$ is  a QDQ for $\bm{\varepsilon}\mapsto x_{\bm{\varepsilon}}(T)$ { at in the direction of $\rr_+^N$}.
	\end{proof}
	\subsection{Finitely many variations}
	\begin{lem} \label{TeoremaPrincipaleFinitelymanyvar}
		Let the hypotheses on the vector fields and on the cost as in Theorem \ref{TeoremaPrincipale}, and 	let $(\ol{u},\ol{x})$ be a  local $L^1$ minimizer  such that $\ol{u}$ takes values in the interior of $U$. Consider $N$ variation generators $\mathbf{c}_1,\ldots,\mathbf{c}_N \in \mathfrak{V}$, respectively  at $N$ instants $ 0<{t}_{1} <  \ldots <{t}_N\leq T$.
		Let $ \mathfrak{C}$ be a QDQ-approximating multicone to the target set $\mathfrak{T}$ at $x(T)$, and 
		let $H=H(x,p,u)$ be the Hamiltonian defined as
		$$H(x,p,u):=p\cdot\Big(f(x)+\sum\limits_{i=1}^m g_i(x)u^i\Big)$$
		Then there exist multipliers $(p,\lambda) \in AC\Big([0,T],(\mathbb{R}^n)^*\Big) \times \mathbb{R}^*$,\footnote{If $V$ is a real vector space, we use $V^*$ to denote its algebraic dual space.} with $\lambda\geq 0$, such that the following conditions are satisfied:\begin{itemize}
			\item[\rm\bf i)]{\sc(non triviality)}\,\,\, $(p,\lambda)\neq 0.$ \,\,\,\,		
			\item[\rm\bf ii)]{\sc (adjoint differential inclusion)} $$\frac{dp}{dt}(t)\in - \partial_x H(\ol{x}(t),p(t),\ol{u}(t))\qquad \forall t\in [0,T]\,\, a.e. $$
			\item[\rm\bf iii)]{\sc(transversality)} $$p(T)\in -\lambda \partial\Psi\Big(\ol{x}(T)\Big)-\ol{\bigcup_{\mathcal{T}\in  \mathfrak{C}}\mathcal{T}^\perp}\,\,.$$ 
			\item[\rm\bf iv)]{\sc (first order maximization)}\,\,\, If, for some $k=1,\ldots,N$, $\mathbf{c}_k=(u_k) \in \mathfrak{V}_{ndl}$, then
			$$H(\ol{x}({t}_k),p({t}_k),u_k) \le H(\ol{x}({t}_k),p({t}_k),\ol{u}({t}_k)).
			$$
			\item[\rm\bf v)]{\sc(nonsmooth Legendre\textendash{}Clebsch of step $2$)} If, for some $k=1,\ldots,N$ and $i=1,\ldots,m$, \\ $\mathbf{c}_k=(i,0) \in \mathfrak{V}_{LC2}$ then
			\bel{goh1}
			{0\in p({t}_k)\,\cdot [f,g_i]_{\mathrm{set}}(\ol{x}({t}_k)) }.
			\eeq
			\item[\rm\bf vi)]{\sc(nonsmooth Goh condition)
			} If, for some $k=1,\ldots,N$, $\mathbf{c}_k=(j,i) \in \mathfrak{V}_{Goh}$ then
			\bel{goh2}
			{0\in p({t}_k)\,\cdot [g_j,g_i]_{\mathrm{set}}(\ol{x}({t}_k)) }.
			\eeq
			\item[\rm\bf vii)]{\sc(nonsmooth Legendre\textendash{}Clebsch condition of step 3)} If $m=1$, $g:=g_1$, $\mathbf{c}_k = \mathfrak{V}_{LC3}$, and $f,g$ are of class $C^{1,1}$ near $\ol{x}({t}_k)$,  then
			\bel{LC}
			{0\ge \min \Big\{ p({t}_k)\,\cdot [g,[f,g]]_{\mathrm{set}}(\ol{x}({t}_k)) \Big\} }.
			\eeq
			
		\end{itemize}
	\end{lem}
	\begin{proof}
		For some $\delta>0$, let us  consider the {\it $\delta$\textendash{}reachable set $\mathcal{R}_\delta$} $\subset \Bbb R^{n+1}$, defined in the following way: $$\mathcal{R}_\delta=:\left\{{\begin{pmatrix} x(T)\\\Psi(x(T))\end{pmatrix}},\text{ s.t.} \left(u,x\right) \text{is   s.t.}   \left\|x
		-\ol{ x}
		\right\|_{C^0}+\|u-\ol{u}\|_{1}<\delta\right\}.$$ In view of  Theorem \ref{qdqth}, $\hat\Lambda$ is  a QDQ at $0$ for the function $\hat{\mathbf{\Phi}}$, in the direction of $[0,+\infty[^N$,  so that the family
		$$
		\hat\Lambda \cdot [0,+\infty[^N = \big\{ L\cdot  [0,+\infty[^N, \quad L\in \hat\Lambda \big\}
		$$ is  a QDQ\textendash{}approximating multicone to $\mathcal{R}_\delta$ .\\
		On the other hand,   $ \mathfrak{C}\times ]-\infty,0]$ is easily seen to be a QDQ-approximating multicone to the {\it profitable set} $$\mathcal{P}:=\Big(\mathfrak{T}\,\,\times\,\, \big]-\infty,\Psi\big(\ol{x}(T)\big)\big[\Big) \bigcup \left\{\begin{pmatrix}\ol{x}(T)\\\Psi\big(\ol{x}(T)\big)\end{pmatrix}
		\right\}\subset \Bbb R^{n+1}. \footnotemark$$
		\footnotetext{ We use the name {\it profitable set} because this set is made of points which at the same time are admissible and have a cost which is less than or equal to the optimal cost.}
		By the  definition of  local $L^1$ minimizer, the sets $\mathcal{R}_\delta$  and $\mathcal{P}$ are locally separated at { { $\begin{pmatrix}\ol{x}(T)\\\Psi\big(\ol{x}(T)\big)\end{pmatrix}$}}, so that the approximating multicones $\hat\Lambda \cdot [0,+\infty[^N$ and $ \mathfrak{C}\times  ]-\infty,0]$ are not strongly transversal. {  Moreover, every cone $C\in  \mathfrak{C}\times ]-\infty,0]$ is such that $\mu:=(0_{1+n},1)\in C^\perp$ while  $-\mu\notin C^\perp$,  } so by Lemma \ref{nontrasversalitaforte} there exist two cones, one from each of the above  multicones, that are linearly separable. In other words, there exist $ L \in \hat\Lambda$, $\mathcal{T} \in  \mathfrak{C}$, and $ (\xi,\xi_c) \in \Big(L\rr_+^N\Big)^\perp\backslash{0}$ verifying $\xi \in -\mathcal{T}^\perp$ and $\xi_c\le 0$.\\
		We can rephrase this fact by saying that there exist a measurable selection \\ $\displaystyle M(\cdot)\overset{meas}\in \partial_y \Big(f(\ol{x}(t)) +\sum_{i=1}^mg_i(\ol{x}(t))\ol{u}^i(t)\Big) $ 
		and $$V_k \in v_{\mathbf{c}_k,{t}_k},\quad \forall  k=1,\ldots,N,\qquad \omega\in \partial\Psi \Big(\ol{x}(T)\Big)$$such that 
		\begin{equation}\ds\label{mp1}\ds\xi  e^{\int^{T}_{\ol{t}_k} M}  V_k+\xi_c\omega e^{\int_{\ol{t}_k}^{T} M}    V_k\le 0. \end{equation}
		
		For all $t\in [0,T]$, let us set   $p(t):=\left(\xi-\lambda \omega\right)e^{\int_t^{T} M}$
		and notice that $p(\cdot)$ satisfies the adjoint differential inclusion in $\bf ii)$, namely $$ \frac{dp}{dt}(t) =-p(t)  M(t)\in -\partial_x H(\ol{x}(t),p,\ol{u}(t))$$ as well as  the  condition $\bf iii)$, in that 
		$$p(T) = -\lambda\omega + \xi   \in -\lambda \partial\Psi\Big(\ol{x}(T)\Big)-\ol{\bigcup_{\mathcal{T}\in  \mathfrak{C}}\mathcal{T}^\perp}\,\,.$$
		
		Hence $\bf i)$, $\bf ii)$ and $\bf iii)$ are satisfied by $(p,\lambda)$.
		Furthermore,  inequality \eqref{mp1} can be written as  \begin{equation}\label{prop1} p({t}_k)  V_k \le 0. \end{equation}Specializing \eqref{prop1} to variation generators in $\mathfrak{V}_{ndl}$ and $\mathfrak{V}_{LC3}$  one obtains $\bf  iv)$ and $ \bf vii)$, respectively. Instead, in the case of a variation   $(0,i)\in\mathfrak{V}_{LC2}$ one gets ${0\ge \min \Big\{ p({t}_k)\,\cdot[f,g_i]_{\mathrm{set}}(\ol{x}({t}_k)) \Big\} }$.\\ On the other hand, considering  $-[f,g_i]_{\mathrm{set}}$  one obtains 
		$0\ge \min \Big\{ -p({t}_k)\,\cdot[f,g_i]_{\mathrm{set}}(\ol{x}({t}_k)) \Big\} $,\linebreak i.e.   ${0\le \max \Big\{ p({t}_k)\,\cdot[f,g_i]_{\mathrm{set}}(\ol{x}({t}_k)) \Big\} }$. \\ By the convexity  of $[f,g_i]_{\mathrm{set}}(\ol{x}(\ol{t}_k))$  one then gets 
		$0\in p({t}_k)\,\cdot[f,g_i]_{\mathrm{set}}(\ol{x}({t}_k))$, i.e. point {\bf v)} of the statement.  Similarly, from the antisymmetry of the bracket $[g_j,g_i]_{\mathrm{set}}$, using variation generators $(i,j)$ and $(j,i)$ one can conclude that $0\in p(\ol{t}_k)\,\cdot[g_j,g_i]_{\mathrm{set}}(\ol{x}(\ol{t}_k))$, i.e. point {\bf vi)} of the statement.\end{proof}

	%
	%
	
	{ 
		\subsection{Infinitely many variations} To complete  the proof  of Theorem \ref{TeoremaPrincipale}, we now combine a standard procedure,  based on Cantor's non\textendash{}empty  intersection theorem,  with the crucial fact that the set\textendash{}valued brackets are convex\textendash{}valued.
		
		Begin with observing that Lusin's Theorem implies that  there exists a sequence of subsets $E_q\subset [0,\bar T]$, $q\geq 0$, such that
		$E_0$ has null measure,
		for every $q>0$  $E_q$ is  a compact set such that the restriction to $E_q$ of the map
		$t\mapsto\left(\ f(\ol{x}(t))+\sum\limits_{i=1}^m g_i(\ol{x}(t))\ol{u}^i(t) \right)$ is continuous, and
		$(0,T)_{\mathrm{Leb}}=\displaystyle\bigcup\limits_{q=0}^{+\infty} E_q.$
		For every $q>0$ let us use $D_q\subseteq E_q$  to denote  the set of all density points of $E_q$\footnote{A point $x$ is called a {\it density poin}t for a Lebesgue\textendash{}measurable set $E$ if $\ds\lim\limits_{\rho \to 0} \frac{|B_\rho(x)\cap E|}{|B_\rho(x)|}=1$}, which,  by  Lebesgue's Theorem  has the same Lebesgue measure as $E_q$. In particular, the subset  $D:=\bigcup\limits_{q=1}^{+\infty} D_q\subset [0,T]$ has measure equal to $T$.\\
		\begin{definition}Let $X\subseteq D\times \mathfrak{V}$ be any subset of {variation\textendash{}generator} pairs. We will say that a pair of multipliers $(p,\lambda)\in \mathbb{R}^* \times AC\big([0,T],(\mathbb{R}^n)^*\big)\times \mathbb{R}_+^*$ {\rm satisfies property ${\mathcal P}_X$} if the following conditions  {\bf(1)}-{\bf(6)} are verified:
			\begin{itemize}
				\item[\bf(1)] $p$ is a solution on $[0,T]$ of  the differential inclusion $$\dot{p}\in -p \,\, \partial_x^C\left(f(\ol{x})+\sum\limits_{i=1}^m g_i(\ol{x})\ol{u}^i\right).$$
				\item[\bf(2)] One has  $$ p(T)\in -\lambda \partial^C\Psi \big(\ol{x}(T)\big)  -\ol{\bigcup\limits_{\mathcal{T}\in  \mathfrak{C}} \mathcal{T}^\perp}   .$$
				\item[\bf(3)] For every $(t,\mathbf{c}) \in X$, if $\mathbf{c}=u\in\mathfrak{V}_{ndl}$, then $$  p(t)\,\left( {f(\ol{x}(t))} + \sum\limits_{i=1}^m g_i(\ol{x}(t)){u}^i\right)\le   
				p(t)\,\left(f(\ol{x}(t))+\sum\limits_{i=1}^m g_i(\ol{x}(t))\ol{u}^i(t)\right).
				$$
				\item[\bf(4)]  For every $(t,\mathbf{c}) \in X$ and if $\mathbf{c}=(0,i) \in \mathfrak{V}_{LC2}$, then $$\displaystyle\min\limits_{\qquad V \in [f,g_i]_{\mathrm{set}}(\ol{x}(t))} {p}(t)  V\le 0,$$
				{ and, taking $-[f,g_i]_{\mathrm{set}}$ instead of $[f,g_i]_{\mathrm{set}}$ ,  $$\displaystyle\max\limits_{\qquad V \in [f,g_i]_{\mathrm{set}}(\ol{x}(t))} {p}(t)  V\ge 0$$  which implies $$0\in {p}(t)[f,g_i]_{\mathrm{set}}(\ol{x}(t)).   $$

					\item[\bf(5)] For every $(t,\mathbf{c}) \in X$ and if $\mathbf{c}=(j,i)\in \mathfrak{V}_{Goh}$, then $$\displaystyle\min\limits_{\qquad V \in [g_j,g_i]_{\mathrm{set}}(\ol{x}(t))} {p}(t)  V\le 0.$$
					\item[\bf(6)] For every $(t,\mathbf{c}) \in X$, if $m=1$, $g:=g_1$, $\mathbf{c} = \mathfrak{V}_{LC3}$, and $f,g$ are of class $C^{1,1}$ near $\ol{x}(t)$,  then
					$$
					{0\ge \min \Big\{ p(t)\,\cdot [g,[f,g]]_{\mathrm{set}}(\ol{x}(t)) \Big\} }.
					$$}
			\end{itemize}

			Finally, let us  define the subset $\Theta(X)\subset AC\big([0,T],(\mathbb{R}^n)^*\big)\times \mathbb{R}_+^*$ as $$\Theta(X):=\left\{\begin{aligned} & (p,\lambda) \in \mathbb{R}^* \times AC\big([0,T],(\mathbb{R}^n)^*\big)\times \mathbb{R}_+^*\,:\, |(p(T),\lambda)|=1,\,\\ & (p,\lambda) \text{ verifies the property } {\mathcal P}_X\end{aligned}\right\}.\footnotemark$$
		\end{definition}
		\footnotetext{The norm inside the parentheses is the operator norm of $\left(\Bbb R\times{\Bbb R}^n\times\Bbb R\right)^*$.}
		{
			Clearly Theorem \ref{TeoremaPrincipale} is proved as soon as we are able to show that set $\Theta(D\times \mathfrak{V})$ is non\textendash{}empty.  
			
			We begin with proving the following fact.
			
			{\bf Claim 1} {The set \it $\Theta(X)$ is non\textendash{}empty as soon as  $X$ is finite.} 
			\begin{proof} [Proof of Claim 1]
				Indeed, by \eqref{mp1}-\eqref{prop1} we already know that  that $\Theta(X)\neq\emptyset$  whenever $X$ comprises $N$ couples $(t_k,\mathbf{c}_k)\in D\times \mathfrak{V}$ such that $t_k<s_l$ whenever $1\le k<l\le N$.
				We have to show  that we can allow  $X$ to have the general  form 
				$$X=\Big\{ (t_k,\mathbf{c}_k)\in D\times \mathfrak{V},  \quad t_k\leq t_l\,\,\text{as soon as}\,\, 1\le k<l\le N\Big\}.$$  For any $(k,r)\in\{1,\ldots,N\}\times\mathbb{N}$, choose an instant   $t_{k,r}\in E:=\bigcup\limits_{q=1}^\infty E_q$  in a way such that $t_{1,r}<\ldots<t_{N,r}$ and the sequences  $(t_{k,r})_{r\in\mathbb{N}}$, $1\leq k\leq N$,  converge to $t_k$.  For every $r\in \mathbb{N}$, consider  the sets $X_r:=\{(t_{k,r},\mathbf{c}_k),\,\,k\le N\}$.
				By the previous steps we know that  $\Theta(X_r)$ is non\textendash{}empty, so that we can choose  $(p_{0,r},p_r(s),\lambda_r) \in \Theta(X_r)$.
				Then it is easy to see\footnote{See e.g. Lemma 4.11 in \cite{papermandatoadesaim}.} that there exists a sequence $(p_{0,r},p_r(s),\lambda_r) \in \Theta(X_r)$, that, modulo extracting subsequences, converges to $(p_{0},p(s),\lambda)$ such that {\bf(1)}, {\bf(2)},{\bf(3)},{\bf(5)} hold true.  So it is sufficient to prove that {\bf(4)} and {\bf(6)} hold  as well for this $(p_{0},p(s),\lambda)$.
				Now, passing to the limit for $r\to\infty$, \eqref{prop1} is verified for every $k=1,\ldots,N$, therefore the Legendre\textendash{}Clebsch relation {\bf(4)} holds true.   To prove  the Goh relation {\bf(5)} one proceeds similarly (see also \cite{papermandatoadesaim}). \\
				\noindent
				Finally, let us prove {\bf(6)}. Since  for every natural number $r$  it is   $p_r(t_{k,r}) V_{k,r} \le 0$ for some \\ $V_{k,r}\in[g,[f,g]]_{\mathrm{set}}(\ol{x}(t_{k,r}))$ \\ and 
				$$[g,[f,g]]_{\mathrm{set}}(\ol{x}(t_{k,r}))= {\mathrm{co}} \Big\{ \lim_{{(x_{k,r,n},z_{k,r,n})}\to x(t_{k,r})} D [f,g](x_{k,r,n}) g(z_{k,r,n}) - Dg(z_{k,r,n})  [f,g](x_{k,r,n})  \Big\},$$  for every $r$ there exists a natural number $ N_{1,r}$ such that, provided
				$ n > N_{1,r}$  one has 
				$$p(t_{k,r}) (D [f,g](x_{k,r,n}) g(z_{k,r,n}) - Dg(z_{k,r,n})  [f,g](x_{k,r,n})) < \frac{2}{r}.$$
				Since $(D [f,g](x_{k,r,n}) g(z_{k,r,n}) - Dg(z_{k,r,n})  [f,g](x_{k,r,n}))$ is bounded  as $r$ and $n$ range in $\mathbb{N}$, we can assume, modulo extracting a subsequence, that the limit
				\begin{equation}\label{limitex}V_k:=\lim\limits_{r\to +\infty}\lim\limits_{n\to +\infty} =  D [f,g](x_{k,r,n}) g(z_{k,r,n}) - Dg(z_{k,r,n})  [f,g](x_{k,r,n})\end{equation} 
				does exist.
				
				Since, by construction, $\ds\lim_{n\to\infty}x_{k,r,n}= \ol{x}(t_{k,r})$ for any $r$ , $\ds\lim_{n\to\infty}z_{k,r,n}= \ol{x}(t_{k,r})$  and  $\ds  \lim_{r\to\infty}\ol{x}(t_{k,r}) =\ol{x}(t_k)$, we can construct a sequence  $\ds (N_{2,r})_{r\in\mathbb{N}}$ of natural numbers  such that \begin{enumerate}\item[i)]$N_{2,r}\to +\infty $, \item[ii)] $x_{k,r,N_{2,r}}\to \ol{x}(t_k)$ and $z_{k,r,N_{2,r}}\to \ol{x}(t_k)$, and \item[iii)] $x_{k,r,N_{2,r}}$ is a  point of differentiability for both $g$ and $[f,g]$ for any $r\in\mathbb{N}$.\end{enumerate} By the existence of the limit \eqref{limitex} we deduce that,  by taking, for any $r$,  a suitably large $N_{3,r}>N_{2,r}>N_{1,r}$  one has 
				$$V_k=\lim\limits_{r\to +\infty}\lim\limits_{N_{3,r}\to +\infty}  D [f,g](x_{k,r,N_{3,r}}) g(z_{k,r,N_{3,r}}) - Dg(z_{k,r,N_{3,r}})  [f,g](x_{k,r,N_{3,r}}).$$

				As $N_{3,r}>N_{2,r}$, this means $V_k \in [g,[f,g]]_{\mathrm{set}}(\ol{x}(t_k))$, in that  it is the limit of the Lie bracket $[g[f,g]]$ computed along a sequence of points that converges to $\ol{x}(t_k)$. Moreover, since $N_{3,r}>N_{1,r}$,  one has
				$$p_r(t_{k,r})\cdot (D [f,g](x_{k,r,N_{3,r}}) g(z_{k,r,N_{3,r}}) - Dg(z_{k,r,N_{3,r}})  [f,g](x_{k,r,N_{3,r}})) < \frac{2}{r}, \quad \forall k,r \in \mathbb{N}.$$  By passing to the limit as $r$ goes to infinity, we get  that $$p(t_k) \, V_k\le 0\qquad \Big(\text{ and } V_k \in [g,[f,g]]_{\mathrm{set}}(\ol{x}(t_k))\Big),$$  which concludes the proof of { \bf(6)} for $t_k$.
			\end{proof}
			
			To conclude the proof of Theorem  \ref{TeoremaPrincipale} 
			we are going to use a non\textendash{}empty intersection argument. Let us notice that $$\Theta(X_1\cup X_2)=\Theta(X_1)\cap \Theta(X_2),
			\quad \forall X_1,X_2\subseteq D \times \mathfrak{V},$$ so that
			\begin{equation}\label{lemint}\Theta(D\times \mathfrak{V})=\bigcap\limits_{\substack{X \subseteq D \times \mathfrak{V} \\X \text{ finite }}} \Theta(X).\end{equation}
			We have to prove that this infinite intersection is non\textendash{}empty. But this holds true because of Cantor's theorem, by virtue of the fact that    this is the intersection of a family of sets such that, as we have shown, every finite intersection  is non empty.  }
	}
	
	\section{Some possible generalizations} The main theorem of this paper can be easily adapted  to a result concerning   the following  more general issues. To begin with, instead of considering the hypothesis $\ol{u}(t)\in Int(U)$ for every  $t\in [0,T]$ one might  assume that  it holds just for every  $t\in \bigcup_{k=1}^N I_k$, where $I_1,\ldots, I_N$ of $[0,T]$ is  a finite sequence of subintervals of $[0,T]$. Secondly, one could   let  the final time $T$ to be free and consider  both a cost and a target shape depending  on it.   This would give rise to the  optimal control problem 
	$$(P)_{end-time}\qquad
	\left.
	\begin{array}{l}
		\quad \ds	
		\min_{u\in\mathcal{U}} \Psi(T,x(T)),\\ \\
		\begin{cases}
			
			\displaystyle \frac{dx}{dt}=f(x(t))+\sum\limits_{i=1}^m g_i(x(t))u^i(t), \quad  a.e.  \,\, t \in [0,T],\\
			\displaystyle x(0) =\hat{x},\qquad (T,x(T)) \in \mathfrak{T},
	\end{cases} \end{array}\right.$$ 
	the  minimization being performed over the set of the {\it  feasible processes  $(T,u,x)$ over $[0,T]$}, where, for some $T>0$ and some integers $n,m$, $x$ and $u$ have domain equal to $[0,T]$ and take values in $\mathbb R^n$  and in  $U\subseteq\mathbb{R}^m$, respectively:    by  {\it process}  we  now mean  a  triple  $(T,u,x)$
	such that $\displaystyle(T,u)\in\mathcal{U}_{end-time}:=\bigcup_{T>0}\left(\{T\}\times L^\infty([0,T],U)\right)$  and 
	$x\in W^{1,1}([0,T],\mathbb{R}^n)$ is the corresponding (Carath\'eodory)  solution of the above Cauchy problem.  The subset $\mathfrak{T}\subset \mathbb{R}^{1+n}$ is called the (space\textendash{}time) {\it target}\footnote{Such a $\mathfrak{T}$ can be thought as a space target changing shape with time.}, and  a process  $(T,u,x)$ is  called   {\it feasible}  as soon as $(T,x(T))\in \mathfrak{T}$. 
	
	A feasible process  $(\overline{T},\overline{u},\overline{x})$ is a weak  {\it local minimizer} for problem ($P$)$_{end-time}$ if there exists $\delta>0$ that verifies $$\Psi(\overline{T},\overline{x}(\overline{T})) \le \Psi(T,x(T))$$
	for all feasible processes $(T,u,x)$ such that $|T-\overline{T}|+\|x-\ol{x}\|_{C^0}+ \|u-\bar u\|_1
	<\delta$. \footnote{ Here one means that $x$ and $\ol{x}$ are extended  (continuously and) constantly outside their domains, while $u$ and $\bar u$ are taken equal to zero  outside their domains.  }
	
	For this generalized problem one can prove the following result.
	
	\begin{theorem}\label{ThGenT}
		Let us assume that $f, g_1, \ldots, g_m \in C^{0,1}_{\mathrm loc} (\mathbb{R}^n,\mathbb{R})$ and $\Psi \in C^{0,1}_{\mathrm loc}(\mathbb{R}^n,\mathbb{R})$. 	Let $(\ol{T},\ol{u},\ol{x})$ be a weak    local minimizer for  problem ($P$)$_{end-time}$ such that, for a finite sequence of subintervals $I_1,\ldots, I_N$ of $[0,T]$,
		$ \ol{u}(t)\in Int(U)$ as soon as $t\in  \bigcup_{k=1}^N I_k$.
		Let $ \mathfrak{C}$ be a QDQ\textendash{}approximating multicone to the target set $\mathfrak{T}$ at $(\ol{T},x(\ol{T})$.
		Then there exist multipliers $(p_0,p,\lambda) \in \mathbb{R}^* \times AC\Big([0,\ol{S}],(\mathbb{R}^n)^*\Big) \times \mathbb{R^*} $, with $\lambda\geq 0$, such that the following facts are verified:
		
		\begin{itemize}
			\item[\rm\bf i)]{\sc(Non\textendash{}triviality)}\,\,\, $(p_0,p,\lambda)\neq 0.$ \,\,\,\,
			
			\item[\rm\bf ii)]{\sc(Adjoint differential inclusion)} $$\frac{dp}{dt}\in -  \partial^C_x H(\ol{x}(t),p_0,p,\ol{u}(t))\qquad a.e.\quad t\in [0,\ol{T}]$$
			\item[\rm\bf iii)]{\sc(Transversality)} $$(p_0,p)(\ol{T})\in -\lambda  \partial^C\Psi\Big(\ol{T},\ol{x}(\ol{T})\Big)-\ol{\bigcup_{\mathcal{T}\in  \mathfrak{C}}\mathcal{T}^\perp}  .\,\,$$

			\item[\rm\bf iv)]{\sc(Hamiltonian's maximization)}\,\,\,
			$$\max_{u \in U} H(\ol{x}(t),p_0,p(t),u)=H(\ol{x}(t),p_0,p(t),\ol{u}(t)) \qquad a.e.\quad t\in [0,\ol{T}]\,\,\,a.e. \,\,
			$$
		\end{itemize}
		\,\,
		{Moreover, the following higher order conditions are verified:}
		\begin{itemize}
			\item[\rm\bf v)]{\sc(Nonsmooth Goh condition)}\\ 
			\bel{goh2}
			{0\in p(t)\,\cdot [g_j,g_i]_{\mathrm{set}}(\ol{x}(t))  \qquad i,j \in \{1,\ldots,m\}  }
			\qquad a.e.\quad t\in  \bigcup_{k=1}^N I_k 	\eeq 
			\item[\rm\bf vi)]{\sc(Nonsmooth Legendre\textendash{}Clebsch  condition of step $2$)}\\ 
			\bel{goh1}
			{0\in p(t)\,\cdot [f,g_i]_{\mathrm{set}}(\ol{x}(t))  \qquad i \in \{1,\ldots,m\} } \qquad a.e.\quad t\in  \bigcup_{k=1}^N I_k  
			\eeq

			\item[\rm\bf vii)]{\sc(Nonsmooth Legendre\textendash{}Clebsch condition of step $3$ with $m=1$)}\\  If $f,g_1 \in C^{1,1}(\mathbb{R}^n,\mathbb{R})$, for almost all $t \in \bigcup_{k=1}^N I_k$ such that $f$ and  $g:=g_1$ are of class $C^{1,1}$ around $\ol{x}(t)$, one has 
			\bel{LC3}
			{0\ge \min \Big\{ p(t)\,\cdot [g,[f,g]]_{\mathrm{set}}(\ol{x}(t)) \Big\}. }
			\eeq
			
		\end{itemize}
	\end{theorem}
	
	To obtain the proof  of Theorem \ref{ThGenT} one begins with considering the reparametrized  problem 
	$$(\hat P)\qquad \left\{
	\left.
	\begin{array}{l}
		\quad \ds	
		\min \Psi(x^0(\ol{T}),x(\ol{T}))\\ \\
		\begin{cases}
			\displaystyle \frac{dx^0}{dt}=1+\zeta, \\
			\displaystyle \frac{dx}{dt}=[f(x(t))+\sum\limits_{i=1}^m g_i(x(t))u^i(t)](1+\zeta), \quad a.e.\quad t\in [0,T]\\
			\displaystyle (x^0,x)(0) =(0,\hat{x}),\qquad (x^0(\ol{T}),x(\ol{T})) \in \mathfrak{T},
	\end{cases} \end{array}\right.\right.$$
	on the fixed time interval  $[0,\ol{T}]$. The additional {\it time\textendash{}reparameterizing} control $\xi$ ranges on the interval $[-\rho,\rho]$ for some   $\rho> 0$ and the control maps $(\xi,u)$ are taken in the set $\hat{\mathcal{U}}:=L^\infty\Big([0,\ol{T}],[-\rho,\rho]\times U \Big)$.\footnote{Notice that $x^0(t)=t, \forall t\in [0,\ol{T}]$ as soon as $\xi\equiv 0$.}
	
	%
	\begin{definition}\label{weakmin}
		We say that the control\textendash{}trajectory pair $((0,\overline{u}),(\overline{x}^0,\overline{x}))$ is  a  weak local minimizer of problem $(\hat P)$ if there exists $\delta>0$ such that $$\Psi(\overline{T},\overline{x}(\overline{T})) \le \Psi(\overline{T},x(\overline{T}))$$
		for all feasible processes $((\xi,u),(x^0,x))$ such that $\|(x^0,x)-(0,\ol{x})\|_{C^0} + \|(\xi,u)-(0,\bar u)\|_1
		<\delta.$ \end{definition}
	
	By means of a well\textendash{}known argument involving the time\textendash{}invariant form of the control system  in ($\hat P$) one can show   that the fact that $(\ol{T},\ol{u},\ol{x})$ is a $L^1$\textendash{}local minimizer in the original problem ($P$)$_{end-time}$ implies that the control\textendash{}trajectory pair  $\Big((0,\ol{u}),(\ol{x}^0,\ol{x})\Big)$ is a $L^1$\textendash{}local minimizer of the fixed\textendash{}time problem $(\hat P)$ (in the sense of Definition \ref{weakmin}).
	So, in the case of $N=1$ and $I_1 = [0,\ol{T}]$ Theorem  \ref{ThGenT}  is a direct consequence of Theorem \ref{TeoremaPrincipale}. It remains only to prove Theorem \ref{ThGenT} in the special case when the end\textendash{}point is fixed. That is, we only have to modify the proof of Theorem \ref{TeoremaPrincipale} when $N>1$ and $meas\big(\bigcup_{k=1}^N I_k\Big)<\ol{T}$.
	This can be done without any particular difficulty. Indeed, on the one hand, since $N<\infty$, the part of Theorem\ref{TeoremaPrincipale}'s proof  concerning a finite number of variations clearly do not need any change. On the other hand,  the successive application of Cantor's Theorem  turns out to be pratically identical to the one performed to get the  thesis of Theorem  \ref{ThGenT}. 
	
	\section{Appendix}

	\subsection{The Agrachev-Gamkrelidze  formalism.}
Let us just mention how this formalism works in practice. An important fact is that functionals act on their domain {\it from the left}. For instance, points  acts on maps by evaluating them: in particular,  a time-dependent vector field that in standard notation is denoted as  $(t,x)\mapsto F(t,x)$, in AGF is written as  $(t,x)\mapsto x F(t)$, so that a differential equation which in standard  notation would be written as $\dot x(t) = F(t,x(t))$, in the AGF reads $\dot x(t) = x(t)F(t)$. On the the hand, a vector field $X:\mathcal M\to T\mathcal M$ acts on real functions $\varphi:\mathcal M\to \mathbb R$ by mapping them on their directional derivative, so we write the latter $X\varphi$. In other words what in standard notation is written as $x\to X\cdot D\varphi(x)$ in AGF reads $x\to x\varphi X$. More in general, if $x\in\mathcal M$, $X,Y$ are vector fields and $\varphi, \psi$ are real functions one writes $ x\varphi XY \psi$ what in standard notation would be written as $D\psi(DY X)D\varphi(x)$. Notice in particular that the Lie bracket $[X,Y]$ of two vector fields $X,Y$, that in standard notation is defined as $[X,Y]=DYX-DXY$, in AGF reads $XY-YX$. 
Similarly to what is done with standard notation, the (assumed unique)  solution $x(\cdot)$ on some interval $[0,T]$  of a Cauchy problem  $\dot x(t) =  x(t)F(t), x(0)=\tilde x$ will be denoted in exponential form by $\displaystyle x(t)= \tilde x e^{\int_0^t F }$.
In the particular case of an autonomous vector field $x\mapsto x F(t)\equiv xF$,  we will use the notation $\tau \mapsto \tilde x e^{\tau F}$ in place of $\tau \mapsto  x e^{\int_0^\tau F}$. Let us point out that under AGF  the above exponentials formally behave like real exponential maps: indeed  $\frac{d}{d\tau} x e^{\int_0^\tau F} = \dot x = 
x F e^{\int_0^\tau }$.
\subsection{A family of optimal control problems}
There 
$$(P)\qquad\left\{
\left.
\begin{array}{l}
	\quad \ds	
	\min\Psi(x(T)),\\ \\
	
	\begin{cases}
		
		\displaystyle \frac{dx}{dt}=f(x(t))+\sum\limits_{i=1}^m g_i(x(t))u^i(t), \quad  a.e. \,\,\, t \in [0,T],\\
		\displaystyle x(0) =\hat{x},\qquad x(T) \in \mathfrak{T}.
\end{cases} \end{array}\right.\right.$$ 
where, for some integers $n,m$, \begin{itemize}
	\item  the state $x$ takes values in $\mathbb R^n$;
	the vector fields $f,g_1,\ldots,g_m$ as well as the {\it cost} $\Psi$ are assumed to be of class $C^{1}$;
	\item the control  $u$  takes values in  a  subset  $U\subseteq\mathbb{R}^m$ such that $intU \neq \emptyset$;
	\item the subset $\mathfrak{T}\subset \mathbb{R}^{n}$ is called the  {\it target};
	\item  by {\it  process}  we mean  a  pair  $(u,x)$
	such that $\displaystyle u\in L^\infty([0,T],U)$
	and 
	$x\in W^{1,1}([0,T],\mathbb{R}^n)$ is the corresponding  solution of the above Cauchy problem; a process  $(u,x)$ is  called   {\it feasible}  as soon as $x(T)\in \mathfrak{T}$; 
	\item 	the  minimization is performed over the set of the {\it feasible processes} $(u,x)$.
\end{itemize}

\section{Crucial tools to an application to an optimal control problem}
\subsection{Goh and Legendre Clebsh condition involving Lie Brackets}
We have obtained Goh and Legendre\textendash{}Clebsch type conditions for a a generalization of the problems of the family of $(P)$.In particular we permits to the vector fields $(\Psi,f_0,f_1,\cdot, f_m)$ to not be in $C^1$ but to be only locally Lipschitz continuous.\\ 
In particular we obtained the following theorem: \\
\begin{theorem}[{A non-smooth Maximum Principle with Goh and Legendre\textendash{}Clebsch conditions}] \label{TeoremaPrincipale}
	Let us assume that $f, g_1, \ldots, g_m \in C^{0,1}_{\mathrm loc} (\mathbb{R}^n,\mathbb{R})$ and $\Psi \in C^{0,1}_{\mathrm loc}(\mathbb{R}^n,\mathbb{R})$.	Let $(T,\ol{u},\ol{x})$ be a non singular,   local $L^1$ minimizer for our problem.
	Then there exist multipliers $(p,\lambda) \in AC\Big([0,\ol{S}],(\mathbb{R}^n)^*\Big) \times \mathbb{R^*} $, with $\lambda\geq 0$, which satisfy the Pontryagin Maximum Principle and also the following higher order conditions: 
	\begin{itemize}
		\item[\rm\bf i)]{\sc(Nonsmooth Goh condition)}\\ 
		\bel{goh22}
		{0\in p(t)\,\cdot [g_j,g_i]_{\mathrm{set}}(\ol{x}(t))  \qquad i,j \in \{1,\ldots,m\}  }
		\qquad a.e.\quad t\in I	\eeq 
		\item[\rm\bf ii)]{\sc(Nonsmooth Legendre\textendash{}Clebsch  condition of step $2$)}\\ 
		\bel{goh12}
		{0\in p(t)\,\cdot [f,g_i]_{\mathrm{set}}(\ol{x}(t))  \qquad i \in \{1,\ldots,m\} } \qquad a.e.\quad t\in I
		\eeq
		\item[\rm\bf iii)]{\sc(Nonsmooth Legendre\textendash{}Clebsch condition of step $3$ with $m=1$)}\\ 
		If $m=1$, for almost all $t \in [0,T]$ such that $f$ and  $g:=g_1$ are of class $C^{1,1}$ around $\ol{x}(t)$, then the adjoint  differential inclusion of point {\bf ii)} reduces to the usual adjoint differential  equation and
		\bel{LC32}
		{0\ge \min \Big\{ p(t)\,\cdot [g,[f,g]]_{\mathrm{set}}(\ol{x}(t)) \Big\}. }
		\eeq
		where as $[\cdot,\cdot]_{set}$ a generalization of classical Lie Brackets for Lipschitz continuous vector field introduced in \cite{RS1}
	\end{itemize}
\end{theorem}
\begin{remark}
	An important ingredient to prove the abovementionated result was the validity of the following well known asymptotic formula: \\
	\begin{equation}
		\Psi_{[f_1,f_2]}(t_1,t_2)(x)=x + t_1 t_2 [f_1,f_2](x) + o(t_1,t_2)
	\end{equation}
	
	
	
	where $f_1, f_2$ are $C^1$ vector fields on an open subset of an Euclidean
	space $[f1, f2]$ is their Lie bracket, $t_1$ and $t_2$ small enough and, for definition of multiflow, $\Psi_{[f_1,f_2]}(t_1,t_2)=\exp(-t_2 f_2) \circ \exp(t_1 f_1)$.
\end{remark}
Let us now introduce a technical notation: \\
\begin{definition}[Variation generators]
	Let us define the set $\mathfrak{V}$ of {\it  variation generators} as the union
	$$\mathfrak{V}:=\mathfrak{V}_{ndl}\bigcup \mathfrak{V}_{Goh} \bigcup \mathfrak{V}_{LC2} \bigcup \mathfrak{V}_{LC3},$$
	where
	$\mathfrak{V}_{ndl}:=  U ,$ 
	$\mathfrak{V}_{Goh} :=\Bigl\{(i,j)\,\,\,i\neq j\,\,\,\, i,j=1,\ldots,m\},$
	$\mathfrak{V}_{LC2}=:\{(0,i)\,\,, i=1,\ldots,m  \},   $
	and $\mathfrak{V}_{LC3}=:\{(1,0,1)\}$.
\end{definition}

\begin{theorem} \label{eccolestime} Let  $(u,x)$ be a process, and 
	let the control $u \in U$ .
	Choose   $\ol{t}, \varepsilon$ such that  the vector fields  $f$ and $g_i$  are  of class $C^1$ around $x(\ol{t})$ and, moreover,  
	$0<\sqrt[3]{\varepsilon}<\min\{\ol{t},1\}$. 
	Then there exists $M>0$ such that, for any
	any  $\mathbf{c} \in \mathfrak{V} $, using $x_\varepsilon$ to denote the trajectory corresponding to the perturbed control  
	${u}_{\varepsilon,\mathbf{c},\ol{t}}$ we get the following three statements.
	
	\begin{enumerate}	\item If  $\mathbf{c}=(0,i) \in \mathfrak{V}_{LC2}$, for some $i=1,\ldots,m$,
		then
		\begin{equation}\label{1} x_\varepsilon(\ol{t})-x(\ol{t})=M\varepsilon[f,g_i](x(\ol{t}))+o(\varepsilon)\,\,
		\end{equation}
		\item If $\mathbf{c}=(j,i) \in \mathfrak{V}_{Goh}$  (i.e. $j=1,\ldots,m$, $1<i\leq m$) 
		then \begin{equation}\label{2}
			x_\varepsilon(\ol{t})-x(\ol{t})=M\varepsilon[g_j,g_i](x(\ol{t}))+o(\varepsilon)\,\,\end{equation}
		\item If $m=1$, $\mathbf{c} = \mathfrak{V}_{LC3}$ and, moreover, $f,g (:=g_1)$ are of class $C^2$ around ${x(\ol{t})}$,
		then \begin{equation}\label{3} x_\varepsilon(\ol{t})-x(\ol{t})=M\varepsilon[g,[f,g]](x(\ol{t}))+o(\varepsilon)\,\,\end{equation}
	\end{enumerate} 
	
\end{theorem}
\begin{remark}
	Clearly this is a "smooth" result, but we are able to apply it at our non-smooth situation because in \cite{AngrisaniRampazzoQDQ} a result of approximation of $[\cdot,\cdot]$ by some smooth Lie Brackets is proved. \\
\end{remark}

We will now prove the Theorem \ref{eccolestime} and in particular we find explicitly a family of perturbed example of controls ${u_{\varepsilon}}$ such that \ref{1}, \ref{2} and \ref{3} hold true. \\
\subsection{About non-commutativity}
Let us begin with a trivial remark:   if  $X$ is a given Lipschitz continuous vector field and $b\in L^1([0,T], \rr)$, the solution
$\tilde x e^{\int_0^t  bX}:[0,T]\to\Bbb R^n$  of the  Cauchy problem  $\dot{x}(t)=x(t)\,\,b(t)X$, $x(0)=\hat{x}$ , $t\in [0,T]$, verifies 
$$  \tilde x e^{\int_0^t \left( bX\right)}= \tilde{x}e^{B(t)X}\quad \forall t\in [0,T],$$
where  we have set $B(t) = \ds \int_0^t b(s)ds, \,\forall t\in[0,T]$.  In other words, the value $x(t)$ of the solution at time $t$ coincides with  the value $y(\tau)$  of the solution of the autonomous problem $\dot y(\tau) = y(\tau) X$ at time $\tau = B(t)$. 

An estimate of the error caused by  non-commutativity is given by the following elementary  result:
\begin{lem}\label{lemmasemplificativo}
	Let $q$ be any natural number, let $c>0$ be a positive constant and let $b_i(s):[0,T]\to \rr, \quad i=0,\ldots,q$ be $q+1$ arbitrary  $L^1$ functions such that $b_i(s)\le cs$ almost everywhere. Let  us consider $q+1$ vector fields  $X_i \in C^\infty(\Omega\subseteq \rr^n;\rr^n)$ defined on an open subset $\Omega\subseteq \rr^n,\quad i=0,\ldots,q$, and  let $x(\cdot)$ be a solution on an interval $[0,T]$ of the  differential equation
	\begin{equation}\label{eqdiffcomplicatalemma}\dot{x}(t)=\sum\limits_{i=0}^q b^i(t) \cdot x(t)\,X_i \quad \footnote{In standard notation this o.d.e. would be written as $\dot{x}(t)=\sum\limits_{i=0}^q b^i(t) \, X_i\left(x(t)\right)$
		}
	\end{equation} Then, setting,  for every $t\in (0,T)$ and every $i=0,\ldots,q$, $B_i(t):=\int_0^t b^i(s) ,\ds$ we have 
	\begin{equation}\label{primo}
		x(t)=
		x(0) e^{B^{q}(t)X_{q}}\cdot\ldots \cdot e^{B^{0}(t)X_{0}}+o(t^3) \quad t\in [0,T].
	\end{equation}
\end{lem}
To prove \eqref{primo} one makes use of the following basic fact:\footnote{The proof of this classical  result is trivially obtained by direct computation.}

\begin{lem}\label{sviluppi}
	If $X$ and $Y$ are $n$ times differentiable vector fields then
	\begin{equation}\label{derivateordinealto}
		\frac{d^n}{d\tau^n} (xe^{\tau X}Ye^{-\tau X})=xe^{\tau X}\ad^n X(Y)e^{-\tau X}, \qquad \text{ for all } n\ge 1,
	\end{equation} 
	where we have used the "ad" operator defined recursively  as  $$\ad^0 X(Y)=Y\qquad \ldots\qquad \ad^q X(Y)=[X,\ad^{q-1}X(Y)],\,\,\,\forall q\in \mathbb{N}.$$
	In particular, for any  natural number $n\geq 0$,
	\begin{equation} \label{Taylor}
		xe^{\tau X}Ye^{-\tau X}=\sum\limits_{i=0}^n \left(\frac{\tau^i}{i!}\right) x\ad^i X(Y) + o(\tau^n).
	\end{equation}

\end{lem}
\begin{proof}[Proof of Lemma \ref{lemmasemplificativo}]
	Let us define $x_1(s):=x(s)e^{-B^0(s)X_0}$ so that $x(s)=x_1(s)e^{B^0(s)X_0}.$ \\ Differentiating both sides with respect to $s$ we get, for almost every $s\in (0,T)$, $$\dot{x}(s)=\dot{x}_1(s)e^{B^0(s)X_0}+b^0(s)\cdot x_1(s)e^{B^0(s)X_0}\,X_0=\dot{x}_1(s)e^{B^0(s)X_0} + b^0(s)\cdot x(s) X_0 .$$
	
	Solving for $\dot{x}_1(s)$ and remembering that $x(\cdot)$ is a solution to Equation \eqref{eqdiffcomplicatalemma}, we find (for almost every $s\in (0,T)$) \bel{x1}\begin{array}{l}\ds\dot{x}_1(s)=\dot{x}(s)-b^0(s)\cdot x(s)X_0e^{-B^0(s)X_0}=\sum\limits_{i=1}^q b^i(s)\cdot x(s)X_ie^{-B^0(s)X_0}=\\\ds =\sum\limits_{i=1}^q b^i(s) \cdot x_1(s) e^{B^0(s)X_0}X_ie^{-B^0(s)X_0}
	\end{array}\eeq
	
	\noindent 
	Now let $E\subset (0,T)$ be a full measure subset such that  the o.d.e.\eqref{eqdiffcomplicatalemma} is  verified for every $t\in E$. 
	Moreover, let $\mathcal{L} \subset [0,T]$ be the subset of Lebesgue points for every $b^i$, $i=0,\ldots,q$.
	If $t\in  E$ ---so that, in particular, \eqref{x1} is verified at $t$--- using  \eqref{Taylor}  we get
	$$\dot{x}_1(t)=\sum\limits_{i=1}^q b^i(t)\cdot x_1(t)X_i+B^0(t)\sum\limits_{i=1}^q b^i(t)\cdot x_1(t) [X_0,X_i]+o(B^0(t)) =\sum\limits_{i=1}^{q}b^i(t)\cdot x_1(t)X_i
	+o(t^2),$$ where we have also used that
	$\ds B^0(t)\le \int_0^t (cs)\,ds\le c\frac{t^2}{2}$,  so that  $o(B^0(t))=o(t^2)$ and also $B^0(t)b^i(t)\le \frac{c^2t^3}{2}=o(t^2)$.
	\noindent
	Let us now set, for all $s\in(0,T)$, $x_2(s):=x_1(s)e^{-B^1(s)X_1}$, so that $x_1(s)=x_2(s)e^{B^1(t)X_1}$. Differentiating both sides at $s=t$ we obtain,  as before, 
	$$\dot{x}_1(t)=\dot{x}_2(t)\,\,e^{B^1(t)X_1}+b^1(t)\cdot x_2(t)\,\,e^{B^1(t)X_1}X_1=\dot{x}_2(t)\,\,e^{B^1(t)X_1}+\,b^1(t)\cdot x_1(t)X_1.$$
	Solving for $\dot{x}_2(t)$ and using again  the expansion \eqref{Taylor},  we find $$\begin{array}{c}\dot{x}_2(t)=\left(\dot{x}_1(t)-b^1(t)\cdot x_1(t)X_1\right)e^{-B^1(t)X_1}= x_1(t)\Bigg(\sum\limits_{i=2}^q b^i(t)X_i\Bigg) e^{-B^1(t)X_1}+o(t^2)=\\ x_2(t)e^{B^1(t)X_1} \Bigg(\sum\limits_{i=2}^q b^i(t)X_i\Bigg) e^{-B^1(t)X_1}+o(t^2)=x_2(t)\Bigg(\sum\limits_{i=2}^q b^i(t)X_i+\sum\limits_{i=2}^q 	 B^1(t)b^i(t)[X_1,X_i]\Bigg)+o(t^2)=\\x_2(t)\left(\sum\limits_{i=2}^q b^i(t)X_i\right)+o(t^2).
	\end{array}$$
	If we keep going with the same idea and set  
	$x_{r}(t)=x_{r-1}e^{-B^{r-1}(t)X_{r-1}}$ for all $r\le q$, we can conclude that at the $r$-th step one gets
	$$\dot{x}_r(t)=x_r(t)\left(\sum\limits_{i=r}^q b^i(t)X_i\right)+o(t^2).$$
	At the $q$-th step we simply have 
	$$\dot{x}_q(t)=x_q(t)\,b^q(t)X_q+o(t^2).$$ Applying the same idea one more time we finally obtain
	$\dot{x}_{q+1}(t)=o(t^2),$  where $x_{q+1}(t):=x_q(t)e^{-B^q(t)X_q}$. 
	Integrating both sides we get (keeping in mind that $meas ((0,t)\cap E\cap \mathcal{L}) = meas\big((0,t)) =t$\,\,\big) \begin{equation}\label{semplicissima}x_{q+1}(t)=x_{q+1}(0)+o(t^3).\end{equation}
	Remembering that \begin{multline}\label{ricorsioneq}x_{q+1}(t)=x_{q}(t)e^{-B^q(t)X_q}=x_{q-1}(t)e^{-B^{q-1}(t)X_{q-1}}e^{-B^{q}(t)X_q}=\\ \, \, =x_2(t)\prod\limits_{i=2}^q e^{-B^i(t)X_i}=x_1(t)\prod\limits_{i=1}^q e^{-B^i(t)X_i}=x(t)\prod\limits_{i=0}^q e^{-B^i(t)X_i}
	\end{multline}
	by\eqref{semplicissima} we get \begin{equation}\label{semplicissimab}
		x(t)\prod\limits_{i=0}^q e^{-B^i(t)X_i}=x_{q+1}(0)+o(t^3)
	\end{equation}
	Notice that \eqref{ricorsioneq} also implies $x_{q+1}(0)=x(0)$. \\  Bringing all  exponentials on the right-hand side  we obtain
	\begin{equation*}
		x(t)=x(0)\prod\limits_{i=0}^q e^{B^{q-i}(t)X_{q-i}}+o(t^3),
	\end{equation*} 
	so estimate \eqref{primo} is proved for every $t\in  E\cap \mathcal{L}$. \\ Since this set has full measure, by continuity we conclude that the above equality holds true for all $t\in [0,T]$.
\end{proof}
\noindent

\subsection{Tools to prove main theorem}
\medspace \\
In Proposition \ref{ispiratoaSchattler1} below we dismiss the hypothesis of sublinearity of the controls, and  deduce an  estimate of the flow of $\dot x =\sum\limits_{i=0}^m a^i(s)\cdot x(s)g_i$   in terms of the  the original vector fields $g_0,\ldots,g_m$ and their  Lie Brackets.
\begin{prop}\label{ispiratoaSchattler1}  For any $i=0,\ldots,m$, consider $a^i(s)\in L^\infty([0,T];\rr)$ and 
	let $g_i$ be a  $C^{\fra 1}$ vector field.\\ Then, if we set, for all $ i=0,\ldots,m$ and 
	$i<j<m$,
	\bel{A}
	\begin{array}{c}A^i(s):=\ds\int_0^s a^i(\sigma)d\sigma \quad A^{ji}(s)  :=\ds\int_0^s A^j(\sigma)a^i(\sigma)d\sigma  \qquad s\in [0,T],\end{array}
	\eeq
	every solution $  t\mapsto x(t):=x(0) \displaystyle e^{ {\displaystyle\tiny\int_0^t}  \sum\limits_{i=0}^m a^i(s)\,\,g_i }$ to the ordinary differential equation {\rm(}on some interval $[0,\tau]${\rm)}
	\begin{equation}\dot {x}(s)= x(s)\sum\limits_{i=0}^m a^i(s)\,\,g_i\end{equation}
	verifies 
	\begin{equation}\label{Cruciale1}x(t)=x(0)\prod_{m \ge j>i\ge 0}e^{A^{ji}(t)[g_j,g_i]}\prod_{i=0}^m e^{A^{m-i}(t)g_{m-i}}+o(t^2).
	\end{equation}
\end{prop}

As a straightforward consequence, we get  
\begin{cor}\label{corollariodischattlerutile}Let us fix $t>0$ and let us choose $i^*,j^*\in \{0,\ldots,m\}$, $j^*>i^*$. If, for every $i=0,1,\ldots,m$, the $L^\infty$ controls  $a^i:[0,t]\to U $ in Proposition \ref{ispiratoaSchattler1} are chosen   such that $A^i(t)=0$ and $A^{j^*,i^*}(t)=ct^2+o(t^2)$ with $c>0$ is the only non-vanishing term in the family $\{A^{ji}(t)\}_{j>i,{  i=0,\dots,m}}$  (see below), then one has $$x(t)-x(0)=ct^2[g_{j^*},g_{i^*}](x(0))+o(t^2).$$
\end{cor}

\begin{proof}[Proof of Proposition \ref{ispiratoaSchattler1}]
	Setting $x_1(t) := x(t) e^{-A^0(t)g_0}$ and arguing as in the proof of Lemma \ref{lemmasemplificativo} (with the the $b^i(t)$ replaced by the $a^i(t)$), we get 
	\begin{multline} \label{second}\dot{x}_1(t)=x_1(t)\Bigg(\sum\limits_{i=1}^m a^i(t)g_i+A^0(t)\sum\limits_{i=1}^m a^i(t)[g_0,g_i]\Bigg)+o(A^0(t))
		=\\ =x_1(t)\Bigg(\sum\limits_{i=1}^m a^i(t)g_i+A^0(t)\sum\limits_{i=1}^m a^i(t)[g_0,g_i]\Bigg)+o(t),
	\end{multline}
	the last equality deriving from $A^0(t)\le t\|a^0\|_{\infty}$. \\
	We now define $x_2(t)=x_1(t)e^{-A^1(t)g_1}$, so that $x_1(t)=x_2(t)e^{A^1(t)g_1}$, and differentiate both sides of this last equality, obtaining
	$$\dot{x}_1(t)=\dot{x}_2(t)e^{A^1(t)g_1}+x_2(t)e^{A^1(t)g_1}a^1(t)g_1=\dot{x}_2(t)e^{A^1(t)g_1}+x_1(t)(a^1(t))g_1.$$
	Solving for $\dot{x}_2(t)$, and utilizing Taylor expansions as in \eqref{Taylor}, we obtain   $$\begin{array}{c}\dot{x}_2(t)=\Big(\dot{x}_1(t)-x_1(t)a^1(t)g_1\Big)e^{-A^1(t)g_1}= x_1(t) \Bigg(\sum\limits_{i=2}^m a^i(t)g_i+A^0(t)\sum\limits_{i=1}^m a^i(t)[g_0,g_i]\Bigg) e^{-A^1(t)g_1}+o(t)=\\ x_2(t)e^{A^1(t)g_1} \Bigg(\sum\limits_{i=2}^m a^i(t)g_i+A^0(t)\sum\limits_{i=1}^m a^i(t)[g_0,g_i]\Bigg) e^{-A^1(t)g_1}+o(t)
		
		=\\ x_2(t)\Bigg(\sum\limits_{i=2}^m a^i(t)g_i+A^0(t)\sum\limits_{i=1}^m a^i(t)[g_0,g_i] + A^1(t)\sum\limits_{i=2}^m a^i(t)[g_0,g_i]  + o(A^1(t))\Bigg) +o(t)
		
		=\\=x_2(t)\Bigg(\sum\limits_{i=2}^m a^i(t)g_i+\sum\limits_{j=0}^1 \sum\limits_{i=j+1}^m 	\left( A^j(t)a^i(t)[g_j,g_i]\right)\Bigg)+o(t)\end{array}$$
	where we have utilized the estimate  $A^i(t)\le cs$ $(i=0,\ldots,m)$.\\
	If we do this again and again, defining $x_{r}(t)=x_{r-1}e^{-A^{r-1}(t)g_{r-1}}$ for all $r\le m$, it can be proven by induction that at the $r$-th step we get
	$$\begin{array}{c}\dot{x}_r(t)=\left\{\dot{x}_{r-1}(t)-x_{r-1}(t)a^{r-1}(t)g_{r-1}\right\}e^{-A^{r-1}(t)g_{r-1}}=\\=x_{r}(t)\Bigg(\sum\limits_{i=r}^m a^i(t)g_i+\sum\limits_{j=0}^{r-1} \sum\limits_{i=j+1}^m A^j(t)a^i(t)[g_j,g_i]\Bigg) +o(t).
	\end{array}$$
	The $m$-th step the first sum is then  reduced to a single term, and with the $m+1$-th step we get 
	\bel{xm+1}\dot{x}_{m+1}(t)=x_{m+1}(t)\Bigg(\sum\limits_{ i> j \ge 0} A^j(t)a^i(t)\cdot[g_j,g_i]\Bigg) +o(t)
	\eeq
	Observe also that, by definition, $x_{m+1}(0) = x(0)$.
	Let us forget the $o(t)$ error in \eqref{xm+1}, and let us consider the  Cauchy problem (still in the AGF notation)
	\bel{xm+1.}\left\{\begin{array}{l}
		\ds\dot{y}(t)=\sum\limits_{ i> j \ge 0} A^j(t)a^i(t)\cdot y(t)[g_j,g_i] \\
		y(0)= x(0)\end{array}\right.
	\eeq

	Since the $a^i$ are  bounded  in $L^\infty$, one has  $|A^j(t)a^i(t)|\leq cs$,  for every $j>i, \, i=0,\ldots,m$ and  for  some positive constant $c$.  Therefore we are in the hypotheses of  Lemma \ref{lemmasemplificativo}  (as {soon as  one identifies the $X_k$'s and the $b^k$'s with the brackets $[g_i,g_j]$  and the scalars $A^ja^i$, respectively), which implies  $$y(t)=x(0)\prod_{m \ge j>i\ge 0}e^{A^{ji}(t)[g_j,g_i]} +o(t^3)	$$ where the functions 
		$A^{ji}$, $i,j \in  \{0,\ldots, m\}, j>i$, are as in \eqref{A}. \\ Since $\dot x_{m+1}(t) =\dot  y(t) + o(t) $ almost everywhere, we get 
		\begin{equation}\label{Quasifinito1}x_{m+1}(t) =y(t) + o(t^2) = x(0)\prod_{m \ge j>i\ge 0}e^{A^{ji}(t)[g_j,g_i]} +o(t^2). 
		\end{equation}
		%
		\\
		From $x_{m+1}(t)=x(t)\prod\limits_{i=0}^m e^{-A^i(t)g_i}$ (and  $x_{m+1}(0)=x(0)$), we get 
		$$\begin{array}{c} x(t) = x_{m+1}(t) \prod\limits_{i=0}^m e^{A^i(t)g_i} = x(0)\left(\prod_{m \ge j>i\ge 0}e^{A^{ji}(t)[g_j,g_i]} +o(t^2)\right)\prod\limits_{i=0}^m e^{A^i(t)g_i}=
			\\
			=x(0)\prod_{m \ge j>i\ge 0}e^{A^{ji}(t)[g_j,g_i]} \prod\limits_{i=0}^m e^{A^i(t)g_i}+o(t^2),
		\end{array}
		$$
		so   the proof is concluded. }
\end{proof}
\noindent
Using \eqref{Taylor} at third order and not at second as in \eqref{second}, with the same approach of the proof of \ref{corollariodischattlerutile}, one can also prove the following third order approximation:

\begin{prop} \label{ispiratoaSchattler}
	Let $a^i(s), \quad i=0,\ldots,m$ be $m+1$ arbitrary functions in $L^\infty([0,T];\rr)$.\\
	the value at $t$ of a solution $x$ to the ordinary differential equation 
	
	\begin{equation}\label{eqdiffcomplicata2}\dot{x}(s)=x(s)\sum\limits_{i=0}^m a^i(s)g_i
	\end{equation}
	verifies
	\begin{equation}\label{Cruciale}\begin{array}{l}x(t)=x(0)\prod_{m \ge j>i,k\ge 0}e^{A^{k,j,i}(t)[g_k,[g_j,g_i]]}\cdot \prod_{m \ge j>i\ge 0} e^{A^{j,j,i}(t)[g_j,[g_j,g_i]]}\,\,\,\cdot\\ \qquad\qquad\qquad\qquad\qquad\qquad\qquad\qquad\cdot\prod_{m \ge j>i\ge 0}e^{A^{ji}(t)[g_j,g_i]}\cdot\prod_{i=0}^m e^{A^{m-i}(t)g_{m-i}}+o(t^3) \end{array}
	\end{equation}
	where
	$$
	\hat{A}^{k,j,i}(t):=	\int_0^{t} \hat{A}^k(s)\hat{A}^j(s)\hat{a}^i(s) ds     
	$$
	for all $k,j,i=0,\ldots,m$ and $t\in [0,T]$.
\end{prop}

%
%
%

%
\begin{cor}\label{corollariodischattlerutile3}Let us fix $t>0$. If $m=1$, and we set $f:=g_0, g=g_1$, and the $L^\infty$ control  $a^0, a^1:[0,t]\to U $  is chosen   such that $A^0(t)=A^1(t)=A^{0,1}(t)=A^{0,0,1}(t)=0$ and $A^{1,0,1}(t)=ct^3+o(t^3)$ with $c>0$, then one has $$x(t)-x(0)=ct^3[g,[f,g]](x(0))+o(t^3).$$
\end{cor}
\section{Polynomials}
\begin{definition}[Variation generators]
	Let us define the set $\mathfrak{V}$ of {\it  variation generators} as the union
	$$\mathfrak{V}:=\mathfrak{V}_{ndl}\bigcup \mathfrak{V}_{Goh} \bigcup \mathfrak{V}_{LC2} \bigcup \mathfrak{V}_{LC3},$$
	where
	$\mathfrak{V}_{ndl}:=  U ,$ 
	$\mathfrak{V}_{Goh} :=\Bigl\{(i,j)\,\,\,i\neq j\,\,\,\, i,j=1,\ldots,m\},$
	$\mathfrak{V}_{LC2}=:\{(0,i)\,\,, i=1,\ldots,m  \},   $
	and $\mathfrak{V}_{LC3}=:\{(1,0,1)\}$.
\end{definition}

{ In order to define control variations associated to the four different  kinds of variation generators, let us consider
	{the Hilbert space $\mathcal{P}([0,1])$} of polynomials defined on $[0,1]$, endowed with the scalar product $\langle P_1,P_2\rangle=\displaystyle\int_0^1 P_1(s)P_2(s) ds$. We will be concerned with the subspace
	$\mathcal{P}^\sharp([0,1]) \subset\mathcal{P}([0,1])$ defined as 
	$$
	\mathcal{P}^\sharp([0,1]) := \left\{P\in \mathcal{P}([0,1]) : \ \ \ \ 
	0=P(0)=P(1) = \int_0^1 P(t)dt \right\}.
	$$

	Clearly,
	the vector  subspace  $\mathcal{P}^\sharp([0,1])\subset  \mathcal{P}([0,1])$ has  finite codimension.\footnote{The choice of these spaces of functions to construct control variations as well as of the polynomials introduced  in Definition \ref{iP} below finds its  motivation in the expansions of solutions mentioned at the end of the Introduction.}
}
In order to define control variations associated to the four different  kinds of variation generators, 
let us consider 
{the Hilbert space $\mathcal{P}([0,1])$} of polynomials defined on $[0,1]$, endowed with the scalar product $\langle P_1,P_2\rangle=\displaystyle\int_0^1 P_1(s)P_2(s) ds$. We will be concerned with the subspace
$\mathcal{P}^\sharp([0,1]) \subset\mathcal{P}([0,1])$ defined as 
$$
\mathcal{P}^\sharp([0,1]) := \left\{P\in \mathcal{P}([0,1]) : \ \ \ \ 
0=P(0)=P(1) = \int_0^1 P(t)dt \right\}.
$$

Clearly,
the vector  subspace  $\mathcal{P}^\sharp([0,1])\subset  \mathcal{P}([0,1])$ has  finite codimension.\footnote{The choice of these spaces of functions to construct control variations as well as of the polynomials introduced  in Definition \ref{iP} below finds its  motivation in the expansions of solutions mentioned at the end of the Introduction.}

\begin{definition}\label{iP} 
	For all   $r,i =1,\ldots,m$  and  $j=0,\ldots,m, \ \ j<i$, we select (arbitrarily)
	the polynomials $P^{r}_{(i,j)}$  
	defined as follows:	
	\begin{itemize}

		
		\item  if $j\not=0$
		we choose $P_{(j,i)}^{m}\in \mathcal{P}^\sharp([0,1])\backslash\{0\}$ arbitrarily, then
		we proceed recursively on $r$, backward from $m$ to $1$,   and  select (arbitrarily) $$\begin{array}{l}\ds P_{(j,i)}^{r}\in \mathcal{P}^\sharp([0,1]) \bigcap \left\{{dP_{(j,i)}^{m}\over dt},{dP_{(j,i)}^{m-1}\over dt},\ldots,{dP_{(j,i)}^{r+1}\over dt}\right\}^{\bigperp}\setminus\Big\{0\Big\},\quad  \forall r\in \left\{1,\ldots,m-1\right\}\setminus\left\{j\right\},\\\ds P_{(j,i)}^{j}\in \mathcal{P}^\sharp([0,1]) \bigcap \left\{{dP^{m}_{(j,i)}\over dt},{dP_{(j,i)}^{m-1}\over dt},\ldots,{dP_{(j,i)}^{j+1}\over dt}\right\}^{\bigperp}\setminus\left\{{dP_{(j,i)}^{i}\over dt}\right\}^{\bigperp}.\end{array}\qquad \footnotemark$$ 
		\footnotetext{Because of a certain sign requirement that is made clear in the proof of Theorem \ref{eccolestime} in Appendix, we might need to change the sign of one  but not both of the polynomials $\left\{P^{i},P^{ji}_{j}\right\}$.} 
		
		%
		%

		\item if $j=0$, we choose $P_{(0,i)}^{m}\in \mathcal{P}^\sharp([0,1])\setminus\{0\}$ arbitrarily, and once again we proceed recursively on $r$ backward from $m$ to $1$, and select (arbitrarily)
		$$\begin{array}{l}\ds P^{r}_{(0,i)}\in \mathcal{P}^\sharp([0,1]) \bigcap \left\{{dP^{m}_{(0,i)}\over dt},{dP^{m-1}_{(0,i)}\over dt},\ldots,{dP^{r+1}_{(0,i)}\over dt}\right\}^{\bigperp}\setminus\Big\{0\Big\},\quad  \forall r\in \left\{1,\ldots,m-1\right\}\setminus\left\{i\right\},\\\ds P^{i}_{(0,i)}\in \left\{p\in\mathcal{P}([0,1]):\ \ \ \ds\int_0^1{dp\over dt}(t)dt=0, \ \int_0^1t{dp\over dt}(t)dt\neq0 \right\} \bigcap \left\{{dP^{m}_{(0,i)}\over dt},{dP^{m-1}_{(0,i)}\over dt},\ldots,{{dP^{1}}_{(0,i)}\over dt}\right\}^{\bigperp}.\end{array}  $$

		\item Furthermore, let us consider  the polynomial 	$  \tilde P(t):=5t^4-10t^3+6t^2-t                   $, and observe that $\tilde P\in \mathcal{P}^\sharp([0,1])$.


	\end{itemize}
\end{definition}

\begin{definition}[Families of control variations]\label{defvar}
	Let us fix a map  $u\in L^\infty([0,T],\rr^m) $ and a time  $\ol{t} \in (0,{T})$, and
	let the polynomials $P^{r}_{(j,i)}, \tilde P$,  ($r,i=1,\ldots,m$, $j=0,\ldots,m,$  $j<i$), be  as in Definition \ref{iP}. To every    $\mathbf{c}=\hat u\in \mathfrak{V}_{ndl}(=U)$  we associate  the family  of $\varepsilon$-dependent  controls ${\big\{ {{ u}}_{,\varepsilon,\mathbf{c},\ol{t}}(t): \, \varepsilon \in [0,\ol{t})\big\}}$ defined as follows.
	\begin{enumerate}
		\item If  $\mathbf{c}=\hat u\in \mathfrak{V}_{ndl}(=U)$  we set  \bel{perneedle}{{u}}_{\varepsilon,\mathbf{c},\ol{t}}(t):=\begin{cases}{{u}}(t) &\text{ if } t \in [0,\ol{t}-\varepsilon) \cup (\ol{t},T]\\  \hat u &\text{ if } t \in [\ol{t}-\varepsilon,\ol{t}],\end{cases}\eeq 
		This  family  is usually referred as a {\rm needle variation} of $u$. .
		
		\item 
		If
		$\mathbf{c}=(j,i)\in  \mathfrak{V}_{Goh}\cup \mathfrak{V}_{LC2} $ we associate the $\varepsilon$-parameterized family   ${\Big\{ {{u}}_{\varepsilon,\mathbf{c},\ol{t}}(t): 0<\sqrt{\varepsilon}\le\ol{t}\Big\}}$ of controls defined as
		\bel{perGoh}{u}_{\varepsilon, \mathbf{c},\ol{t}}(t):=\begin{cases}
			{{u}}(t) & \text{ if } t \not \in [\ol{t}-\sqrt{\varepsilon},\ol{t}]\\\ds
			{{u}}(t)+ \sum\limits_{r=1}^m {dP^{r}_{(j,i)}\over dt}\left(\frac{t-\ol{t}}{\sqrt{\varepsilon}}+1\right)\mathbf{e}_r& \text{ if }t \in [\ol{t}-\sqrt{\varepsilon},\ol{t}] .
		\end{cases}\eeq
		We call  this family {\rm Goh variation  of ${{u}}$ at $\bar t$} in the case when $i>0,j >0 $, while we call it   {\rm  Legendre\textendash{}Clebsch variation of step 2 {\rm(}of ${{u}}$ at $\bar t${\rm)} }  as soon as $j > 0$ $i=0$.
		\item 
		If $m=1$, $g:=g_1$, 
		$\{\mathbf{c}\} = \mathfrak{V}_{LC3}$,
		we call   {\rm  Legendre\textendash{}Clebsch variation of step $3$ (of ${{u}}$ at $\bar t$)} the family  ${\Big\{ {{u}}_{\varepsilon,\mathbf{c},\ol{t}}(t): \, 0<\sqrt[3]{\varepsilon}\le\ol{t}\Big\}}$ of controls defined as
		\bel{perLC}{{u}}_{\varepsilon, \mathbf{c},\ol{t}}(t):=\begin{cases}
			{{u}}(t) & \text{ if } t \not \in [\ol{t}-\sqrt[3]{\varepsilon},\ol{t}]\\\ds
			{{u}}(t)+ {d\tilde P\over dt}\left(\frac{t-\ol{t}}{\sqrt[3]{\varepsilon}}+1\right)& \text{ if }t \in [\ol{t}-\sqrt[3]{\varepsilon},\ol{t}]\\
		\end{cases}\eeq
		
\end{enumerate}\end{definition}



Theorem \ref{eccolestime}  below is a classical result  concerning  smooth control systems.   Though it can  be recovered  in classical literature, possibly stated in a slightly  different form (see e.g. \cite{Schattler}, Subsect.2.8).

\begin{theorem} \label{eccolestime} Let  $(u,x)$ be a process, and 
	let the control $u$ be  singular.
	Choose   $\ol{t}, \varepsilon$ such that  the vector fields  $f$ and $g_i$  are  of class $C^1$ around $x(\ol{t})$ and, moreover,  
	$0<\sqrt[3]{\varepsilon}<\min\{\ol{t},1\}$. 
	Then there exist $M,>0$ such that, for any   $\mathbf{c} \in \mathfrak{V} $, using $x_\varepsilon$ to denote the trajectory corresponding to the perturbed control  ${u}_{\varepsilon,\mathbf{c},\ol{t}}$,\footnote{Of course $x_\varepsilon$ depends on $ \mathbf{c}$, and $\ol{t}$, as well. } we get the following three statements.
	
	\begin{enumerate} 	\item If  $\mathbf{c}=(0,i) \in \mathfrak{V}_{LC2}$, for some $i=1,\ldots,m$,
		then $$x_\varepsilon(\ol{t})-x(\ol{t})=M\varepsilon[f,g_i](x(\ol{t}))+o(\varepsilon)\,\,$$
		\item If $\mathbf{c}=(j,i) \in \mathfrak{V}_{Goh}$  (i.e. $j=1,\ldots,m$, $1<i\leq m$) 
		then $$x_\varepsilon(\ol{t})-x(\ol{t})=M\varepsilon[g_j,g_i](x(\ol{t}))+o(\varepsilon)\,\,$$
		\item If $m=1$, $\mathbf{c} = \mathfrak{V}_{LC3}$ and, moreover, $f,g (:=g_1)$ are of class $C^2$ around ${x(\ol{t})}$,
		then $$x_\varepsilon(\ol{t})-x(\ol{t})=M\varepsilon[g,[f,g]](x(\ol{t}))+o(\varepsilon)\,\,$$
	\end{enumerate} 
	
\end{theorem}

\vskip5truemm
\subsection{Proof of Theorem \ref{eccolestime}}
{
	
	\subsubsection{\sc Proof of {\rm 1)} and {\rm 2)} in Theorem \ref{eccolestime} }\,
	\noindent
	Let us  fix $j \in \{0,\ldots,m-1\}$ and $i \in \{1,\ldots,m\}$, with $j<i$, and 
	let us begin with  observing that 
	\bel{var1}x_\varepsilon(\ol{t})-x(\ol{t}) = \hat x(2\sqrt{\varepsilon}) -x(\ol{t}),\eeq
	where $\hat x$ is  the solution on $[0,2\sqrt{\varepsilon}]$ of the Cauchy problem
	$$\left\{\begin{array}{l}\hat{x}(s)= \hat{a}^0(s)\cdot\hat{x}(s) f + \sum\limits_{{r=1}}^m \hat{a}^r(s)\cdot\hat{x}(s){g}_r\\
		\hat x(0) = x(\bar t)\end{array}\right. \footnote{In standard notation this system reads
		$$\hat{x}(s)= \hat{a}^0(s)\cdot f(\hat{x}(s)) + \sum\limits_{{r=1}}^m \hat{a}^r(s)(s)\cdot{g}_r(\hat{x}(s))
		$$
	}$$ where
	
	$$ (\hat a^{0},\hat a^{1},\ldots \hat a^{m})(s):=\begin{cases}
		-\left(1, u(\ol{t}-s)\right)  & \text{ if } s \in [0,\sqrt{\varepsilon}]\\\displaystyle
		\,\,\,\left(1, u(\ol{t}-2\sqrt{\varepsilon}+s)\right)+\left(0 ,\sum\limits_{r=1}^m \frac{dP^r_{(j,i)}}{dt}\left(\frac{ s-\sqrt{\varepsilon}}{{\sqrt{\varepsilon}}}\right)\mathbf{e}_r\right) &\text{ if } s \in [\sqrt{\varepsilon},2\sqrt{\varepsilon}]
	\end{cases}$$

	%
	\noindent
	Indeed, notice that  one has $\hat x(\sqrt{\varepsilon}) = x(\bar t -\sqrt{\varepsilon})$, while  in $[\bar t -\sqrt{\varepsilon}, \bar t]$  we are implementing the corresponding  control ${\bf u}_{{\varepsilon},{\bf c} ,\bar t}$ in \eqref{perGoh}, so obtaining  \eqref{var1}.  
	
	Indeed, notice that  one has $\hat x(\sqrt{\varepsilon}) = x(\bar t -\sqrt{\varepsilon})$, while  in $[\bar t -\sqrt{\varepsilon}, \bar t]$  we are implementing the corresponding  control ${\bf u}_{{\varepsilon},{\bf c} ,\bar t}$ in \eqref{perGoh}, so obtaining  \eqref{var1}.  
	Let us set,
	%
	{
		$$
		\ds\hat A^{h}(s) :=\int_0^s \hat a^{h}(\sigma)d\sigma , \quad \hat A^{h,k}(s) :=\int_0^s \hat A^h(\sigma)\hat a^{k}(\sigma)d\sigma\quad \forall h,k =0,\ldots m ,
		$$
		
	}
	\noindent
	Let us  check that $\hat{A}^0(2\sqrt{\varepsilon})=0$  and that  the choice  of  the map $\frac{dP}{ds}_{(i,j)}$  yields  $\hat{A}^h(2\sqrt{\varepsilon})=0$ $\forall h\in \{1,\dots,m\}$, $\hat{A}^{h,k}(2\sqrt{\varepsilon})=0$ $\forall (h,k)\in \{1,\dots,m\}^2\backslash \{(i,j)\}$,  and $\hat{A}^{i,j}(2\sqrt{\varepsilon})\neq 0$
	(This will allow us to  apply  Corollary \ref{corollariodischattlerutile} and get the desired conclusion).
	%
	For all $s\in[0,2\sqrt{\varepsilon}]$,
	one has
	$$\hat{A}^0(s)=-s{\bf 1}_ {[0,\sqrt{\varepsilon}]}+(s-2\sqrt{\varepsilon}){\bf 1}_{[\sqrt{\varepsilon},2\sqrt{\varepsilon}]},$$ and, for all $h\in\{1,\ldots,m\}$, 
	{	$$\begin{array}{c}	\hat{A}^{h}(s) = \displaystyle \bigintss_0^s  \left[-\mathbf{1}_{[0,1\sqrt{\varepsilon}]}(\sigma)\cdot u^h(\bar t-\sigma)+ \mathbf{1}_{[\sqrt{\varepsilon},2\sqrt{\varepsilon}]}(\sigma)\cdot \left(u^h(\bar t+\sigma - 2\sqrt{\varepsilon} )+ \frac{dP}{ds}_{(j,i)}\left(\frac{ \sigma-\sqrt{\varepsilon}}{{\sqrt{\varepsilon}}}\right)\right)\right]\,d\sigma\end{array} $$}
	Hence,
	\bel{A0}\hat{A}^0(2\sqrt{\varepsilon})=0,
	\eeq and, by  $P_{(i,j)}(1) - P_{(i,j)}(0) =0$, one also has \bel{Ah}\begin{array}{c}	\hat{A}^{h}(2\sqrt{\varepsilon})= \displaystyle \bigintsss_0^{2\sqrt{\varepsilon}} \left[-\mathbf{1}_{[0,1\sqrt{\varepsilon}]}(\sigma)\cdot u^h(\bar t-\sigma)+ \mathbf{1}_{[\sqrt{\varepsilon},2\sqrt{\varepsilon}]}(\sigma)\cdot\left(u^h(\bar t+\sigma - 2\sqrt{\varepsilon} )\right)\right]d\sigma   + \\\\\quad\quad\quad\quad\quad \quad\quad\quad\quad\quad\quad\quad\quad\quad\quad\quad\quad\quad\quad\quad\quad\quad\quad\quad\quad\quad\quad\,\,\,\,\,\,\,\,\,\,\,\,\,\,+\sqrt{\varepsilon} \Big(\U^h_{(i,j)}(1) - \U^h_{(i,j)}(0)\Big) = \\\\
		=	\displaystyle -\bigintsss_0^{\sqrt{\varepsilon}} u^h(\bar t-\sigma)d\sigma +	\displaystyle \bigintsss_0^{\sqrt{\varepsilon}}u^h(\bar t-\sigma)d\sigma 
		= 0.\end{array}\eeq
	In other words , for any $h,k=1,\ldots,m$ the time-space curves $({\hat A}^0,{\hat A}^h)$ and the space curve $({\hat A}^h,{\hat A}^k)$ are closed, in that $({\hat A}^0,{\hat A}^k)(0) = ({\hat A}^0,{\hat A}^k)(2\sqrt{\varepsilon})=0$ and $({\hat A}^h,{\hat A}^k)(0) = ({\hat A}^h,{\hat A}^k)(2\sqrt{\varepsilon})=0$.		
	
	In order to compute the coefficients $\hat{A}^{h,k}(2\sqrt{\varepsilon})$ when $h,k=1,\ldots,m$, let us set, for every $h=1,\ldots,m$ and every $s\in [0,2\sqrt{\varepsilon}]	$,
	{ $$\check{a}^h_u(s):=
		-u^h(\ol{t}-s)\cdot{\bf 1}_{[0,\sqrt{\varepsilon}]} + 
		u^h(\ol{t}-2\sqrt{\varepsilon}+s)\cdot {\bf 1}_{[\sqrt{\varepsilon},2\sqrt{\varepsilon}]},$$  $$\check{A}^h_u(s) := \ds\int_0^s \check a^h_u(\sigma)d\sigma =\int_0^s\Big(-u^h(\ol{t}-\sigma)\cdot{\bf 1}_{[0,\sqrt{\varepsilon}]} + 
		u^h(\ol{t}-2\sqrt{\varepsilon}+\sigma)\cdot {\bf 1}_{[\sqrt{\varepsilon},2\sqrt{\varepsilon}]} \Big) d\sigma  $$}
	and	
	{\small$$\begin{array}{c}\displaystyle  \check {a}^{h}(s):=
			\sum\limits_{r=1}^m \frac{dP^h_{(i,j)}}{dt}\left(\frac{ s-\sqrt{\varepsilon}}{{\sqrt{\varepsilon}}}\right)\cdot {\bf 1}_{[\sqrt{\varepsilon},2\sqrt{\varepsilon}] },
			\displaystyle \\ \check A^{h}(s) :=\ds\int_0^s\check a^{h}(\sigma)d\sigma  =\left\{\begin{array}{ll} 0  &\forall s\in [0,\sqrt{\varepsilon}]\\  \sqrt{\varepsilon}\left(P_{(i,j)}^h\left(\frac{ s-\sqrt{\varepsilon}}{{\sqrt{\varepsilon}}}\right)- P^h_{(i,j)}(0)\right) \,\,\,&\forall s\in [\sqrt{\varepsilon},2\sqrt{\varepsilon}],\end{array}\right.\end{array}.$$}
	Hence 
	$$
	\quad \hat {a}^{h}(s) = \check{a}^h_u(s) + \check{a}^{h}(s)\qquad 	\hat A^{h}(s) = \check{A}^h_u(s) +\check A^{h}(s)\qquad \forall s\in[0,2\sqrt{\varepsilon}]
	$$
	Observing that   $\ds \int_0^{2\sqrt{\varepsilon}}\check{A}^h_u(s)\check a_{u}^k(s) ds =0$, we get 
	\bel{variation21}
	\hat{A}^{h,k}(2\sqrt{\varepsilon})=
	\ds \int_0^{2\sqrt{\varepsilon}}{\check A^{h}}(s)\check a_{u}^{k}(s)ds+\int_0^{2\sqrt{\varepsilon}}\check{A}^h_u(s)\check a^{k}(s)ds +\int_0^{2\sqrt{\varepsilon}}\check A^{h}(s)\check a^{k}(s)ds \eeq 	
	Since the curves $(\hat{A}^0,\hat{A}^k)$, $(\hat{A}^h,\hat{A}^k)$, and $(\check{A}^h_u,\check{A}^k_u)$ are closed, we can give the following interpretation to the some of above coefficients: $$\hat{A}^{0,k}(2\sqrt{\varepsilon}) = Area(\hat{A}^0,\hat{A}^k),\quad \hat{A}^{h,k}(2\sqrt{\varepsilon}) = Area(\hat{A}^h,\hat{A}^k),$$
	and $$ 
	\check{A}_u^{h,k}(2\sqrt{\varepsilon}):=\int_0^{2\sqrt{\varepsilon}}\check{A}^h_u(s)\check a_{u}^k(s) ds =Area(\check{A}^h_u, \check{A}^k_u) =0	$$		
	Let us compute the three terms on the right-hand side of \eqref{variation21}.		 Since $2\sqrt{\varepsilon}$ is a Lebesgue point  for the map 	$$[\sqrt{\varepsilon},2\sqrt{\varepsilon}]\ni s\to \check A^{h}(s)\check a^k_u(s) = 
	u^k\left(\bar t  -2\sqrt{\varepsilon}+s \right)\cdot\left(\sqrt{\varepsilon}\int_{0}^{\frac{s-\sqrt{\varepsilon}}{\sqrt{\varepsilon}}}  \frac{dP}{dt}^h_{(i,j)}\left(\sigma\right)\,d\sigma \right),  $$
	one gets 	$$\begin{array}{c}
		\ds \int_0^{2\sqrt{\varepsilon}}\check A^{h}(s)\check a^k_u(s)ds  =   \int_{\sqrt{\varepsilon}}^{2\sqrt{\varepsilon}}\check A^{h}(s)\check a^k_u(s)ds = 
		\ds\sqrt{\varepsilon}\sqrt{\varepsilon}\left( u^k\left(\bar t\right)\cdot \left(\int_{0}^{1}  \frac{dP^h_{(j,i)}}{d\sigma}\left(\sigma\right)\,d\sigma\right) \right) ds + o({\varepsilon})= \\   
		\ds {\varepsilon} u^k\left(\bar t  \right) \left(P^h_{(j,i)}(1)- P^h_{(j,i)}(0)\right)\,ds + o({\varepsilon}){ = o({\varepsilon})}\qquad \forall h,k\in\{1,\ldots,m\}
	\end{array}$$
	
	Moreover, since $\check{A}^h(2\sqrt{\varepsilon})=\check{A}^h(\sqrt{\varepsilon})=0$ for all $h,k\in\{1,\ldots,m\}$, the above estimate implies 

	{	$$\begin{array}{c}
			\ds \int_0^{2\sqrt{\varepsilon}}\check{A}^h_u(s)\cdot \check a^{k}(s)ds  = \ds \int_{\sqrt{\varepsilon}}^{2\sqrt{\varepsilon}}\check{A}^h_u(s)\cdot \check a^{k}(s)ds =\\\ds 
			\left.\check{A}^h_u(s)\cdot \check A^{k}(s)\right|_{\sqrt{\varepsilon}}^{2\sqrt{\varepsilon}} -
			\int_{\sqrt{\varepsilon}}^{2\sqrt{\varepsilon}}\check{a}^h_u(s)\cdot \check{A}^{k}(s)ds= 
			\Big(\check{A}^h_u(2\sqrt{\varepsilon})\cdot \check A^{k}(2\sqrt{\varepsilon})- \check{A}^h_u(\sqrt{\varepsilon})\cdot \check A^{k}(\sqrt{\varepsilon})\Big) + o({\varepsilon}) =  o({\varepsilon})
			
		\end{array}$$}
	
	Finally, for any $h,k=1,\ldots,m$, \bel{Ahk}\begin{array}{l} 
		\ds	\hat{A}^{h,k}(2\sqrt{\varepsilon})
		= \int_0^{2\sqrt{\varepsilon}}\check A^{h}(s)\check a^{k}(s)ds =	\int_{\sqrt{\varepsilon}}^{2\sqrt{\varepsilon}}\left(\int_{\sqrt{\varepsilon}}^{s}   \ \frac{dP^h_{(i,j)}}{dt}\left(\frac{\sigma-\sqrt{\varepsilon}}{\sqrt{\varepsilon}}\right)d\sigma\right)  \frac{dP^k_{(i,j)}}{dt}\left(\frac{s-\sqrt{\varepsilon}}{\sqrt{\varepsilon}}\right)ds= \\\\
		=\ds {\varepsilon}	\int_{0}^{1} \left(\int_{0}^{s}    \frac{dP^h_{(i,j)}}{d\sigma}\left(\sigma\right)d\sigma\right) \frac{dP^k_{(i,j)}}{ds}\left(s\right)ds  = {\varepsilon}	\int_{0}^{1} \left(P^h_{(i,j)}(s) -P^h_{(i,j)}(0) \right) \frac{dP}{dt}^k_{(i,j)}(s)ds =\\\\
		\ds  {\varepsilon}	\int_{0}^{1} P^h_{(i,j)}(s) \frac{dP}{dt}^k_{(i,j)}(s) ds	 = \left\{\begin{array}{ll}\varepsilon Area\left(P^h_{(i,j)},P^k_{(i,j)}\right) = 0 &\text{if} \,\,\,\, (h,k)\neq (i,j) \,\,\,\, \text{or} (h,k)\neq (j,i)\,\,\\
			{\varepsilon} Area\left(P^h_{(i,j)},P^k_{(i,j)}\right) \quad &\text{if} \,\,\,\,  (h,k) = (i,j)
			\\	- {\varepsilon} Area\left(P^h_{(i,j)},P^k_{(i,j)}\right) &\text{if} \,\,\,\,  (h,k) = (j,i)\\\end{array}\right.
	\end{array}\eeq
	
	%
	Moreover, if $k>0$, one has
	{\small
		\bel{A0k}	\begin{array}{c} \hat{A}^{0,k}(2\sqrt{\varepsilon})  =\qquad\qquad\qquad\qquad\qquad\qquad\qquad\qquad\qquad\qquad\qquad\qquad\qquad\qquad\\\\ \ds\int_0^{\sqrt{\varepsilon}} \sigma \cdot u^i(\bar t-\sigma)\, d\sigma +
			\int_{\sqrt{\varepsilon}}^{2\sqrt{\varepsilon}} (\sigma-2\sqrt{\varepsilon}) \cdot  u^i(\bar t -2\sqrt{\varepsilon}+\sigma )\, d\sigma   +\ds\int_{\sqrt{\varepsilon}}^{2\sqrt{\varepsilon}}\frac{dP}{dt}^k_{(i,j)}\left(\frac{ \sigma-\sqrt{\varepsilon}}{{\sqrt{\varepsilon}}}\right) \cdot (\sigma-2\sqrt{\varepsilon})\,d\sigma=\\\\	\ds\int_{\sqrt{\varepsilon}}^{2\sqrt{\varepsilon}}\frac{dP}{dt}^k_{(i,j)}\left(\frac{ \sigma-\sqrt{\varepsilon}}{{\sqrt{\varepsilon}}}\right) \cdot (\sigma-2\sqrt{\varepsilon})\,d\sigma
			= \\
			\ds	\varepsilon\int_0^1 \frac{dP^k_{(i,j)}}{ds}(s)\cdot (s-1)\,ds = P^k_{(i,j)}(s)\cdot (s-1)\Big|_0^1 	-\varepsilon\int_0^1 P^k_{(i,j)}(s)ds  = 0 \end{array}
		\eeq}
	
	By applying Corollary \ref{corollariodischattlerutile}, we then get 
	$$
	x_\varepsilon(\ol{t})-x(\ol{t}) = \hat x(2\sqrt{\varepsilon}) -x(\ol{t}) = \varepsilon M \, [g_i,g_j](x(\ol{t})) + o(\varepsilon),
	$$
	where $\ds M:=	Area\Big(P_{(i,j)}^i,P_{(i,j)}^j\Big) =\frac{1}{2r^2}$, so \eqref{2} is proved.

	The proof of points (1) and (2) of the thesis of Theorem \ref{eccolestime} is thus concluded.\\
	\\\\
	
	\subsubsection{\sc Proof of {\rm 3)} in Theorem \ref{eccolestime} }
	Let us consider the control $\mathbf{{u}}_{\varepsilon, \mathbf{c},\ol{t}}$ as in \eqref{perLC}, with the   polynomial $P$ defined as
	$$P(t):=5t^4-10t^3+6t^2-t.$$
	This choice, that one can find e.g. in \cite{Schattler}, 
	comes from the application to estimate \eqref{Cruciale} of arguments akin to those utilized to prove points (2) (3) , through  several but trivial (linear) computations to establish the coefficients. \\
	\noindent 
	Set $\hat{g}_0=f$, $\hat{g}_1=g =g_1$ and  observe that , as before, 
	\bel{var11}x_\varepsilon(\ol{t})-x(\ol{t}) = \hat x(2\sqrt[3]{\varepsilon}),\eeq
	where $\hat x$ is  the solution of the Cauchy problem
	$$\left\{\begin{array}{l}\hat{x}(t)= \hat{a}^0(s)\cdot \hat{x}(s)\hat{g}_0 + \hat{a}^1(s)\cdot \hat{x}(s)\hat{g}_1\\
		\hat x(0) = x(\bar t)\end{array}\right., $$ 
	and
	$$ (\hat a^0,\hat a^1)(s):=\begin{cases}
		-\left(1, u(\ol{t}-s)\right)  & \text{ if } s \in [0,{\sqrt[3]{\varepsilon}}]\\
		\,\,\,\left(1, u(\ol{t}-2\sqrt[3]{\varepsilon}+s)\right)+\left(0 , \ds\frac{dP}{dt}\left(\frac{s-\sqrt[3]{\varepsilon}}{\sqrt[3]{\varepsilon}}\right)\right) &\text{ if } s \in [\sqrt[3]{\varepsilon},2\sqrt[3]{\varepsilon}]
	\end{cases}$$

	Once again,  we will finally  obtain that $x_\varepsilon(\ol{t}) = \tilde x(2\sqrt{\varepsilon})$. \\
	Let us set
	$$
	\ds\tilde A^{h}(s) :=\int_0^s \tilde a^{h}(\sigma)d\sigma , \quad \tilde A^{h,k}(s) :=\int_0^s \tilde A^h(\sigma)\tilde a^{k}(\sigma)d\sigma\quad \forall h,k =0,\ldots m .
	$$
	As in the previous case we have
	$$ \tilde A^{h}(2\sqrt{\varepsilon}) =0  \quad   \forall h=0,\ldots m $$
	Moreover, we get
	$$	\begin{array}{c} \tilde{A}^{0,i}(2\sqrt{\varepsilon})  =\ds\int_0^{\sqrt{\varepsilon}} \sigma \cdot u^i(\bar t-\sigma)\, d\sigma +
		\int_{\sqrt{\varepsilon}}^{2\sqrt{\varepsilon}} (\sigma-2\sqrt{\varepsilon}) \cdot  u^i(\bar t -2\sqrt{\varepsilon}+\sigma )\, d\sigma + \\\\\quad\quad\quad\quad\quad\quad\quad\quad\quad\quad\quad\quad\quad\quad\quad\quad\quad\quad\quad\quad\quad +\ds\int_{\sqrt{\varepsilon}}^{2\sqrt{\varepsilon}}\frac{dP^i_{(0,i)}}{d\sigma}\left(\frac{ \sigma-\sqrt{\varepsilon}}{{\sqrt{\varepsilon}}}\right) \cdot (\sigma-2\sqrt{\varepsilon})\,d\sigma=\\\\	\ds\int_{\sqrt{\varepsilon}}^{2\sqrt{\varepsilon}}\frac{dP^i_{(0,i)}}{ds}\left(\frac{ \sigma-\sqrt{\varepsilon}}{{\sqrt{\varepsilon}}}\right) \cdot (\sigma-2\sqrt{\varepsilon})\,d\sigma
		= 
		\varepsilon\int_0^1 \frac{dP^i_{(0,i)}}{ds}(s)\cdot (s-1)\,ds = 	-\varepsilon\int_0^1 {P^i_{(0,i)}}(s)ds  = \tilde{M}\varepsilon ,  \end{array}
	$$ while, for every $ k\neq i$, one has 
	$$	\tilde{A}^{0,k}(2\sqrt{\varepsilon})  =\ds\int_0^{\sqrt{\varepsilon}} \sigma \cdot u^i(\bar t-\sigma)\, d\sigma +
	\int_{\sqrt{\varepsilon}}^{2\sqrt{\varepsilon}} (\sigma-2\sqrt{\varepsilon}) \cdot  u^i(\bar t -2\sqrt{\varepsilon}+\sigma )\, d\sigma = 0 .
	$$ }
By applying Corollary \ref{corollariodischattlerutile}, we now get 
$$
x_\varepsilon(\ol{t})-x(\ol{t}) = \tilde x(2\sqrt{\varepsilon}) -x(\ol{t}) = \varepsilon \tilde M \, [g_0,g_i](x(\ol{t})),
$$
so \eqref{1} is proved.

\subsubsection{\sc Proof of  {\rm 4)} in Theorem \ref{eccolestime}}  Once again we will expolit the equality  \bel{coincidono4}x_\varepsilon(\ol{t}) = \ol{x}(2\sqrt{\varepsilon})
\eeq	where $\ol{x}$ is  the solution on $[0,2\sqrt{\varepsilon}]$ of the Cauchy problem
\bel{var2}\left\{\begin{array}{l}\ol{x}(s)= \ol{a}^0(s)\cdot f(\ol{x}(s)) +  \ol{a}^1(s)\cdot g_1(\ol{x}(s)),\\
	\ol{x}(0) = x(\bar t) ,\end{array}\right.\eeq  the controls $(\ol{a}^0,\ol{a}^1)$ being now  defined by  $\ol{v}$
$$ (\ol{a}^{0},\ol{a}^{1})(s):=\begin{cases}
	-\left(1, u(\ol{t}-s)\right)  & \text{ if } s \in [0,\sqrt{\varepsilon}]\\\displaystyle
	\,\,\,\left(1, u(\ol{t}-2\sqrt{\varepsilon}+s)\right)+ \frac{dP}{ds}\left(\frac{ s-\sqrt{\varepsilon}}{{\sqrt{\varepsilon}}}\right) &\text{ if } s \in [\sqrt{\varepsilon},2\sqrt{\varepsilon}]
\end{cases}$$

Let us set
$$
\ds\ol{A}^{h}(s) :=\int_0^s\ol{a}^{h}(\sigma)d\sigma , \quad \quad \forall h =0,1\qquad \ol{A}^{0,1}(s) :=\int_0^s \ol{A}^0(\sigma) \ol{a}^{1}(\sigma)d\sigma
$$
$$\ol{A}^{101}(s) := \ds\int_0^s \ol{A}^1(\sigma)\frac{d\ol{A}^{0,1}}{d\sigma}(\sigma)d\sigma  = \int_0^s \ol{A}^1(\sigma) \ol{A}^0(\sigma)a^1(\sigma)d\sigma$$
$$\ol{A}^{001}(s) := \ds\int_0^s \ol{A}^0(\sigma)\frac{d\ol{A}^{0,1}}{d\sigma}(\sigma)d\sigma  = \int_0^s \ol{A}^0(\sigma) \ol{A}^0(\sigma)a^1(\sigma)d\sigma$$
As in the previous case one has
$$
\ol{A}^{0}(2\sqrt{\varepsilon}) = 0
$$
$$
\begin{array}{l} \ds \ol{A}^{1}(2\sqrt[3]{\varepsilon}):=\int_0^{2\sqrt[3]{\varepsilon}} \ol{a}^{1}(s)ds = - \int_0^{\sqrt[3]{\varepsilon}} u(\ol{t}-s)ds +  \int_{\sqrt[3]{\varepsilon}}^{2\sqrt[3]{\varepsilon}} u(\ol{t}-2\sqrt[3]{\varepsilon}+s) ds +\int_{\sqrt[3]{\varepsilon}}^{2\sqrt[3]{\varepsilon}}\frac{dP}{ds} \Big( \frac{s-\sqrt[3]{\varepsilon}}{\sqrt[3]{\varepsilon}}  \Big) ds =\\\ds\qquad\qquad= { \sqrt[3]{\varepsilon}} \int_{0}^{1} \frac{dP}{ds} ( {\sigma}) d\sigma  = { \sqrt[3]{\varepsilon}} \Big(P(1) -P(0)\Big)=0  \end{array}$$
Moreover, 

$$\begin{array}{l} \ds\ol{A}^{0,1}(2\sqrt[3]{\varepsilon}) :=\int^{2\sqrt[3]{\varepsilon}}_{0} \ol{A}^0(s) \ol{a}^{1}(s) ds=\\

	\\ =\ds \int_{0}^{\sqrt[3]{\varepsilon}} s u(\ol{t}-s) ds  +  \ds\int_{\sqrt[3]{\varepsilon}}^{2\sqrt[3]{\varepsilon}} (s-\sqrt[3]{\varepsilon}) u(\ol{t}-2\sqrt[3]{\varepsilon}+s)ds +  \int_{\sqrt[3]{\varepsilon}}^{2\sqrt[3]{\varepsilon}} (s-\sqrt[3]{\varepsilon}) \frac{dP}{ds}\left( \frac{s-\sqrt[3]{\varepsilon}}{\sqrt[3]{\varepsilon}}  \right) ds=\\ =\ds  \sqrt[3]{\varepsilon^2} \int_{0}^{1} \sigma \frac{dP}{ds}(\sigma) ds =\sqrt[3]{\varepsilon^2}\left( 1\cdot P(1) -0\cdot P(0) - \int_0^1P(s)ds \right)= 0 \end{array}
$$
and
$$
\begin{array}{l} \ds\ol{A}^{0,0,1}(2\sqrt[3]{\varepsilon}) :=\int^{2\sqrt[3]{\varepsilon}}_{0} \left(\ol{A}^0(s)\right)^2 \ol{a}^{1}(s) ds =\\ \ds -\int_{0}^{\sqrt[3]{\varepsilon}} s^2 u(\ol{t}-s) ds  +  \ds\int_{\sqrt[3]{\varepsilon}}^{2\sqrt[3]{\varepsilon}} (s-\sqrt[3]{\varepsilon})^2 u(\ol{t}-2\sqrt[3]{\varepsilon}+s)ds +  \int_{\sqrt[3]{\varepsilon}}^{2\sqrt[3]{\varepsilon}} (s-\sqrt[3]{\varepsilon})^2\frac{dP}{ds} \left( \frac{s-\sqrt[3]{\varepsilon}}{\sqrt[3]{\varepsilon}}  \right) ds=\\ =\ds  \sqrt[3]{\varepsilon^2} \int_{0}^{1} \sigma^2\frac{dP}{d\sigma}(\sigma) ds =\sqrt[3]{\varepsilon^2}\left( 1\cdot P(1) -0\cdot P(0) - \int_0^1P(s)ds \right)= 0\end{array}
$$
Finally,
\small
$$
\begin{array}{l}\qquad\qquad\qquad\qquad \ds\ol{A}^{1,0,1}(2\sqrt[3]{\varepsilon}) =
	\ds \int_{0}^{2 \sqrt[3]{\varepsilon}} (\ol{A}^1)^{2} \ol{a}^0 d\sigma = 
	\\
	= - \ds \bigintsss_{0}^{\sqrt[3]{\varepsilon}} \Big(\int_{0}^{\eta} u(t-\sigma)  d\sigma \Big)^2d\eta + \ds \ \bigintsss_{\sqrt[3]{\varepsilon}}^{2\sqrt[3]{\varepsilon}} \Bigg( -  \int_{\eta}^{2\sqrt[3]{\varepsilon}} u(t-\sigma-2\sqrt[3]{\varepsilon}) d\sigma + \int_{\sqrt[3]{\varepsilon}}^{\eta} \frac{dP}{d\sigma}\Big(\frac{\sigma-\sqrt[3]{\varepsilon}}{\sqrt[3]{\varepsilon}}\Big) d\sigma \Bigg)^2 d\eta =  \\
	=- \ds \ \bigintsss_{0}^{\sqrt[3]{\varepsilon}} \Big(\int_{0}^{\eta} u(t-\sigma)  d\sigma \Big)^2d\eta +
	\int_{\sqrt[3]{\varepsilon}}^{2\sqrt[3]{\varepsilon}} \Bigg(\int_{\eta}^{2\sqrt[3]{\varepsilon}} u(t-\sigma-2\sqrt[3]{\varepsilon})\Bigg)^2 d\sigma +
	\\
	\ds +\ \bigintsss_{\sqrt[3]{\varepsilon}}^{2\sqrt[3]{\varepsilon}} \Bigg(   - 2\int_{\eta}^{2\sqrt[3]{\varepsilon}} u(t-\sigma-2\sqrt[3]{\varepsilon}) d\sigma \int_{\sqrt[3]{\varepsilon}}^{\eta} \frac{dP}{d\sigma}\left(\frac{\sigma-\sqrt[3]{\varepsilon}}{\sqrt[3]{\varepsilon}}\right) d\sigma + \Big( \int_{\sqrt[3]{\varepsilon}}^{\eta} \frac{dP}{d\sigma}\left(\frac{\sigma-\sqrt[3]{\varepsilon}}{\sqrt[3]{\varepsilon}}\right) d\sigma \Big) ^2 \Bigg) d\eta = \\=  \ds \ \bigintsss_{\sqrt[3]{\varepsilon}}^{2\sqrt[3]{\varepsilon}} \Bigg(   - 2\int_{\eta}^{2\sqrt[3]{\varepsilon}} u(t-\sigma-2\sqrt[3]{\varepsilon}) d\sigma \int_{\sqrt[3]{\varepsilon}}^{\eta} \frac{dP}{d\sigma}\Big(\frac{\sigma-\sqrt[3]{\varepsilon}}{\sqrt[3]{\varepsilon}}\Big) d\sigma + \Big( \int_{\sqrt[3]{\varepsilon}}^{\eta} \frac{dP}{d\sigma}\Big(\frac{\sigma-\sqrt[3]{\varepsilon}}{\sqrt[3]{\varepsilon}}\Big) d\sigma \Big) ^2 \Bigg) d\eta = \\
	= \ds 2  \bigintsss_{\sqrt[3]{\varepsilon}}^{2\sqrt[3]{\varepsilon}}  \left(   - \fint_{\eta}^{2\sqrt[3]{\varepsilon}} u(t-\sigma-2\sqrt[3]{\varepsilon}) d\sigma \right) \cdot \Bigg(2 \sqrt[3]{\varepsilon} - \eta    \Bigg) \cdot \Bigg(\int_{\sqrt[3]{\varepsilon}}^{\eta} \frac{dP}{d\sigma}\left(\frac{\sigma-\sqrt[3]{\varepsilon}}{\sqrt[3]{\varepsilon}}\right) d\sigma \Bigg) d\eta +\\ \qquad\qquad\qquad+\ds\bigintsss_{\sqrt[3]{\varepsilon}}^{2\sqrt[3]{\varepsilon}} \Bigg( {\sqrt[3]{{\varepsilon}}} \Bigg( P \Big(\frac{\eta-\sqrt[3]{\varepsilon}}{\sqrt[3]{\varepsilon}}\Bigg) - P(0) \Big)\Bigg)^2 d\eta = \\ 
	=\ds   2  \bigintsss_{\sqrt[3]{\varepsilon}}^{2\sqrt[3]{\varepsilon}} \left[\left(-\fint_{\eta}^{2\sqrt[3]{\varepsilon}}  u(t-\sigma-2\sqrt[3]{\varepsilon}) d\sigma\right)\cdot  \Bigg( 2 \sqrt[3]{\varepsilon^2}\int_0^{\frac{\eta-\sqrt[3]{\varepsilon}}{\sqrt[3]{\varepsilon}}}\ \frac{dP}{ds}(s)  ds  - \sqrt[3]{\varepsilon}\eta \int_0^{\frac{\eta-\sqrt[3]{\varepsilon}}{\sqrt[3]{\varepsilon}}}\ \frac{dP}{ds}(s)  ds  \Bigg)\right]d\eta  + \\\qquad\qquad\qquad\ds
	+\varepsilon \bigintsss_{0}^{1} \left(P(s)\right)^2 ds  =\\= \ds    -4 \sqrt[3]{\varepsilon^2} \bigintsss_{\sqrt[3]{\varepsilon}}^{2\sqrt[3]{\varepsilon}} \left[\left(\fint_{\eta}^{2\sqrt[3]{\varepsilon}}  u(t-\sigma-2\sqrt[3]{\varepsilon}) d\sigma\right) \cdot   P\left( \frac{\eta-\sqrt[3]{\varepsilon}}{\sqrt[3]{\varepsilon}}\right)\right] d\eta-   \\ \ds- 2 \sqrt[3]{\varepsilon}  \bigintsss_{\sqrt[3]{\varepsilon}}^{2\sqrt[3]{\varepsilon}} \left[\eta \left( - \fint_{\eta}^{2\sqrt[3]{\varepsilon}}  u(t-\sigma-2\sqrt[3]{\varepsilon}) d\sigma \right) \cdot  P\left( \frac{\eta-\sqrt[3]{\varepsilon}}{\sqrt[3]{\varepsilon}}\right)\right] d\eta  + \varepsilon \bigintsss_{0}^{1} \left(P(s)\right)^2 ds  = \\ =\ds - 4 \varepsilon   \bigintsss_{0}^{1} \left[ \left(\fint_{(s+1)\sqrt[3]{\varepsilon}}^{2\sqrt[3]{\varepsilon}}  u(t-\sigma-2\sqrt[3]{\varepsilon}) d\sigma  \right) \cdot P\Big( s \Big)\right] ds +  \\\ds + 2 \varepsilon   \bigintsss_{0}^{1} (s+1)\left[ \left(\fint_{(s+1)\sqrt[3]{\varepsilon}}^{2\sqrt[3]{\varepsilon}}  u(t-\sigma-2\sqrt[3]{\varepsilon}) d\sigma  \right) \cdot P\Big( s \Big)\right] ds + \varepsilon \bigintsss_{0}^{1} \left(P(s)\right)^2 ds= \end{array}$$ 
$$\begin{array}{l} = \ds- 2    \varepsilon   \bigintsss_{0}^{1} \left[ \left(\fint_{(s+1)\sqrt[3]{\varepsilon}}^{2\sqrt[3]{\varepsilon}}  u(t-\sigma-2\sqrt[3]{\varepsilon}) d\sigma  \right) \cdot P\Big( s \Big)\right] ds +\\ \ds+ 2 \varepsilon  \bigintsss_{0}^{1} \left[\left( \fint_{(s+1)\sqrt[3]{\varepsilon}}^{2\sqrt[3]{\varepsilon}}  u(t-\sigma-2\sqrt[3]{\varepsilon}) d\sigma\right) \cdot s \cdot P\Big(s \Big)  \right]ds 
	+ \varepsilon \bigintsss_{0}^{1} \left(P(s)\right)^2 ds =\end{array}$$ 
\centerline{\footnotesize(using  integration by parts and the equality $P(1)=P(0) =0$)}
$$\begin{array}{l}
	\ds =- 4    \varepsilon   \bigintsss_{0}^{1} \left[ \left(\fint_{(s+1)\sqrt[3]{\varepsilon}}^{2\sqrt[3]{\varepsilon}}  u(t-\sigma-2\sqrt[3]{\varepsilon}) d\sigma  \right) \cdot P\Big( s \Big)\right] ds 
	+ \varepsilon \bigintsss_{0}^{1} \left(P(s)\right)^2 ds 
\end{array} $$
Since $\ds \lim_{\varepsilon\to 0} \frac{1}{\sqrt[3]{\varepsilon}-s\sqrt[3]{\varepsilon}} \int_{(s+1)\sqrt[3]{\varepsilon}}^{2\sqrt[3]{\varepsilon}} \Big( u(t-\sigma-2\sqrt[3]{\varepsilon}) -u(t)\Big) d\sigma  = 0$, then 
$$\ds\int_{(s+1)\sqrt[3]{\varepsilon}}^{2\sqrt[3]{\varepsilon}} \Big( u(t-\sigma-2\sqrt[3]{\varepsilon}) -u(t) d\sigma \Big) = o(\sqrt[3]{\varepsilon})$$
and
$$\fint_{(s+1)\sqrt[3]{\varepsilon}}^{2\sqrt[3]{\varepsilon}}  u(t-\sigma-2\sqrt[3]{\varepsilon}) d\sigma  \Big) = u(t) + o(\sqrt[3]{\varepsilon})$$. \\

So,
\begin{equation}
	\ol{A}^{1,0,1}(2\sqrt[3]{\varepsilon}) = - 4 \varepsilon (u(t) + o(\sqrt[3]{\varepsilon})) \int_{0}^{1} P\Big( s \Big) ds  +  \varepsilon \int_{0}^{1} \left(P(s)\right)^2 ds = \varepsilon   \int_{0}^{1} \left(P(s)\right)^2 ds
\end{equation}

In view of Corollary \ref{corollariodischattlerutile3} we then deduce that 
$$
x_\varepsilon(t) -x(t) = c\varepsilon [g,[f,g]](x(t)) + o(\varepsilon)
$$
$$
x_\varepsilon(t) -x(t) = ct^3 [g,[f,g]](x(t)) + o(t^3)
$$
with

$$\begin{array}{l}
	\ds c:=\int_{0}^{1} \left((s)\right)^2 ds = Area\Big(\left({\tilde id}\right)^2,{\tilde id}\Big)\end{array}$$

Therefore, by  \eqref{Cruciale} we get
$$
{\ol x}_\varepsilon({\ol t}) - {\ol x}({\ol t})=\hat x(2\sqrt[3]{\varepsilon}){- {\ol x}({\ol t})} =c\varepsilon [g,[f,g]]({\ol x}({\ol t})) +o(\varepsilon).$$

\section{Acknowledgements}
The first author is supported by MathInParis project by Fondation Sciences mathématiques de Paris (FSMP), funding from the European Union’s Horizon 2020 research and innovation programme, under the Marie Skłodowska-Curie grant agreement No 101034255. \\
She is also supported by Sorbonne Universite, being affiliated at Laboratoire Jacques Louis Lions (LJLL). \\
Both the author during the writing of this paper were members of the Gruppo Nazionale per
l'Analisi Matematica, la Probabilit\`{a} e le loro Applicazioni (GNAMPA) of the
Istituto Nazionale di Alta Matematica (INdAM). \\
In particular they were members of "INdAM -GNAMPA Project 2023", codice CUP E53C22001930001, "Problems with constrained dynamics: non-smoothness and geometric aspects, impulses and delays"\\
and "INdAM -GNAMPA Project 2024", codice CUP E53C23001670001, "Non-smooth optimal control problems". \\
The second author is still a GNAMPA member. \\
The second author are also members of:
PRIN 2022, Progr-2022238YY5-PE1, "Optimal control problems: analysis, approximations and applications".\\

\end{document}